\documentclass[reqno]{amsart}

\usepackage{graphicx}
\usepackage{amsmath}
\usepackage{amscd}
\usepackage{amssymb}
\usepackage{amstext}
\usepackage{amsmath}
\usepackage[all]{xy}
\usepackage{yhmath}
\usepackage{mathrsfs}
\usepackage{bbding}

\def\E{\ifmmode{\mathbb E}\else{$\mathbb E$}\fi} 
\def\N{\ifmmode{\mathbb N}\else{$\mathbb N$}\fi} 
\def\R{\ifmmode{\mathbb R}\else{$\mathbb R$}\fi} 
\def\Q{\ifmmode{\mathbb Q}\else{$\mathbb Q$}\fi} 
\def\C{\ifmmode{\mathbb C}\else{$\mathbb C$}\fi} 
\def\H{\ifmmode{\mathbb H}\else{$\mathbb H$}\fi} 
\def\Z{\ifmmode{\mathbb Z}\else{$\mathbb Z$}\fi} 
\def\P{\ifmmode{\mathbb P}\else{$\mathbb P$}\fi} 
\def\T{\ifmmode{\mathbb T}\else{$\mathbb T$}\fi} 
\def\SS{\ifmmode{\mathbb S}\else{$\mathbb S$}\fi} 
\def\DD{\ifmmode{\mathbb D}\else{$\mathbb D$}\fi} 
\def\K{\ifmmode{\mathbb K}\else{$\mathbb K$}\fi}

\newcommand{\rank}{\operatorname{rank}}

\theoremstyle{theorem}
\newtheorem{thm}{Theorem}[section]
\newtheorem{cor}[thm]{Corollary}
\newtheorem{lem}[thm]{Lemma}

\newtheorem{prop}[thm]{Proposition}

\theoremstyle{definition}
\newtheorem{defn}[thm]{Definition}
\newtheorem{rem}[thm]{Remark}

\newtheorem{conds}[thm]{Condition}

\newtheorem{situ}[thm]{Situation}

\numberwithin{equation}{section}
\begin{document}
\title[Floer homology of 3-manifolds with boundary]{$SO(3)$-Floer homology of 3-manifolds with boundary 1}
\author{Kenji Fukaya}

\address{Simons Center for Geometry and Physics,
State University of New York, Stony Brook, NY 11794-3636 U.S.A.
\& Center for Geometry and Physics, Institute for Basic Sciences (IBS), Pohang, Korea} \email{kfukaya@scgp.stonybrook.edu}

\begin{abstract}
In this paper the author discuss the relation between Lagrangian 
Floer homology and Gauge-theory (Donaldson theory) Floer homology.
It can be regarded as a version of Atiyah-Floer type conjecture in the case 
of $SO(3)$-bundle with non-trivial second Stiefel-Whitney class.
This is a first of a series of papers, where we describe the main results and 
geometric and algebraic parts of the proof. The half of analytic detail 
was in \cite{fu3} which was published in 1998. The other half will appear 
in subsequent papers.
\end{abstract}

\maketitle
\tableofcontents

\section{Introduction }

Let $M$ be an oriented 3 manifold with boundary $\Sigma$ and $\mathcal E_M$ an 
$SO(3)$ bundle on $M$ such that the restriction of $\mathcal E_M$ to each of the 
connected components of $\Sigma$ is a nontrivial bundle.
(Namely $w_2(\mathcal E_M)$ is the fundamental class of $\Sigma$.)
We denote by $\mathcal E_{\Sigma}$ the restriction of $\mathcal E_M$ to $\Sigma$.
Note $\Sigma$ is necessary disconnected.
\par
Let $R(M,\mathcal E_M)$ (resp. $R(\Sigma,\mathcal E_{\Sigma})$) 
be the set of all gauge equivalence classes of flat connections of $\mathcal E_M$ on $M$ (resp. 
on $\Sigma$). Restriction of connections define a map $R(M,\mathcal E_M) 
\to R(\Sigma,\mathcal E_{\Sigma})$. Typically this map is a Lagrangian immersion.
More precisely we can perturb $R(M,\mathcal E_M)$ so that $R(M,\mathcal E_M) \to R(\Sigma,\mathcal E_{\Sigma})$ becomes a 
Lagrangian immersion in a way similar to 
\cite[1(b)]{fl1}, \cite[Section 2(b)]{Don2}.
(See also \cite{He}.)
Hereafter we write $R(M)$ and $R(\Sigma)$ in place of 
$R(M,\mathcal E_M)$ and $R(\Sigma,\mathcal E_{\Sigma})$ usually for simplicity.
\par
In \cite{AJ} Akaho-Joyce associated a filtered $A_{\infty}$ algebra to a (relatively spin) 
immersed Lagrangian submanifold $L$ of a symplectic manifold $X$. 
In this article we use $\Z_2$ coefficient.
Then its underlying 
module can be taken as 
$$
CF(L) = H(L;\Lambda^{\Z_2}_{0}) \oplus \bigoplus_{p} (\Lambda_0^{\Z_2})^2[p] 
$$ 
where the direct sum is taken over all the self-intersection points of $L$.
Here $\Lambda^{\Z_2}_0$ is the universal Novikov ring .
(See (\ref{Novikovring}) and \cite{fooobook}.)

Our main result is the following.
\begin{thm}\label{mainthm}
\begin{enumerate}
\item
The immersed Lagrangian submanifold $R(M)$ is unobstructed in 
the sense of \cite{fooobook}, \cite{AJ}.
Namely there exists a bounding cochain $b_M$ of the filtered 
$A_{\infty}$-algebra $CF(R(M))$.
\par
Furthermore there is a canonical choice of $b_M$.
Namely the gauge equivalence class of $b_M$ is an invariant of the pair $(M,\mathcal E_M)$.
\item
Let $(M_1,\mathcal E_1)$ and $(M_2,\mathcal E_2)$ be pairs of 3 manifolds 
and bundles with common boundary 
$(\Sigma,\mathcal E) = \partial(M_1,\mathcal E_1) = \partial (M_2,\mathcal E_2)$, 
such that $w_2(\mathcal E_i\vert_{\partial M_i}) = [\partial M_i]$, for $i=1,2$.
Let $(M,\mathcal E)$ be the pair obtained by gluing $(M_1,\mathcal E_1)$ and $(-M_2,\mathcal E_2)$ along 
their boundaries. (Here $-M_2$ denote $M_2$ with orientation reversed.)
Then we have a canonical isomorphism
\begin{equation}\label{glueiso}
HF(M,\mathcal E;\Lambda_0^{\Z_2})  \cong HF((R(M_1),b_{M_1}),(R(M_2),b_{M_2})).
\end{equation}
Here the group $HF(M,\mathcal E;\Lambda_2^{\Z_2})$  is the gauge theory Floer homology 
of the closed 3 manifold with (nontrivial) $SO(3)$ bundle $\mathcal E$.
It is defined in \cite{fl2}, \cite{BD1}.
See Definition \ref{defHFgauge} for the version with Novikov ring coefficient.
\item
If $R(M) \to R(\Sigma)$ is an embedding then $b_{M} = 0$.
\end{enumerate}
\end{thm}
\par
Statements (2) and (3) imply the next:
\begin{cor}\label{maincor}
Let  $(M,\mathcal E)$ be a pair of 3 manifold and $SO(3)$ bundle on it which is obtained from $(M_1,\mathcal E_1)$ and $-(M_2,\mathcal E_2)$
as in Theorem \ref{mainthm} (2). We assume $R(M_1)$, $R(M_2)$
are embedded in $R(\Sigma)$.  Then:
\begin{equation}\label{4}
HF(M,\mathcal E) \cong HF(R(M_1),R(M_2)).
\end{equation}
Here the right  hand side is the Floer homology of a pair of monotone Lagrangian submanifolds 
defined by Oh \cite{Oh}.
The isomorphism is one between $\Z_4$ periodic $\Z_2$ vector spaces.
\end{cor}
We may regard Corollary \ref{maincor} as a version of $SO(3)$ analogue of 
Atiyah-Floer conjecture \cite{At}.
\par
Note in case $M_1 = M_2 = \Sigma \times [0,1]$ 
the isomorphism (\ref{4}) is proved by S. Dostoglou and D.A. Salamon
in \cite{DS}.
 \par
Theorem \ref{mainthm} together with other results 
we will explain below
realize the project the author proposed in \cite{fu0}, \cite{fu1}, \cite{fu2}.
We like to mention that there were various proposals such as 
\cite{LLW}, \cite{Sa}, \cite{Yo}
 around the same time (early 1990's).
\par
To put Theorem \ref{mainthm} to its natural perspective we 
first state some other results which are more closely related to the 
project in  \cite{fu1}.
We consider the filtered $A_{\infty}$ category
$\mathscr{FUK}(R(\Sigma))$  whose object is 
$(L,b)$ where $L$ is an immersed Lagrangian submanifold
and $b$ is the bounding cochain of
Akaho-Joyce's filtered $A_{\infty}$ algebra associated to 
$L$.
In the current situation it is one over $\Lambda^{\Z_2}_{0}$.
\par
In case we consider only embedded Lagrangian submanifolds as an object 
such a category $\mathscr{FUK}(R(\Sigma))$  is constructed in 
\cite{fu4} and \cite{ancher}.
Based on the ideas of \cite{AJ} we can enhance it to the 
version including immersed Lagrangian submanifolds.
\begin{thm}\label{functordef}
For each $(M,\mathcal E_M)$ which bounds $(\Sigma,\mathcal E_{\Sigma})$
there exists a filtered $A_{\infty}$ functor
$\mathscr{FUK}(R(\Sigma)) \to {\mathscr{CH}}$.
Here ${\mathscr{CH}}$ is the differential graded category 
(in the sense of \cite{bondkap}) whose object 
is a chain complex.
\par
The homotopy equivalence classes of this filtered $A_{\infty}$ 
functor is an 
invariant of $(M,\mathcal E_M)$.
\end{thm}
We denote this filtered $A_{\infty}$ functor by $\mathcal{HF}_{(M,\mathcal E_{M})} : 
\mathscr{FUK}(R(\Sigma)) \to {\mathscr{CH}}$.
The construction of such functor is explained in \cite{fu0}, \cite{fu1}, \cite{fu2}.
(See in particular \cite[Section 4]{fu2}.)
A part of such idea is  realized by  Wehrheim \cite{We} and 
Salamon-Wehrheim \cite{Sawe}
based on a similar but a slightly different analytic setting.
We will explain the proof of Theorem \ref{functordef} in Section \ref{HF3bdfunctor}
based on the analytic setting in \cite{fu3}.
(Its detail will appear in \cite{fu8}.)

\begin{thm}\label{represent}
The filtered $A_{\infty}$ functor $\mathcal{HF}_{(M,\mathcal E_{M})}$ is homotopy equivalent to the 
filtered $A_{\infty}$ functor represented by $(R(M),b_M)$.
\end{thm}
Let us elaborate the statement of Theorem \ref{represent}.
Let $(L,b)$  be an object of $\mathscr{FUK}(R(\Sigma))$.
Theorem \ref{functordef} associates a group
$$
HF((M,\mathcal E_M),(L,b))
= 
\mathcal{HF}_{(M,\mathcal E_{M})}(L,b).
$$
More precisely the right hand side is a homology group of the 
chain complex 
associated to the object $(L,b)$ by the functor $\mathcal{HF}_{(M,\mathcal E_{M})}$
in Theorem \ref{functordef}.
Then Theorem \ref{represent} claims the existence of an isomorphism 
\begin{equation}\label{formula13}
HF((M,\mathcal E_M),(L,b)) \cong
HF((R(M,\mathcal E_{M}),b_M),(L,b)).
\end{equation}
\par
Moreover this isomorphism is functorial in the following sense.
Using the fact that $\mathcal{HF}_{(M,E_{M})}$ is a filtered $A_{\infty}$
functor we obtain a map
\begin{equation}
HF((M,\mathcal E_M),(L_1,b_1)) \otimes HF((L_1,b_1),(L_2,b_2)) 
\to HF((M,\mathcal E_M),(L_2,b_2)).
\end{equation}
Then the following diagram commutes.
$$
\begin{CD}
HF((M,\mathcal E_M),(L_1,b_1)) \otimes HF((L_1,b_1),(L_2,b_2))  @ >{}>>
HF((M,\mathcal E_M),(L_2,b_2))  \\
@ V{}VV @ VV{}V\\
HF((R(M),b_M),(L_1,b_1)) \otimes HF((L_1,b_1),(L_2,b_2)) @ > {} >> 
HF(R(M,\mathcal E_{M}),(L_2,b_2)) 
\end{CD}
$$
Here the horizontal arrow in the second line is the composition 
of the morphisms in the filtered $A_{\infty}$ category $\mathscr{FUK}(R(\Sigma))$.
The vertical arrows are induced by (\ref{formula13}).
\par
We remark that we need to show the commutativity of similar diagrams 
including higher multiplication operators 
to obtain an $A_{\infty}$ functor and prove Theorem \ref{represent}. See \cite[Definition 7.1]{fu4}
for the definition of the notion of filtered $A_{\infty}$ functor.
See \cite[Definition 7.31]{fu4} for the definition of the notion of representable 
filtered $A_{\infty}$ functor.
\par
The existence of the $A_{\infty}$ functor in Theorem \ref{functordef} is 
\cite[Theorem 4.8$^*$]{fu2}.
Here $^*$ was put to the number of theorems in \cite{fu2} according to the rule mentioned in 
\cite{fu2} page 8 second paragraph, that it,
`Since we postpone the analytic detail to subsequent papers, we put * to the statements which will be proved in subsequent papers.'
\par
It is the understanding of the author that an $A_{\infty}$ functor of a similar 
nature in a related context of Heegaard Floer theory of  
Ozsvath-Szabo is now established, based on more combinatorial method. (See for example \cite{LOT}.)
The author believes that there is a similar story as those in this paper 
in the case of Seiberg-Witten Floer theory, which implies, for example, 
the coincidence of Seiberg-Witten Floer homology and Heegard Floer homology.
\par
Note Theorems \ref{represent} and \ref{mainthm} (2) together with 
$A_{\infty}$ analogue of Yoneda's lemma (\cite[Theorem 8.4]{fu4}. See also 
\cite{Lef}.) implies the next corollary.
\begin{cor}\label{oldconj}
Let $(M_i,\mathcal E_i)$ be as in Theorem \ref{mainthm} (2).
Then we have
$$
HF(M,\mathcal E_M) \cong H(\mathscr{HOM}(\mathcal{HF}_{(M_1,\mathcal E_1)},
\mathcal{HF}_{(M_2,\mathcal E_2)})).
$$
Here the right hand side is the homology group of the chain complex 
consisting of all pre-natural transformations from the 
filtered $A_{\infty}$ functor  
$\mathcal{HF}_{(M_1,\mathcal E_1)}$ to $\mathcal{HF}_{(M_2,\mathcal E_2)}$,
which is defined in \cite[Definition 7.49]{fu4} and 
\cite[Definition 10.1]{fu2}.
\end{cor}
Corollary \ref{oldconj} is 
\cite[Conjecture 5.24]{fu0}, \cite[Conjecture 3.3]{fu1} and \cite[Conjecture 8.9]{fu4}.
(More precisely it was conjectured there that a particular map 
defined there gives this isomorphism. We can show the statement 
in that form also.)\footnote{We  need to invert the 
formal generator $T$ of the Novikov ring to prove Corollary \ref{oldconj},
because $A_{\infty}$ Yoneda's lemma is proved only after inverting $T$.}
\par\smallskip
As we mentioned already $M = \Sigma \times [0,1]$
is an example where $b_M = 0$. 
In fact $R(M)$ is embedded in $R(\partial M) = R(\Sigma)^2$,
as a diagonal.
The author has no doubt that there are plenty of examples 
where $b_M \ne 0$.
A possible way to cook up such an example is as follows.
\par
We pretend that the result of this paper holds for $\Q$-coefficient in place 
of $\Z_2$-coefficient.
We take $M_1 = \Sigma \times [0,1]$. Let 
$M_2$ be the connected sum of Poincar\'e homology sphere and 
Poincar\'e homology sphere with orientation reversed.
It is proved in \cite[Proposition 5.5]{fu15} that 
$SU(2)$ Floer homology of $M_2$ vanishes over $\Q$.
So  it seems very likely that
$
\mathcal{HF}_{(M_1\#M_2,\mathcal E)}
$
is represented by the diagonal $R(\Sigma) \subset R(\Sigma)^2$
over $\Q$ coefficient.
Here $\mathcal E$ is the (unique) $SO(3)$ bundle such that $w_2(\mathcal E)$
is the generator of $H_2(M;\Z_2) = \Z_2$.
On the other hand, using the fact that the fundamental group of 
the Poincar\'e homology sphere has 2 irreducible representations
over $SO(3)$,
the space $R(M)$ has many connected components.
(4 connected components diffeomorphic to $SO(3) \times SO(3) 
\times R(\Sigma)$, 4 connected components diffeomorphic to $SO(3)  
\times R(\Sigma)$, and one connected components diffeomorphic to 
$R(\Sigma)$.) Because of $SO(3)$ factors, the space $R(M)$ is not of correct dimension.
(Namely its dimension is different from $\dim R(\partial M)/2$.)
After perturbation we have still many connected components 
corresponding to the generators of $H(SO(3)^2)$, $H(SO(3))$ etc..,
each of which can be taken to be diffeomorphic to $R(\Sigma)$.
The image of them in $R(\partial M)$ all can be taken to be the 
diagonal.
We can perturb so that those components intersect transversally 
to each other.
We thus end up with a complicated configuration of the 
embedded Lagrangian submanifolds whose 
connected components are perturbations of the diagonal
(and are embedded). 
The union of such Lagrangian submanifolds 
is regarded as an immersed Lagrangian submanifold.
Together with bounding cochain $b_M$
this immersed Lagrangian submanifold 
must be $\Q$-coefficient Floer theoretically equivalent to  the diagonal.
Therefore $b_M$ can not be $0$ over $\Q$.
\par
In this example there are ASD connections on $M_2 \times \R$,
which gives cancellation over $\Q$ of the generators in the Floer's chain 
complex of $M_2$. Those ASD connections are not visible 
from the space $R(M)$. The bounding cochain $b_M$ however 
`remember' those ASD connections.
\par\smallskip
The author studied the problem we discuss in this paper 
in 1990's. There are long blank before he restarted it, in this year 2015.
During this blank the author had not been working on this project.
In 1990's the present author was close to constructing an
$A_{\infty}$ functor,
$\mathcal{HF}_{(M,\mathcal E_{M})} : 
\mathscr{FUK}(R(\Sigma)) \to {\mathscr{CH}}$,
in a series of papers  such as \cite{fu0}, \cite{fu1}, \cite{fu2}, \cite{fu3}.
In this paper we slightly changed the formulation 
of the construction of this functor. 
However the changes are mostly due to the development of the 
Lagrangian Floer theory in last 20 years and are not 
due to one on the gauge theory part of the story. We apply the language of 
filtered $A_{\infty}$ algebra, a module over it, and filtered $A_{\infty}$
category, which we have developed, especially in 
\cite{fooobook}, \cite{fooobook2}, \cite{fu4}.
We also include immersed Lagrangian submanifolds into the story, 
based on the work by Akaho-Joyce \cite{AJ}. 
To include immersed case is essential to formulate Theorem \ref{mainthm} (1).
Especially it is also crucial for the representativity of the 
functor $\mathcal{HF}_{(M,\mathcal E_{M})}$.
\par
Around the time when \cite{fu3}  was written, the author was on the way 
in establishing and writing up the detailed construction of the moduli spaces 
which we use in this paper. In fact \cite{fu3} proved compactness and 
removable singularity theorem of those moduli spaces.
The most important piece of analytic results which was not included in \cite{fu3}
is Fredholm theory, that is, a construction of appropriate Fredholm 
complex which provides the linearization of the nonlinear equation we 
study. 
\footnote{The reason why the author stopped working on this project 
during 1998-2014 is {\it not} because he could not find the way to work
out the Fredholm theory.}
\par
The author stopped working on this project for a long time because 
he could not figure out the way to prove the most 
important part of the project, which consists of 
Theorem \ref{mainthm} of present paper. 
Especially he was unable to find a way to prove gluing result, that is, 
Theorem \ref{mainthm} (2).
(A gluing result similar to Theorem \ref{mainthm} (2) was stated as 
{\it Conjecture} 8.9 and not as Theorem 8.9$^*$ in \cite{fu2}.) 
In 1990's the author did not know, either, an appropriate condition under which 
the relative Floer homology 
functor $\mathcal{HF}_{(M,\mathcal E_M)}$ is determined only by the Lagrangian submanifold
$R(M)$. Theorem \ref{mainthm} (3) says that the relevant 
condition is that $R(M) \to R(\Sigma)$ is an {\it embedding}.
Since the author thought that those key issues were yet very hard to attack at that stage
he did not work on this project during 1999 - 2014 and 
concentrated in working out symplectic geometry side of the story. 
\par
Meanwhile Wehrheim's paper \cite{We} appeared. It provides 
analytic results on a similar moduli space. 
It is used by Salamon-Wehrheim \cite{Sawe} for a part of the construction of the functor 
$\mathcal{HF}_{(M,\mathcal E_{M})} : 
\mathscr{FUK}(R(\Sigma)) \to {\mathscr{CH}}$.
\par
The argument which we explain in Sections \ref{sec:unbolt},\ref{sec:represent},\ref{sec:glue} 
of present paper resolves the points the author could not go through in 1990's. 
So we can now prove the results claimed in this section, using the 
basic properties of the moduli space introduced in \cite{fu3}.
(We explain those moduli spaces in Sections \ref{HF3bdfunctor},\ref{sec:represent},\ref{sec:glue}.)
I like to mention that the idea used in Sections \ref{sec:represent},\ref{sec:glue} is  related to the work by
Y. Lekili and  M. Lipyanskiy \cite{LL},
which they used to study Wehrheim-Woodwards functoriality 
(\cite{WW}).
(See Remarks \ref{rem312}, \ref{rem415}, \ref{rem513} 
for the results we  prove on Wehrheim-Woodwards functoriality 
in a way parallel to the gauge theory results we are discussing here.)
\par
So the point which is the most novel in this paper is an algebraic lemma, Proposition 
\ref{thm35}, and 
an observation that it can be used to prove Theorem \ref{mainthm} (1).
\par
The author is aware that Lipyanskiy have been studying 
gauge theory Floer homology and Atiyah-Floer conjecture (\cite{Li}).
\par
Now it's the time to complete the project we started 20 years ago.

\section{Floer homology of 3 manifolds with boundary as a filtered $A_{\infty}$ 
functor: review with update}\label{HF3bdfunctor}
In this section, we explain the construction of the 
filtered $A_{\infty}$ functor,
$\mathcal{HF}_{(M,\mathcal E_{M})} : 
\mathscr{FUK}(R(\Sigma)) \to {\mathscr{CH}}$.
The construction is in principle the same as those in 
\cite{fu0},\cite{fu1},\cite{fu2}. 
(More precisely it is minor modification of \cite[Part 1: geometry]{fu2}. 
The other part, \cite[Part 2: algebra]{fu2} was
rewritten and was published as \cite[Chapter 2]{fu4}.) 
We however apply the language developed in the past decades systematically 
for the discussion in this section.

\begin{situ}\label{situ21}
$M$ is a 3 manifold with boundary $\Sigma$
and $\mathcal E_M$ is an $SO(3)$-bundle on $M$ such that 
$\mathcal E_{\Sigma} = \mathcal E_M\vert_{\Sigma}$ is nontrivial 
on each of its connected components.
\par
We assume $R(M,\mathcal E_M)$, the set of the gauge equivalence classes of 
flat connections of $\mathcal E_M$, is a smooth
manifold of dimension $\frac{1}{2}\dim R(\Sigma)$ and the restriction map $i_{R(M)} : R(M,\mathcal E_M) \to R(\Sigma,\mathcal E_{\Sigma})$
is a Lagrangian immersion with transversal self intersection.
(More precisely we assume that $H^1(M,d_A) = T_{[A]}R(M,\mathcal E_M)$.)
\end{situ}
Note we can relax the second half of the condition by perturbing 
appropriately. In this article we omit the argument to do so for 
simplicity.
\par
Let $L$ be another immersed Lagrangians submanifold of 
$R(\Sigma)$ with 
transversal self-intersection. 
We write $i_L : \tilde L \to R(\Sigma)$ 
the immersion with image $L$. We consider the module
\begin{equation}
CF(L) 
= C(\tilde L;\Lambda_{0}^{\Z_2}) \oplus
\bigoplus_{(p,q)} \Lambda_{0}^{\Z_2}[p,q]. 
\end{equation}
Here the direct sum runs on the pair of points $(p,q) \in \tilde L^2$
such that $p\ne q$ and $i_L(p) = i_L(q)$. 
In other words we associate two generators to each of the self intersection points, 
following \cite{Ak}, \cite{AJ}.
\par
Here and hereafter $C(\tilde L;\Lambda_{0}^{\Z_2})$ is a chain model of the 
(singular) cohomology of $\tilde L$.
There are various possible choices of the chain model. 
For example we can take 
$C(\tilde L;\Lambda_{0}^{\Z_2}) =  H(\tilde L;\Lambda_{0}^{\Z_2})$,
the (co)homology group itself.
(Here $CF$ stands for Floer chain complex.)
\par
We remark the following isomorphism:
\begin{equation}\label{fiberCFdef}
CF(L) 
\cong 
C(\tilde L \times_{R(\Sigma)} \tilde L;\Lambda_0^{\Z_2}).
\end{equation}
\begin{rem}
We can generalize the story of \cite{AJ} to the case when self-intersection is not 
necessary transversal but is clean, by taking (\ref{fiberCFdef}) as the definition 
of $CF(L)$.
\end{rem}
\par
We denote by $\Lambda_{0}^{\Z_2}$ a universal Novikov ring with $\Z_2$ coefficient.
(See \cite[(1.3)]{toric1}.)
Its element is a formal sum
\begin{equation}\label{Novikovring}
\sum_{i=0}^{\infty} a_i T^{\lambda_0}
\end{equation}
where $0 = \lambda_0 < \lambda_1 < \cdots \uparrow \infty$ and $a_i \in \Z_2$.
\par
Akaho-Joyce (generalizing \cite{fooobook}) defined a series of operators 
$$
\frak m_k : CF(L) ^{k\otimes} \to CF(L)
$$
using appropriate moduli spaces of pseudo-holomorphic polygons
(See Figure 2.1),
and showed that $(CF(L),\{\frak m_k\})$ becomes a filtered $A_{\infty}$
algebra. 
(See \cite[Definition 3.2.3]{fooobook} for the definition of filtered $A_{\infty}$
algebra.)
\par
\begin{center}
\includegraphics[scale=0.2]
{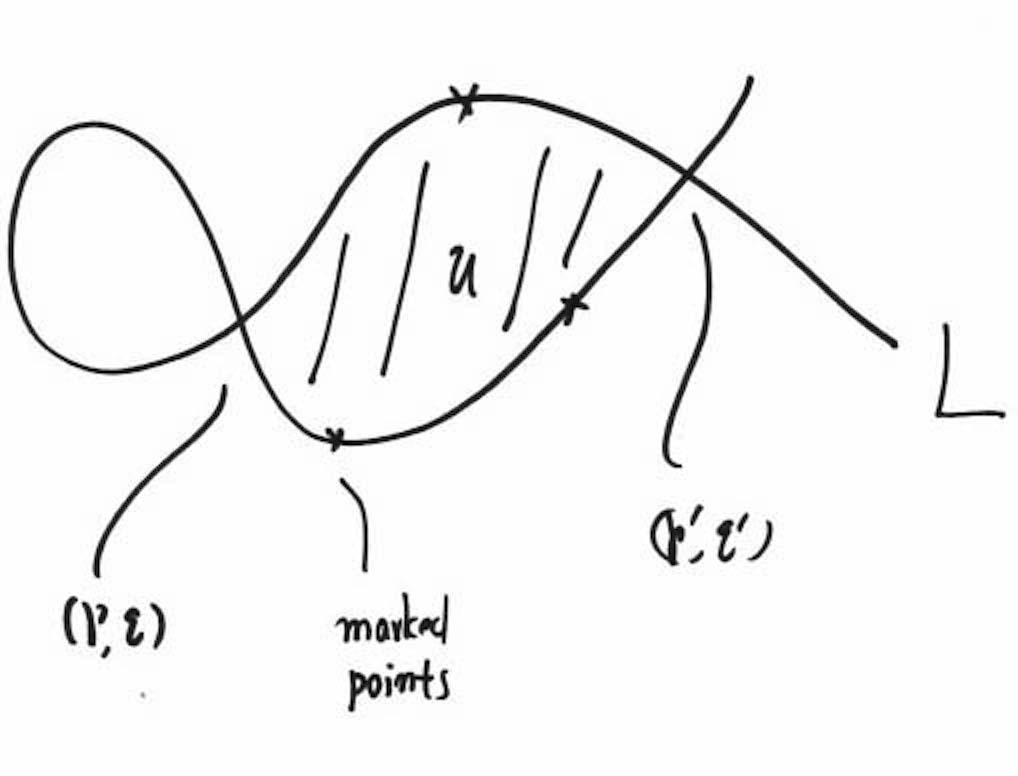}
\end{center}
\centerline{\bf Figure 2.1}
\par
Note $\frak m_0$ is included.  In other words,
our filtered $A_{\infty}$ algebra is curved, in general.
We remark that our symplectic manifold $R(\Sigma,\mathcal E_{\Sigma})$
is monotone. Therefore Lagrangian Floer theory works over $\Z_2$
coefficient by \cite{fooo:overZ}. Note in
\cite{fooo:overZ} we discussed the case of embedded Lagrangian submanifold. 
However we can easily combine its argument with \cite{AJ} to include the 
immersed case.
We also remark that in \cite{fooo:overZ} we assumed that the ambient symplectic 
manifold is spherically positive but did not require any condition for its Lagrangian 
submanifold. A montone symplectic manifold is spherically positive.
\par
We denote by $\mathcal M(L)$ the set of all the
bounding cochains of $(CF(L),\{\frak m_k\})$, that is, an element $b$ of $CF(L) \otimes_{\Lambda_0^{\Z_2}}
\Lambda_+^{\Z_2}$ such that
\begin{equation}\label{MCequ}
\sum_{k=0}^{\infty} \frak m_k(b,\dots,b) = 0.
\end{equation}
Here $\Lambda_+^{\Z_2}$ is the maximal ideal of $\Lambda_0^{\Z_2}$
consisting of formal sum (\ref{Novikovring}) with $\lambda_0 > 0$.

We put
$$
\frak m_k^b(x_1,\dots,x_k)
=
\sum_{\ell_0\ge 0,\ell_1\ge 0,\dots,\ell_k\ge 0}
\frak m_k(b^{\ell_0\otimes},x_1,b^{\ell_1\otimes},\dots,b^{\ell_{k-1}\otimes},x_k,b^{\ell_{k}\otimes}).
$$
$(CF(L),\{\frak m_k^b\})$ is again a filtered $A_{\infty}$ algebra and 
$\frak m_0^b =0$ is equivalent to (\ref{MCequ}).
\par
We refer the reader \cite[Chapter 4]{fooobook} for homological algebra
of bounding cochain.
\par
We assume that $L$ is of clean intersection with the Lagrangian immersion 
$i_{R(M)} : R(M) \to R(\Sigma)$.
Namely we assume that
$$
{\rm Im} (T_xR(M)) \cap {\rm Im} (T_y\tilde L) \subseteq T_{z}R(\Sigma)
$$
has locally constant rank for $(x,y) \in R(M) \times_{R(\Sigma)} \tilde L$ with 
$z = i_{R(M)}(x) = i_L(y)$ and that 
this rank coincides with the dimension of the submanifold 
$R(M) \times_{R(\Sigma)} \tilde L$ of $R(M) \times \tilde L$, which we 
assume to be smooth.
(Including the case when the intersection is clean but not transversal 
is important in the discussion of Section \ref{sec:unbolt}.)
\par
We consider the fiber product
$R(M) \times_{R(\Sigma)} \tilde L$, which is a
smooth manifold by assumption and put
\begin{equation}
CF((M,\mathcal E_M),L) = C(R(M) \times_{R(\Sigma)} L;\Lambda_{0}^{\Z_2}).
\end{equation}
Here the right hand side is a chain model of ordinary cohomology with $\Lambda_{0}^{\Z_2}$ 
coefficient.
We define a structure of filtered $A_{\infty}$ right module on the graded vector space 
$CF((M,\mathcal E_M),L)$ over the filtered $A_{\infty}$ algebra $(CF(L),\{\frak m_k\})$,
following and extending the idea of \cite{fu0}, \cite{fu1}, \cite{fu2}.
Note such a structure by definition assigns a series of operators
\begin{equation}\label{openn}
\frak n_k : CF((M,\mathcal E_M),L) \otimes CF(L) ^{k\otimes} \to CF((M,\mathcal E_M),L)
\end{equation}
which satisfies the relation
\begin{equation}\label{rightmodule}
\aligned
&\sum_{\ell=0}^{k}
\frak n_{k-\ell}(\frak n_{\ell}(y;x_1,\dots,x_{\ell});x_{\ell+1},\dots,x_k) \\
&+
\sum_{0\le \ell \le m\le k}
\frak n_{k-m+\ell+1}(y;x_1,\dots,\frak m_{m-\ell}(x_{\ell},\dots,x_{m-1}),\dots,x_k) = 0.
\endaligned
\end{equation}
(See \cite[Section 3.7.1]{fooobook}.
Note there it was discussed the case of $C_1$-$C_2$ filtered $A_{\infty}$ bimodule $D$. 
The case of right filtered $A_{\infty}$ module is its special case 
where $C_1$ is $\Lambda_0$.).
For $[b] \in \mathcal M(L)$  we define $d^b : CF((M,\mathcal E_M),L) \to CF((M,\mathcal E_M),L)$
by
\begin{equation}\label{form28}
d^b(y) = \sum_{k=0}^{\infty} \frak n_k(y;b,\dots,b).
\end{equation}
(\ref{MCequ}) and (\ref{rightmodule})  imply that
\begin{equation}
d^b \circ d^b = 0.
\end{equation}
(See \cite[Lemma 3.7.14]{fooobook}.)
We now define a $\Lambda_0^{\Z_2}$ module
\begin{equation}
HF((M,\mathcal E_M),(L,b)) = \frac{{\rm Ker} \,d^b}{{\rm Im} \,d^b}.
\end{equation}
In the rest of this section we explain the construction of the maps $\frak n_k$.
\par
We take a Riemannian metric on $M$ such that 
there exists a compact subset $M_0$ and an isometry
$$
M \setminus M_0 \cong (-1,1] \times \Sigma,
$$
where $\partial M = \Sigma$ corresponds to $\{1\} \times \Sigma$.
Here the metric of the right hand side is the direct product metric
with a K\"ahler metric $g_{\Sigma}$ on $\Sigma$.
We fix $g_{\Sigma}$ hereafter.
Note $g_{\Sigma}$ determines a complex structure of $R(\Sigma)$.
We take a smooth function $\chi : (-1,1] \to [0,1]$ such that
$\chi(s) \equiv 1$ on a neighborhood of $\{-1\}$ and 
$$
\{s \mid \chi(s)  = 0 \} = [0,1].
$$
We also assume
$$
\chi(s) = e^{1/s} \qquad \text{for $s \in (-\epsilon,0)$.}
$$
\par
We consider the `ASD-equation' of 4 manifolds $M \times \R$ with respect to the 
singular metric
\begin{equation}\label{form211}
{\bf g} = 
\begin{cases}
g_M + dt^2 &\text{on $M_0 \times \R$} \\
\chi(s)^2  g_{\Sigma} + ds^2 + dt^2  
&\text{on $(M \setminus M_0) \times \R$}.
\end{cases}
\end{equation}
Note we identify 
$$
M \setminus M_0 \cong \Sigma \times (-1,1] \times \R
$$
and use $s$ (resp. $t$) for the coordinate of the $(-1,1]$ (resp. $\R$) factor.
\par
Since the metric ${\bf g}$ is singular the ADS equation
\begin{equation}\label{ASDeq}
F_{\frak A} + *_{\bf g} F_{\frak A} = 0
\end{equation}
does not make sense. Here ${\frak A}$ is a connection of the $SO(3)$ bundle 
$\mathcal E_M \times \R$, $F_{\frak A}$ its curvature,
and $*{\bf g}$ is the Hodge $*$ operator of the `metric' ${\bf g}$.
However we can make sense of the equation (\ref{ASDeq}) as follows.
We write the restriction of ${\frak A}$ to 
$M \setminus M_0 \cong \Sigma \times (-1,1] \times \R$ as 
\begin{equation}\label{AnadfraA}
\frak A = A +\Phi ds + \Psi dt
\end{equation}
where $A = A(s,t)$ is a $(s,t) \in (-1,1] \times \R$ parametrized family 
of connections of $\mathcal E_{\Sigma}$ and $\Phi$, $\Psi$ are $(s,t) \in (-1,1] \times \R$ parametrized families of the sections of 
$so(3)$-bundle associated to $\mathcal E_{\Sigma}$ by the adjoint representation
of $SO(3)$ on $so(3)$.
Then on the domain $\Sigma \times (-1,0) \times \R$ where ${\bf g}$ is indeed a 
Riemannian metric and (\ref{ASDeq}) makes sense, we can rewrite Equation (\ref{ASDeq}) as follows.
\begin{equation}\label{ASDprod}
\aligned
&\frac{\partial A}{\partial  t} - d_A\Psi - 
*_{\Sigma}\left( \frac{\partial A}{\partial s} - d_A\Phi \right) = 0, \\
&\chi(s)^2 
\left(
\frac{\partial \Psi}{\partial s} - \frac{\partial \Phi}{\partial t} 
+ [\Phi,\Psi]
\right)
+ *_{\Sigma} F_A = 0.
\endaligned
\end{equation}
See \cite{DS}, \cite{fu3}.
We observe that Equation (\ref{ASDprod}) makes sense also when $\chi(s) = 0$.
In that case we may regard the solution of (\ref{ASDprod}) as a holomorphic map 
$[0,1]\times \R \to R(\Sigma)$ as follows.
When $\chi(s) = 0$, the second equation means that $F_A$, the curvature of the connection 
$A$ (of $\mathcal E_{\Sigma}$), is $0$. Namely $A(s,t)$ is a flat 
condition. Therefore $(s,t) \mapsto [A(s,t)]$ defines a map
$[0,1] \times \R \to R(\Sigma)$.
The first equation implies that this map is holomorphic as follows.
Note $\frac{\partial A}{\partial s} - d_A\Phi$ and $ \frac{\partial A}{\partial t} - d_A\Psi$
are $d_A$-closed forms since $A(s,t)$ is flat for $(s,t) \in (0,1] \times \R$. Therefore 
the first equation implies that $\frac{\partial A}{\partial s} - d_A\Phi$ and $ \frac{\partial A}{\partial t} - d_A\Psi$
are both harmonic. (Namely they are both $d_A$ and $d_A^*$ closed). They represent the $s$ (resp. $t$) derivative 
of our map $(s,t) \mapsto [A(s,t)]$.
The tangent space $T_{A(s,t)}R(\Sigma)$ is identified with the set of harmonic 
forms.  The complex structure of $R(\Sigma)$ is obtained by the Hodge $*_{\Sigma}$
on harmonic forms. 
(See, for example, \cite[Section 2]{fu3}.)
Thus the first equation implies that $(s,t) \mapsto [A(s,t)]$
is holomorphic.
\par
The operator (\ref{openn}) is obtained by `counting' the order of certain moduli space
of solutions of (\ref{ASDeq}), (\ref{ASDprod}) with appropriate boundary conditions, 
which we will describe below.
We first discuss the case when $L$ is an embedded Lagrangian submanifold and 
will explain its generalization to the immersed case later.
\par
Let $\frak A$ be a solution of (\ref{ASDeq}),(\ref{ASDprod}).
(Namely we require (\ref{ASDeq}) on $M_0 \times \R$ and 
(\ref{ASDprod}) on $(M\setminus M_0) \times \R$.)
We define the energy $\mathcal E(\frak A)$ as follows.
\begin{equation}\label{energy}
\mathcal E(\frak A)
=
\int_{(M\setminus (\Sigma \times [0,1))) \times \R}
\Vert F_{\frak A}\Vert^2 \Omega_{\bf g}
+
\int_{[0,1) \times \R} 
\left(
\left\Vert \frac{\partial u}{\partial s} \right\Vert^2
+
\left\Vert \frac{\partial u}{\partial t} \right\Vert^2
\right) dsdt.
\end{equation}
Here $\Omega_{\bf g}$ is the volume form of the metric ${\bf g}$
and we define $u$ by $u(s,t) = [A(s,t)] \in R(\Sigma)$.
The norm appearing in the second term of the right hand side 
is the norm of the vector. The norm is induced by the 
metric obtained by using $g_{\Sigma}$.
See \cite[Section 2]{fu3}.
\par
Hereafter we write
$$
\Vert \nabla u\Vert^2 = 
\left\Vert \frac{\partial u}{\partial s} \right\Vert^2
+
\left\Vert \frac{\partial u}{\partial t} \right\Vert^2.
$$
\par
We assume that $L$ is of clean intersection with $R(M)$.
Let $R_1$ and $R_2$ be connected components of the fiber product
$
R(M) \times_{R(\Sigma)} \tilde L.
$
We consider the connection $\frak A$ solving (\ref{ASDeq}), (\ref{ASDprod}).
As we explained above $(s,t) \mapsto [A(s,t)]$ defines a 
holomorphic map $[0,1] \times \R \to R(\Sigma)$.
We consider the following boundary condition:\footnote{
In case $L$ is immersed we need to set a boundary condition 
a bit more carefully. See Definition \ref{defn23} (6), (7).}
\begin{conds}\label{cond22}
There exists  a smooth map $\gamma : (-\infty,+\infty) \to L$
such that
\begin{equation}
[A(1,t)] = i_L(\gamma(t)) \in R(\Sigma).
\end{equation}
Here $i_L : L \to R(\Sigma)$ is the Lagrangian embedding .
\end{conds}
See Figure 2.2.
\par
\begin{center}
\includegraphics[scale=0.25]
{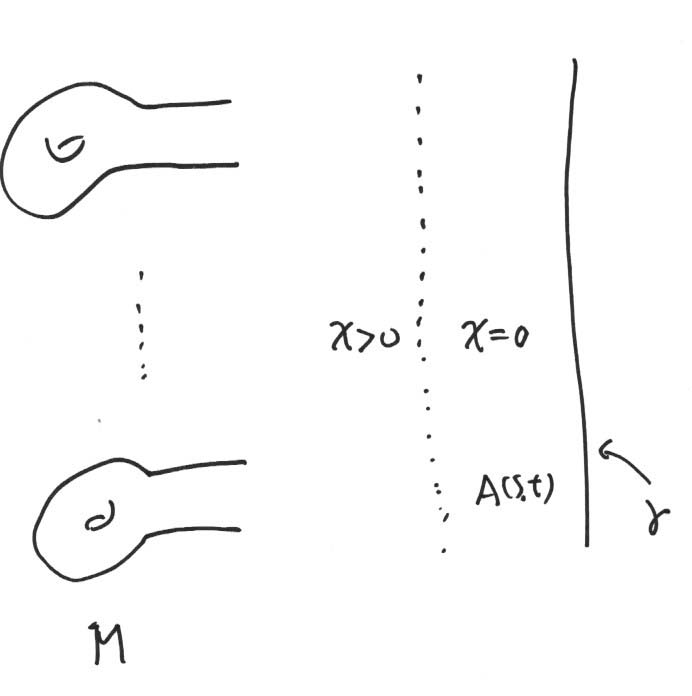}
\end{center}\centerline{\bf Figure 2.2}
\par

We have the following:
\begin{thm}\label{asymptotic}
Suppose $\frak A$ is a solution of (\ref{ASDeq}), (\ref{ASDprod}). We assume that 
its energy is finite. Then there exists a gauge transformation 
$g \in {\rm Aut}(\mathcal E_{M}\times \R)$ and $a_{-}, a_+ \in R(M)$ 
with the following properties.
\begin{enumerate}
\item
$
\lim_{t\to \infty} g^*\frak A\vert_{M\times \{t\}} = a_+,
$
and
$
\lim_{t\to -\infty} g^*\frak A\vert_{M\times \{t\}} = a_-.
$
\item
If $A(s,t)$ is obtained by the restriction of $\frak A$ to $\Sigma \times \{(s,t)\}$
then
$[A(1,t)] = \gamma(t)$.
\item
$\lim_{t\to -\infty}i_L(\gamma(t)) = [a_-\vert_{\Sigma}]$,
$\lim_{t\to +\infty}i_L(\gamma(t)) = [a_+\vert_{\Sigma}]$.
\end{enumerate}
\end{thm}
Here $[*]$ stands for the gauge equivalence class of the connection $*$.
\par
We did not specify the ratio of the convergence 
in Theorem \ref{asymptotic} (1).
Actually we can prove that there exist positive numbers $c_k$, $C_k$ such that:
\begin{equation}\label{expesti}
\Vert g^*\frak A\vert_{M\times \{t\}} - a_{\pm}\Vert_{C^k} \le C_k e^{\mp c_k t}.
\end{equation}
We remark that (\ref{expesti}) implies that the convergence in item (3) is also of exponential 
order.
\begin{rem}\label{rem2525}
We postpone the proof of Theorem \ref{asymptotic} to \cite{fu7}, \cite{fu8}. 
\par
In the case when $M$ has no boundary and $R(M)$ is zero dimensional, 
Floer \cite{fl1}
proved  Theorem \ref{asymptotic} together with the estimate (\ref{expesti}).
In the case $M$ has no boundary and
$R(M)$ is of Bott-Morse type, that is, the critical points set $R(M)$ 
of the Chern-Simons functional is of Bott-Morse type, Theorem \ref{asymptotic} together with the estimate (\ref{expesti})
is proved in \cite[Lemma 7.13]{fu15}. 
See also \cite{MMR}, where (in the case $M$ has no boundary) Theorem 
\ref{asymptotic} is proved {\it without} assuming that $R(M)$ is clean in the Bott-Morse 
sense. (In that generality the estimate (\ref{expesti}) is false.)
We need certain modification of the proof  of \cite{fu15}  to 
prove Theorem \ref{asymptotic} and estimate (\ref{expesti}).
\par
A similar result is proved in \cite[Theorem 5.1]{Sawe} in a slightly different setting.
\end{rem}
We now define:
\begin{defn}\label{defn2626}
We define the moduli space $\overset{\circ}{\widetilde{\mathcal M}}((M,\mathcal E),L;R_-,R_+;E)$
as the set of all the gauge equivalence classes of $\frak A$ such that:
\begin{enumerate}
\item $\frak A$ satisfies (\ref{ASDeq}), (\ref{ASDprod}).
\item
There exists $\gamma : (-\infty,+\infty) \to \tilde L$ such that Condition \ref{cond22}
is satisfied.
\item
Let $a_{-}$ and $a_+$ be as in the conclusion of Theorem \ref{asymptotic}.
Then
\begin{equation}
a_- \in R_-,
\qquad 
a_+ \in R_+.
\end{equation}
\item
$\mathcal E(\frak A) = E$, where the left hand side is the energy defined by
(\ref{energy}).
\end{enumerate}
$\overset{\circ\circ}{\widetilde{\mathcal M}}((M,\mathcal E),L;R_-,R_+;E)$
has a natural $\R$ action defined by the translation on $\R$ direction.
We denote by $\overset{\circ\circ}{\mathcal M}((M,\mathcal E),L;R_-,R_+;E)$
the quotient of $\overset{\circ\circ}{\widetilde{\mathcal M}}((M,\mathcal E),L;R_-,R_+;E)$
by this $\R$ action.
\end{defn}
In the same way as \cite[1(b)]{fl2}, \cite[Section 2 (b)]{Don2}, we can perturb our equation (\ref{ASDeq}), (\ref{ASDprod})
on $M_0 \times \R$ 
(that is, at the gauge theory part) so that 
$\overset{\circ\circ}{{\mathcal M}}((M,\mathcal E),L;R_-,R_+;E)$
becomes a smooth manifold.
Hereafter in this paper we denote by 
$\overset{\circ\circ}{{\mathcal M}}((M,\mathcal E),L;R_-,R_+;E)$ this 
perturbed moduli space.
\par
We next describe a partial compactification of $\overset{\circ\circ}{{\mathcal M}}((M,\mathcal E),L;R_-,R_+;E)$.
Our compactification is a mixture of Uhlenbeck compactification in the 
gauge theory side and of stable map compactification in the (pseudo)holomorphic curve side.
\par
Let $\Omega$ be a bordered nodal curve such that:
\begin{conds}\label{conds27}
\begin{enumerate}
\item
$\Omega$ contains $(0,1] \times \R$ as its irreducible component.
\item
$\overline{\Omega \setminus ((0,1] \times \R)}$ is a finite 
union of trees of sphere components attached to $(0,1) \times \R$
and a finite 
union of trees of disk components attached to $\{1\} \times \R$.
\end{enumerate}
\end{conds}
For each tree of disk or sphere components, its root is by definition its intersection with 
$(0,1] \times \R$.
Note disk components in Condition \ref{conds27} (2) may contain 
a tree of sphere components attached to it.
\par\newpage
\begin{center}
\includegraphics[scale=0.25]
{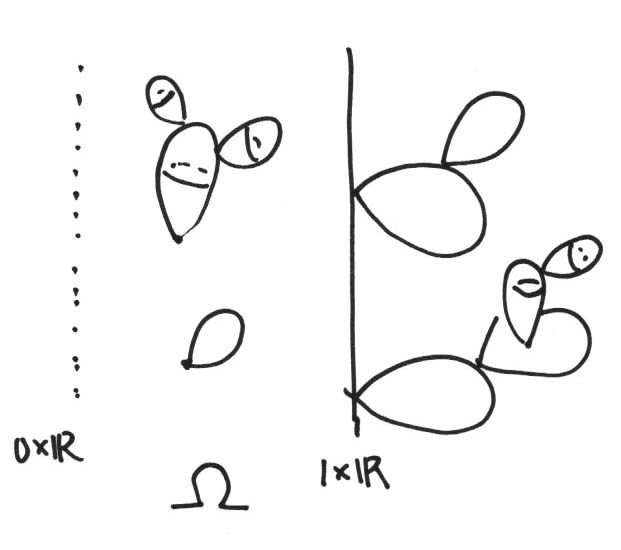}
\end{center}\centerline{\bf Figure 2.3}
\par\medskip
We consider a pair $(\Omega,u)$ where $\Omega$ satisfies 
Condition \ref{conds27} and $u$ is a map 
$: \Omega \to R(\Sigma)$ which satisfies the next condition.
\begin{conds}\label{conds28}
\begin{enumerate}
\item
There exists a continuous map
$\gamma : \partial \Omega \setminus \{\text{boundary nodes}\} 
\to L$ such that 
$u(z) = (i_{L} \circ \gamma)(z)$.
\item
$u$ is holomorphic on each of the irreducible components of $\Omega$ and is continuous
on $\Omega$.
\item
$(\Omega,u)$ is stable. Namely the set of all maps $v : \Omega \to \Omega$ 
satisfying the next three conditions is a finite set.
\begin{enumerate}
\item
$v$ is a homeomorphism and is holomorphic on each of the irreducible components.
\item
$v$ is the identity map on $(0,1] \times \R \subseteq \Omega$.
\item $u \circ v = u$.
\end{enumerate}
\end{enumerate}
\end{conds}
\begin{defn}\label{defn2929}
We define the set $\overset{\circ}{\widetilde{\mathcal M}}((M,\mathcal E_M),L;R_-,R_+;E)$
as the set of all equivalence classes of $(\frak A,{
\frak z},{\frak w},\Omega,u)$ satisfying the following conditions.
\begin{enumerate}
\item
$\frak A$ is a connection of $\mathcal E_{M} \times \R$ satisfying equations  
 (\ref{ASDeq}), (\ref{ASDprod}).
\item
$\frak z = (\frak z_1,\dots,\frak z_{m_1})$ is an {\it unordered} 
$m_1$-tuple of points of $M \setminus (\Sigma\times [0,1]) \times \R$.
We put $\Vert \frak z\Vert = m_1$.
We say the subset $\{\frak z_1,\dots,\frak z_{m_1}\}
\subset M \setminus (\Sigma\times [0,1]) \times \R$ the {\it support} of $\frak z$ 
and denote it by $\vert\frak z\vert$.
We define ${\rm multi} : \vert\frak z\vert \to \Z_{>0}$ by 
$
{\rm multi}(x) = \#\{i \mid z_i = x\}$
and call it the {\it multiplicity function}.
\item
$\frak w = (\frak w_1,\dots,\frak w_{m_2})$ is an {\it unordered} 
$m_2$-tuple of points of $\{1\} \times  \R$.
We put $\Vert \frak w\Vert = m_2$.
We say the subset $\{\frak w_1,\dots,\frak w_{m_2}\} \subset \{1\} \times  \R$ the {\it support} of $\frak w$.
We define ${\rm multi} : \vert\frak w\vert \to \Z_{>0}$ by 
$
{\rm multi}(x) = \#\{i \mid w_i = x\}$
and call it the {\it multiplicity function}.
\item
$\Omega$ satisfies Condition \ref{conds27}.
\item
$(\Omega,u)$ satisfies Condition \ref{conds28}.
\item
For $(s,t) \in (0,1] \times \R \subseteq \Omega$
we have
$$
[A(s,t)] = u(s,t). 
$$
Here $A(s,t)$ is obtained from $\frak A$ by  (\ref{AnadfraA}).
\item
The energy of $(\frak A,{
\frak z},{\frak w},\Omega,u)$
which we will define in Definition \ref{defnenergy2} below is $E$.
\item
Definition \ref{defn2626} (3) holds.
\end{enumerate}
The equivalence relation is defined in Definition \ref{equivlimiobj} below.
\end{defn}
\par
\begin{center}
\includegraphics[scale=0.25]
{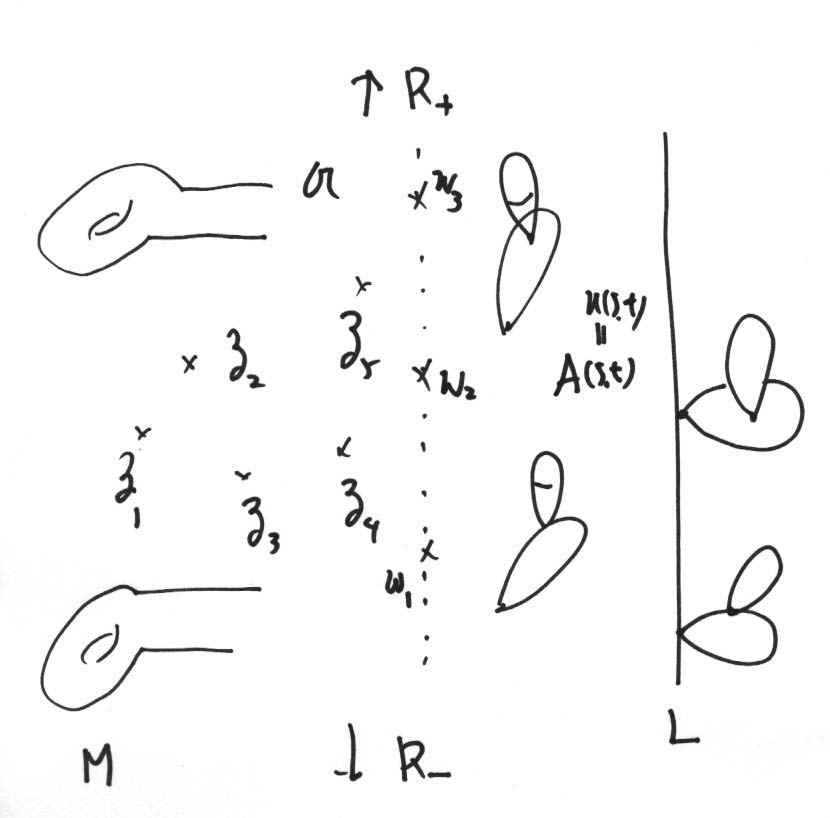}
\end{center}\centerline{\bf Figure 2.4}
\par
\begin{defn}\label{defnenergy2}
Suppose $(\frak A,{
\frak z},{\frak w},\Omega,u)$
satisfies (1)-(6) of Definition \ref{defn2929}. We define its 
{\em energy} $\mathcal E(\frak A,{
\frak z},{\frak w},\Omega,u)$ by the next formula:
\begin{equation}\label{energy2}
\aligned
\mathcal E(\frak A,{
\frak z},{\frak w},\Omega,u) 
=&
\int_{(M\setminus (\Sigma \times [0,1))) \times \R}
\Vert F_{\frak A}\Vert^2 \Omega_{\bf g}\\
&+
 \int_{\Sigma} 
\Vert\nabla u\Vert^2
 dsdt
 + 2\pi^2\Vert \frak z\Vert  + 2\pi^2\Vert \frak w\Vert .
\endaligned
\end{equation}
\end{defn}
\begin{defn}\label{equivlimiobj}
We say $(\frak A_1,{
\frak z}_1,{\frak w}_2,\Omega_1,u_1)$
is {\it equivalent} to 
$(\frak A_2,{
\frak z}_2,{\frak w}_2,\Omega_2,u_2)$
if the following holds.
\begin{enumerate}
\item
There exists a gauge transformation $g$ such that $g^*\frak A_1
= \frak A_2$.
\item
$\frak z_1 = \frak z_2$. $\frak w_1 = \frak w_2$.
\item
There exists a map $v : \Omega_1 \to \Omega_2$ 
such that:
\begin{enumerate}
\item
$v$ is a homeomorphism and is holomorphic on each of the irreducible components.
\item
$v$ is the identity map on $(0,1] \times \R \subseteq \Omega$.
\item $u_2 \circ v = u_1$.
\end{enumerate}
\end{enumerate}
\end{defn}
We define a topology on 
$\overset{\circ}{\widetilde{\mathcal M}}((M,\mathcal E),L;R_-,R_+;E)$
by combining the topology of Uhlenbeck compactification and 
stable map topology as follows.
\begin{defn}\label{defn1222}
Let 
 $[\frak A_n,{
\frak z}_n,{\frak w}_n,\Omega_n,u_n]$ be a sequence of elements of 
the set $\overset{\circ}{\widetilde{\mathcal M}}((M,\mathcal E),L;R_-,R_+;E)$
and 
$[\frak A,{
\frak z},{\frak w},\Omega,u] \in \overset{\circ}{\widetilde{\mathcal M}}((M,\mathcal E),L;R_-,R_+;E)$.
We say  $[\frak A_n,{
\frak z}_n,{\frak w}_n,\Omega_n,u_n]$
converges to $[\frak A,{
\frak z},{\frak w},\Omega,u] $ if the following holds.
\begin{enumerate}
\item
Let  $\vert\frak z\vert \subset (M \times \R) \setminus (\Sigma \times [0,1] \times \R)$
be the support of $\frak z$.
We require that there exists a sequence of gauge transformations $g_n$ such that 
$g_n^*\frak A_n$ converges to $\frak A$ in compact $C^{\infty}$
topology on 
$(M \times \R) \setminus (\Sigma \times [0,1] \times \R) \setminus \vert\frak z\vert$.
\item
For $\epsilon >0$ we denote by $\Omega_n(\epsilon) \subset \Omega_n$ the domain 
$\Sigma \times [\epsilon,1] \times \R$ together with all the 
trees of sphere and disc components of $\Omega_n$ whose 
roots are in $\Sigma \times [\epsilon,1] \times \R$.
We define $\Omega(\epsilon) \subset \Omega$ in the same way.
\par
Then, for any $\epsilon$ such that the root of sphere components 
of $(\Omega,u)$ are not on $\{\epsilon\} \times \R$, we require $(\Omega_n(\epsilon),u_n)$
converges to $(\Omega(\epsilon),u)$ in stable map topology, 
which is defined in the same way as \cite[Definition 10.3]{FO}.
\item
Let  $x \in \vert\frak z\vert$. Then for each sufficiently small  $\epsilon >0$
the next equality holds. Here ${\rm multi}$ is the multiplicity function
and $B_{\epsilon}(x)$ is the metric ball centered at $x$ in $M \times \R$.
\begin{equation}
\aligned
&2\pi^2{\rm multi}(x) + \int_{B_{\epsilon}(x)}\Vert F_{\frak A}\Vert^2 \Omega_{\bf g} \\
&=
\lim_{n\to \infty}
\left(
\sum_{y \in B_{\epsilon}(x) \cap \vert\frak z_n\vert}2\pi^2{\rm multi}(y) 
+
 \int_{B_{\epsilon}(x)}\Vert F_{\frak A_n}\Vert^2 \Omega_{\bf g}
\right).
\endaligned
\end{equation}
\item
Let $x \in \vert\frak w\vert \subset \{0\} \times \R$. 
We define $D_{\epsilon}(x,i)$ $i=1,2,3$ 
and $D_{\epsilon}(x,4;\Omega)$, $D_{\epsilon}(x,4;\Omega_n)$ as follows.
\begin{enumerate}
\item
$D_{\epsilon}(x,1) =  \Sigma \times (D_{\epsilon}(x) \cap ([-1,0) \times \R))$.
Here $D_{\epsilon}(x)$ is the metric ball centered at $x$ in $[-1,1] \times \R$.
\item
$D_{\epsilon}(x,2) = (\{0\} \times \R) \cap D_{\epsilon}(x)$.
\item
$D_{\epsilon}(x,3) = ((0,1] \times \R) \cap D_{\epsilon}(x)$.
\item
$D_{\epsilon}(x,4;\Omega)$ is a subset of $\Omega$ and  is the union of $D_{\epsilon}(x,3)$ and the 
trees of sphere components are rooted on $D_{\epsilon}(x,3)$.
The definition of $D_{\epsilon}(x,4;\Omega_n)$ is similar.
\end{enumerate}
We then require the next equality for sufficiently small positive numbers $\epsilon$.
\begin{equation}
\aligned
&2\pi^2{\rm multi}(x) + \int_{D_{\epsilon}(x;1)}\Vert F_{\frak A}\Vert^2 \Omega_{\bf g} 
+ 2 \int_{D_{\epsilon}(x,4;\Omega)} 
\Vert\nabla u\Vert^2
 dsdt
\\
&=
\lim_{n\to \infty}
\left(
\sum_{y \in D_{\epsilon}(x,1) \cap \vert\frak z_n\vert} 2\pi^2{\rm multi}(y) 
+
 \int_{D_{\epsilon}(x,1)}\Vert F_{\frak A_n}\Vert^2 \Omega_{\bf g}
\right. 
\\
&\qquad\qquad 
+\left.
 \int_{
D_{\epsilon}(x,4;\Omega_n)} 
\Vert\nabla u\Vert^2
 dsdt
+
\sum_{y \in D_{\epsilon}(x,2) \cap \vert\frak w_n\vert}2\pi^2{\rm multi}(y) 
\right).
\endaligned
\end{equation}
\end{enumerate}
\end{defn}
\par
\begin{center}
\includegraphics[scale=0.25]
{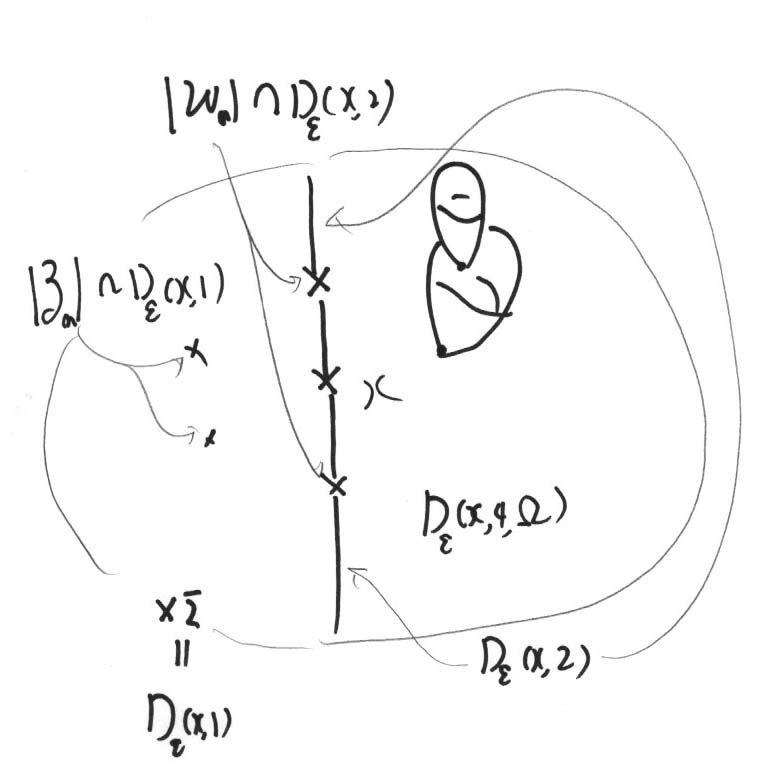}
\end{center}\centerline{\bf Figure 2.5}
\par
\begin{rem}
This topology is a combination of the topology of Uhlenbeck compactification 
defined in \cite[page 292]{Don15} and the stable map topology introduced in \cite[Definition 10.2]{FO}.
\par
The unordered finite set $\frak z$ plays the role to record the total mass of the 
bubble in the gauge theory side. It appeared in Uhlenbeck compactification
in gauge theory.
\par
The unordered set $\frak w$ plays the role to record the total mass of the 
bubble which occurs on the line $\{0\} \times \R$.
The bubble on this line is studied in detail in \cite{fu3}.
Note this is a subset of $\{0\} \times \R$ and is not a subset of 
$\Sigma \times \{0\} \times \R \subset M \times \R$. 
Actually we can not specify where bubble occurs on $\Sigma\times \{0\} \times \R \subset M \times \R$
by the following reason. 
Let $x \in \{0\} \times \R$. The sequence $\frak A_n$ determines a sequence of maps 
$u_n : B_x(\epsilon) \setminus \text{a finite set} \to R(\Sigma)$ which is 
holomorphic. (See \cite[Lemmata 5.5 and 4.22]{fu3}.)
Even in case $u_n$ is defined at $x$ the limit of $u_n$ may not be defined at $x$.
Namely the bubble in the symplectic geometry side may occur on the 
line $\{0\} \times \R$. In such a case the connection $\frak A_n$ 
may diverge everywhere on $\Sigma \times \{0\} \times \{t\}$.
\par
Item (4) also takes into account the following phenomenon.
There may be a sequence of trees of sphere components of $\Omega_n$ whose roots 
are $(s_n,t_n)$ where $s_n > 0$ converges to $0$ and $t_n$ converges to $t$.
Then in the limit this sequence of sphere components moves to $(0,t)$, 
which lies on the line $\{0\} \times \R$.
Therefore the limit is no longer contained in $\Omega$.
In this case, we take $(0,t)$ as a part of $\frak w$ and its multiplicity is the 
limit of the
symplectic area of those trees of the sphere components.
\par
It may also happen that some of the points of $\frak z_n$ converges to 
$\Sigma \times \{0\} \times \R$. In that case it will become one of 
the points of $\frak w$, forgetting the $\Sigma$ factor.
\end{rem}
\begin{defn}
$\overset{\circ}{\widetilde{\mathcal M}}((M,\mathcal E),L;R_-,R_+;E)$
has an $\R$ action by translation along the $\R$ factor.
We denote by 
$\overset{\circ}{{\mathcal M}}((M,\mathcal E),L;R_-,R_+;E)$
the quotient space with quotient topology.
\par
We define a continuous map
\begin{equation}\label{evpm}
{\rm ev}_{\pm} : \overset{\circ}{{\mathcal M}}((M,\mathcal E),L;R_-,R_+;E)
\to R_{\pm}
\end{equation}
by using Definition \ref{defn2929} (8).
\par
We define ${{\mathcal M}}((M,\mathcal E),L;R_-,R_+;E)$
as the disjoint union of fiber products:
\begin{equation}\label{fibercompactka}
\aligned
&\overset{\circ}{{\mathcal M}}((M,\mathcal E),L;R_-,R_1;E_0)
\times_{R_1} 
\overset{\circ}{{\mathcal M}}((M,\mathcal E),L;R_1,R_2;E_1)
\times_{R_2} \dots \\
& \dots  \times_{R_{\ell-1}}\overset{\circ}{{\mathcal M}}((M,\mathcal E),L;R_{\ell-1},R_\ell;E)
\times_{R_{\ell}} {{\mathcal M}}((M,\mathcal E),L;R_{\ell},R_+;E_{\ell})
\endaligned
\end{equation}
where
\begin{equation}
E = E_0 + E_1 + \dots + E_{\ell},
\end{equation}
$R_1,\dots,R_{\ell}$ are connected components of $R(M) \times_{R_{\Sigma}} L$
and $\ell=0,1,2,\dots$.
Using Theorem \ref{asymptotic}, we can define a topology on 
${\mathcal M}((M,\mathcal E),L;R_-,R_+;E)$ in the same way as \cite[Section 7.1.4]{fooobook}.
\end{defn}
\par\begin{center}
\includegraphics[scale=0.25]
{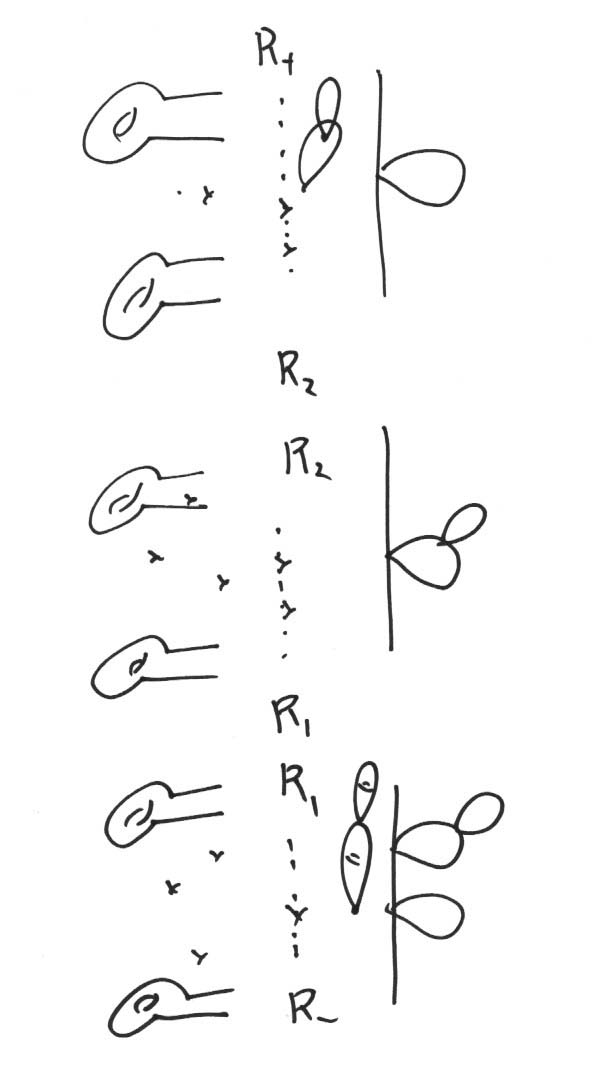}
\end{center}\centerline{\bf Figure 2.6}
\par
\begin{thm}\label{thm215}
${\mathcal M}((M,\mathcal E),L;R_-,R_+;E)$ is compact and Hausdorff.
\end{thm}
The compactness follows from Uhlenbeck compactness in gauge theory side (\cite{uh1}, \cite{uh2}),
Gromov compactness in symplectic geometry side (\cite{Gr}. See \cite{Ye}  for the case of 
moduli space of pseudo-holomorphic disks and \cite[Theorem 11.1]{FO} for the way to adapt the proof to the case when a
particular version of the topology is used  in the pseudo-holomorphic curve side, that is, the stable map 
topology.) and \cite[Theorems 1.6, 1.7 and 1.8]{fu3}.
The Hausdorff-ness can be proved in the same way as the proof of 
\cite[Lemma 10.4]{FO}.
\par
To  describe the boundary of ${\mathcal M}((M,\mathcal E),L;R_-,R_+;E)$ 
we include boundary marked points.
\begin{defn}\label{defn216}
We consider $(\frak A,{
\frak z},{\frak w},\Omega,u,\vec z)$
such that:
\begin{enumerate}
\item
$\frak A$ satisfies Definition 
\ref{defn2929} (1).
\item
$\frak z$ satisfies Definition 
\ref{defn2929} (2).
\item
$\frak w$ satisfies Definition 
\ref{defn2929} (3).
\item
$\vec z = (z_1,\dots,z_{k})$.
$z_i$ lies on $\partial \Sigma$. 
Namely it lies either on $\{1\}\times \R$ or on the boundary of 
one of the disk components. 
None of $z_i$ is a nodal point and
$z_i \ne z_j$ for $i\ne j$.
$(z_1,\dots,z_{k})$ respects counter clockwise orientation of 
$\partial \Omega$.
\item
$\Omega$ satisfies Condition \ref{conds27}.
\item
$(\Omega,u)$ satisfies Condition \ref{conds28} except (3), which we replace
by the stability of 
$(\Omega,u,\vec z)$. Namely the set of all maps $v : \Omega \to \Omega$ 
satisfying the next three conditions is a finite set.
\begin{enumerate}
\item
$v$ is a homeomorphism and is holomorphic on each of the irreducible components.
\item
$v$ is the identity map on $(0,1] \times \R \subseteq \Omega$.
\item $u \circ v = u$.
\item $v(z_i) = z_i$, $i=1,\dots,k$. 
\end{enumerate}
\item Definition \ref{defn2929} (6)(7)(8) hold.
\end{enumerate}
We can define equivalence among such objects by modifying 
Definition \ref{equivlimiobj} in an obvious way.
\par
Let $\overset{\circ}{\widetilde{\mathcal M}}_k((M,\mathcal E),L;R_-,R_+;E)$
be the set of equivalence classes of such objects.
We can define a topology on it by modifying Definition \ref{defn1222}
in an obvious way.
\par
Let 
$\overset{\circ}{{\mathcal M}}_k((M,\mathcal E),L;R_-,R_+;E)$
be its quotient space by $\R$ action.
\par
Replacing (\ref{fibercompactka})
by
\begin{equation}\label{fibercompactka2}
\aligned
&\overset{\circ}{{\mathcal M}}_{k_0}((M,\mathcal E),L;R_-,R_1;E_1)
\times_{R_1} 
\overset{\circ}{{\mathcal M}}_{k_1}((M,\mathcal E),L;R_1,R_2;E_2)
\times_{R_2} \dots \\
& \dots  \times_{R_{\ell-1}}\overset{\circ}{{\mathcal M}}_{k_{\ell-1}}((M,\mathcal E),L;R_{\ell-1},R_{\ell};E_{\ell-1}) \\
&\times_{R_{\ell}} {{\mathcal M}}_{k_{\ell}}((M,\mathcal E),L;R_{\ell},R_+;E_{\ell}),
\endaligned
\end{equation}
where $k_0 + k_1 + \dots +k_{\ell} = k$, 
$E_0+\dots + E_{\ell} = E$, 
and $R_1,\dots,R_{\ell}$ are connected components of 
$R(M)\times_{R(\Sigma)} \tilde L$, 
we obtain ${{\mathcal M}}_k((M,\mathcal E),L;R_-,R_+;E)$.
\par
The space
 ${{\mathcal M}}_k((M,\mathcal E),L;R_-,R_+;E)$ is compact and Hausdorff.
\par
There exists an evaluation map
\begin{equation}
{\rm ev} = ({\rm ev}_1,\dots,{\rm ev}_k) :
{{\mathcal M}}_k((M,\mathcal E),L;R_-,R_+;E) 
\to L^k
\end{equation}
other than ${\rm ev}_-$ and ${\rm ev}_+$.
(See (\ref{evpm}) for ${\rm ev}_{\pm}$.)
If 
$(\frak A,{
\frak z},{\frak w},\Omega,u,\vec z)$
is an element of $\overset{\circ}{\widetilde{\mathcal M}}_k((M,\mathcal E),L;R_-,R_+;E)$
then, by definition  
\begin{equation}
{\rm ev}_i ([\frak A,{\frak z},{\frak w},\Omega,u,\vec z]) 
= u(z_i).
\end{equation}
\end{defn}
We use the compactified moduli space of pseudo-holomorphic disks in 
the next theorem.
Let $\beta \in H_2(R(\Sigma),L;\Z)$.
We denote by
$\mathcal M_{k+1}(L;\beta)$
the compactified moduli space of 
pseudo-holomorphic disks bounding $L$ with $k+1$ boundary marked points
and of homology class $\beta$. (See \cite[Definition 2.1.27]{fooobook}  for its precise definition.)
We have an evaluation map
\begin{equation}
{\rm ev} =({\rm ev}_0,\dots,{\rm ev}_k) : \mathcal M_{k+1}(L;\beta)
\to L^{k+1}.
\end{equation}
We put
\begin{equation}\label{diskmoduli}
\mathcal M_{k+1}(L;E)
= \bigcup_{\beta; \omega(\beta) = E}\mathcal M_{k+1}(L;\beta).
\end{equation}
\begin{thm}\label{thm217}
The space ${\mathcal M}_k((M,\mathcal E),L;R_-,R_+;E)$ has a virtual fundamental chain 
$[{\mathcal M}_k((M,\mathcal E),L;R_-,R_+;E)]$
such that its boundary 
$
\partial [{\mathcal M}_k((M,\mathcal E),L;R_-,R_+;E)]
$
is a sum of the virtual fundamental chains of the following two  
types of spaces.
\begin{enumerate}
\item
The fiber product
$$
{{\mathcal M}}_{k_1}((M,\mathcal E),L;R_-,R;E_1)
\times_{R} 
{{\mathcal M}}_{k_2}((M,\mathcal E),L;R,R_+;E_2),
$$
where $E_1 + E_2 = E$ and $k_1 + k_2 = k$.
We use ${\rm ev}_+$ for the first factor and 
${\rm ev}_-$ for the second factor to define the above 
fiber product.
\item
The fiber product:
$$
{{\mathcal M}}_{k_1}((M,\mathcal E),L;R_-,R;E_1)
{}_{{\rm ev}_i}\times_{{\rm ev}_0} 
{{\mathcal M}}_{k_2}(L,E_2),
$$
\end{enumerate}
where $E_1 + E_2 = E$, $k_1 + k_2 = k+1$, $i=1,\dots,k_1$.
The fiber product is taken over $L$.
\end{thm}
\par
\begin{center}
\includegraphics[scale=0.25]
{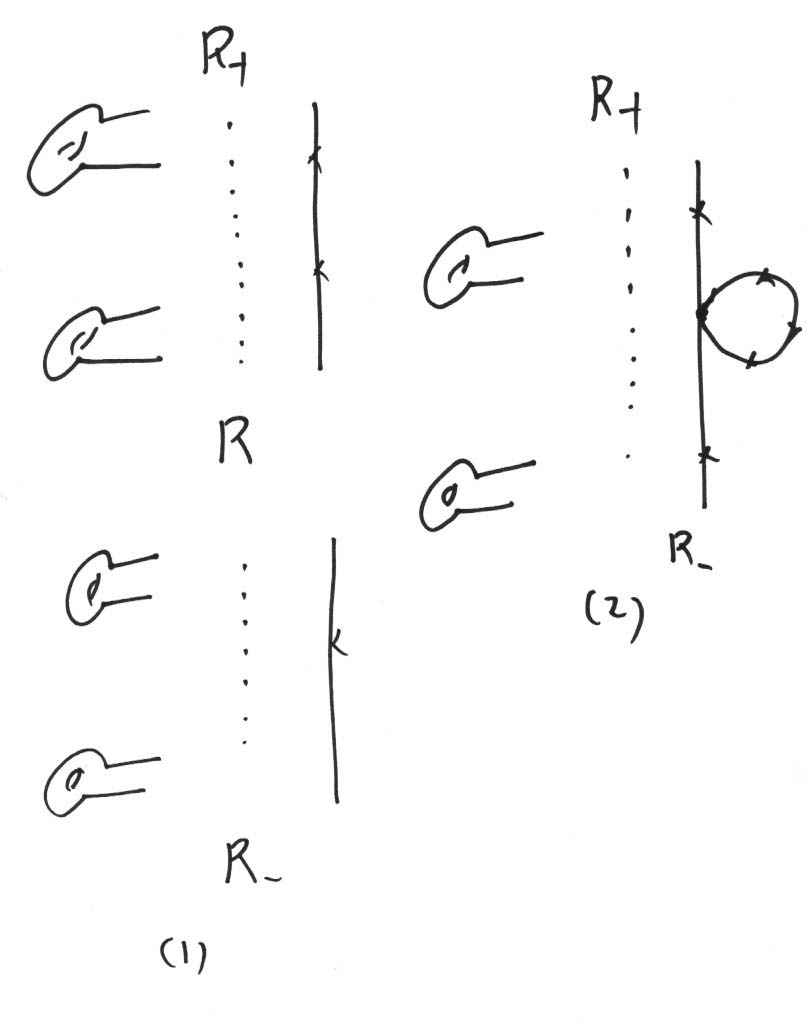}
\end{center}\centerline{\bf Figure 2.7}
\par
\begin{rem}\label{rem2188}
We do {\it not} claim that 
${\mathcal M}_k((M,\mathcal E),L;R_-,R_+;E)$ has a Kuranishi structure 
in the sense of \cite{FO} or \cite{fooobook}.
This is because our moduli space is a mixture of gauge theory and of pseudo-holomorphic 
curve. It is well known among the specialists that Uhlenbeck compactification 
of the moduli space of ASD connections does {\it not} carry a Kuranishi structure.
In Donaldson theory, people, especially Donaldson, used the fact that Uhlenbeck compactification 
has a stratification for which the top stratum has dimension higher by 2 or more 
compared to the other strata, in order to define its fundamental class. 
This fundamental class is nothing but the Donaldson invariant (\cite{Don1}).
\par
In our situation, gauge theory is mixed up with symplectic geometry 
(pseudo-holomorphic curve). 
It seems to the author that the best way to work out transversality issue 
is to use virtual technique, eg. Kuranishi structure.
\par
So we need to work out some generalization of the notion of 
Kuranishi structure which has certain `singularity' at the set of codimension 
equal to or greater than 2. We will provide the detail of the framework 
for such generalization in \cite{fu8} (or in certain separate paper).
\par
We also remark that in our situation where $R(\Sigma)$ is monotone, 
we can work over $\Z_2$ coefficient. (See \cite{fooo:overZ}.)
On the other hand, the author has no doubt that one can work out the whole story 
over $\Z$ by carefully studying orientation and sign.
\par
We also remark that, if we restrict ourselves to the proof of Corollary \ref{maincor},
we can avoid using virtual technique. The reason is that the situation we need to 
study for such a purpose is monotone. 
Especially the Lagrangian submanifold $R(M)$ is monotone if it is embedded.
We will work out this part of the story in detail in \cite{fu7}.
\end{rem}
We use virtual fundamental chain in Theorem \ref{thm217}, to define 
the operation
\begin{equation}\label{form229}
\frak n_{k,E} : C(R(M) \times_{R(\Sigma)} L;\Z_2)
\otimes C(L;\Z_2)^{k\otimes}
\to C(R(M) \times_{R(\Sigma)} L;\Z_2),
\end{equation}
by
\begin{equation}\label{form230}
\aligned
&\left\langle 
\frak n_{k,E}
(y_-;x_1,\dots,x_k), y_+\right\rangle \\
&=
\#
\big([{\mathcal M}_k((M,\mathcal E),L;R_-,R_+;E)]
{}_{({\rm ev}_-,{\rm ev}_1,\dots,{\rm ev}_k,{\rm ev}_+)}\times  \\
&
\qquad\qquad\qquad\qquad\qquad\qquad\qquad
(y_- \times x_1 \times \dots \times x_k \times y_+)\big).
\endaligned
\end{equation}
Here $\langle \cdot,\cdot\rangle$ 
in the left hand side is the Poincar\'e duality of $R(M) \times_{R(\Sigma)} L$, 
and $\#$ is the parity of the order of the set $\in \Z_2$.
The symbol $C(\cdot)$ denotes an appropriate chain model of the homology group.
Then Theorem \ref{thm217} will imply the next formula:
\begin{equation}\label{form231}
\aligned
0 = & (\partial \circ n_{E,k})(y;x_1,\dots,x_k) 
+ (n_{E,k} \circ \partial)(y;x_1,\dots,x_k) \\
&+
\sum_{} \frak n_{k_2,E_2}(\frak n_{k_1,E_1}(y;x_1,\dots,x_{k_1});x_{k_1+1},\dots,x_{k})
\\
&+ 
\sum_{}\frak n_{k_1}(y;x_1,\dots,\frak m_{k_2}(x_i,\dots,x_{i+k_2-1}),\dots,x_k).
\endaligned
\end{equation}
Here the sum in the second line is taken over  $E_1, E_2, k_1, k_2$ satisfying
$E_1 + E_2 = E$ and $k_1 + k_2 = k$.
The sum in the third line is taken over $E_1, E_2, k_1, k_2, i$
satisfying $E_1 + E_2 = E$ and $k_1 + k_2 = k+1$ and $i=1,\dots,k_1$.
\par
Note the second line of (\ref{form231}) corresponds to 
Theorem \ref{thm217} (1) and the third line of 
 (\ref{form231}) corresponds to 
Theorem \ref{thm217} (2).
Therefore, in case we take de Rham theory as our chain model, the formula
(\ref{form231}) follows from Theorem \ref{thm217} 
together with Stokes' formula (\cite[Theorem 8.11]{foootech2}) and composition formula
(\cite[Theorem 10.20]{foootech2}). 
The way to use Stokes' formula and composition formula
to prove a formula like (\ref{form231})
will be explained in detail in \cite{foootech22}.
\par
However, as we mentioned in Remark \ref{rem2188}, 
our moduli space do not carry a genuine Kuranishi structure 
but has a singularity of codimension 2 or higher.
So it seems not easy to use de Rham model.
We will explain in detail the particular chain model we use for our purpose 
and the way how we use it to prove (\ref{form231})
in detail in \cite{fu8}, or in a separate paper.
\begin{defn}\label{defn21919}
We define 
\begin{equation}\label{eq2232}
\aligned
&\frak n_0 = \partial + \sum_{E\ge 0}T^E\frak n_{0,E}, \\
&\frak n_k = \sum_{E\ge 0}T^E\frak n_{k,E}, \qquad k \ge 1. 
\endaligned
\end{equation}
We thus obtained a system of operations:
\begin{equation}
\frak n_k : C(R(M) \times_{R(\Sigma)} L;\Lambda_0^{\Z_2})
\otimes C(L;\Lambda_0^{\Z_2})^{k\otimes}
\to 
C(R(M) \times_{R(\Sigma)} L;\Lambda_0^{\Z_2}).
\end{equation}
\end{defn}
Note we put
$$
CF(R(M),L) = C(R(M) \times_{R(\Sigma)} L;\Lambda_0^{\Z_2}),
\quad
CF(L) = C(L;\Lambda_0^{\Z_2}).
$$
\begin{thm}\label{220}
The system of operations $\frak n_k$, $k=0,1,2,\dots$
defines a structure of a filtered $A_{\infty}$ right module on 
$CF(R(M),L)$ over the 
filtered $A_{\infty}$ algebra $(CF(L),\{\frak m_k\mid 
k=0,1,2,\dots\})$ 
defined in \cite{fooo:overZ}.
Namely the equality (\ref{rightmodule}) holds.
\end{thm}
\begin{rem}
\begin{enumerate}
\item
Because of the problem of `running out'
described in \cite[Section 7.2.3]{fooobook2}
it is actually difficult to construct all the operations $\frak n_{E}$ at
once. So we need to construct $\frak n_{E}$
for $E \le E_0$ and use homological algebra 
to take homotopy limit.
Since this argument is now well established in \cite[Section 7]{fooobook2}
(see also \cite{foootech22}) we do not repeat it here for 
simpicity.
\item
To prove the convergence of the right hand side of (\ref{eq2232}) in $T$-adic topology we need the next Theorem \ref{thm215prime}
which is slightly stronger than Theorem \ref{thm215}.
The proof of Theorem \ref{thm215prime} is the same as that of Theorem \ref{thm215}.
\end{enumerate}
\end{rem}
\begin{thm}\label{thm215prime}
For any $E_0$, the union of the moduli spaces
${\mathcal M}((M,\mathcal E),L;R_-,R_+;E)$ for 
$E\le E_0$ is compact.
\end{thm}
We thus explained the outline of the construction of the filtered 
$A_{\infty}$ right module 
$(CF(R(M),L),\{\frak n_k\})$
over the filtered $A_{\infty}$ algebra
$(CF(L),\{\frak m_k\mid 
k=0,1,2,\dots\})$
in case $L$ is embedded.
\par
We can extend this construction to the 
construction of the filtered $A_{\infty}$ functor
$\mathcal{HF}_{(M,\mathcal E_{M})} : 
\mathscr{FUK}(R(\Sigma)) \to {\mathscr{CH}}$.
Namely we can define a series of operations:
\begin{equation}
\frak n_k : CF(R(M),L_1)
\otimes
\bigotimes_{i=1}^{k-1} CF(L_i,L_{i+1})  
\to
CF(R(M) ,L_k)
\end{equation}
which satisfies the same relation (\ref{rightmodule}).
The proof is similar to the proof of Theorem \ref{220}
using the moduli space of objects shown in the figure below.
See also \cite[Theorem 4.8]{fu2}.
\par
\begin{center}
\includegraphics[scale=0.25]
{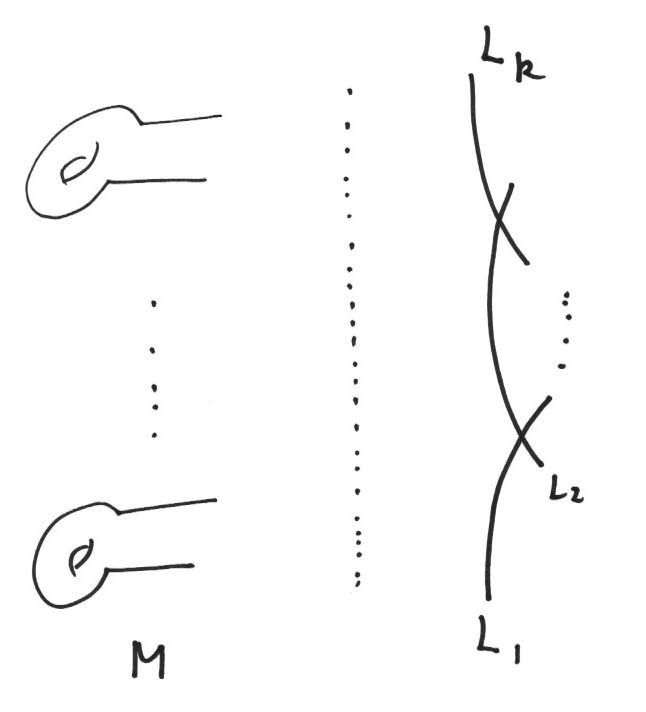}
\end{center}\centerline{\bf Figure 2.8}
\par
We next describe the way how to generalize the construction to the case when $L$ is 
immersed.
Let $i_L : \tilde L \to R(\Sigma)$ be a Lagrangian immersion with image $L$.
We decompose the fiber product $\tilde L \times_{R(M)} \tilde L$
into connected components and write
\begin{equation}\label{form235}
\tilde L \times_{R(\Sigma)} \tilde L = \bigcup_{i\in I(L)} \hat L_i.
\end{equation}
Note in our situation where 
the self-intersection of $L$ is transversal, $\hat L_i$ is either a 
connected component of $L$ or one point consisting 
$(p,q) \in \tilde L^2$ with $p\ne q$, $i_L(p) = i_L(q)$.
In the later case we call $\hat L_i = \{(p,q)\}$ {\it switching component}.
\par
We also decompose 
\begin{equation}\label{form236}
R(M) \times_{R(\Sigma)} \tilde L = \bigcup_{i\in I(R(M),L)} R_i.
\end{equation}
Now we modify Definition \ref{defn216}
 as follows.
Definition \ref{defn23} below is mostly the same 
as Definition \ref{defn216}.
The difference lies in items (6), (7) and (11).
\begin{defn}\label{defn23}
Let $R_+$ and $R_-$ are one of the connected components $R_i$ ($i \in I(R(M),L)$).
Let $\vec i = (i(1),\dots,i(k))$
where $i(1),\dots,i(k) \in I(L)$.
Here $I(L)$ (resp. $I(R(M),L)$) is as in (\ref{form235})
(resp. (\ref{form236})).
We put $\vert \vec i \vert = k$.
We consider $(\frak A,{
\frak z},{\frak w},\Omega,u,\vec z)$
such that:
\begin{enumerate}
\item
$\frak A$ satisfies Definition 
\ref{defn2929} (1).
\item
$\frak z$ satisfies Definition 
\ref{defn2929} (2).
\item
$\frak w$ satisfies Definition 
\ref{defn2929} (3).
\item
$\vec z = (z_1,\dots,z_{k})$.
$z_i$ lies on $\partial \Sigma$.
Namely it lies either on $\{1\}\times \R$ or on the boundary of 
one of the disk components. 
None of $z_i$ is a nodal point and $z_i \ne z_j$ if $i\ne j$. 
$(z_1,\dots,z_{k})$ respects counter clockwise orientation of 
$\partial \Omega$.
\item
$\Omega$ satisfies Condition \ref{conds27}.
\item
There exists $\gamma : \partial \Sigma \setminus \{z_1,\dots,z_k\}
\to 
\tilde L$
such that $u(z) = i_L(\gamma(z))$ on $\partial \Sigma \setminus \{z_1,\dots,z_k\}$.
\item
For $j=1,\dots,k$ the following holds. 
\begin{equation}\label{form23777}
(\lim_{z \uparrow z_j}\gamma(z),\lim_{z\downarrow z_j}\gamma(z))
\in R_{i(j)}.
\end{equation}
Here $z \uparrow z_j$ is the limit when $z \in \partial \Omega$ 
converges to $z_j$ while moving to the counter clockwise direction.
$z \downarrow z_j$ is the limit when $z \in \partial \Omega$ 
converges to $z_j$ while moving to the clockwise direction.
\par
If $\hat L_{i(j)}$ is not a switching component then 
(\ref{form23777}) means that $\gamma$ extends to a continuous map at $z_j$.
If $\hat L_{i(j)}$ is a switching component consisting of the point $(p,q)$, then 
(\ref{form23777}) means $\lim_{z \uparrow z_j}\gamma(z) = p$,
$\lim_{z\downarrow z_j}\gamma(z) = q$.
\item
$(\Omega,u)$ satisfies Condition \ref{conds28} (2).
\item We replace Condition \ref{conds28} (3)
by the stability of 
$(\Omega,u,\vec z)$. Namely the set of all maps $v : \Omega \to \Omega$ 
satisfying the next four conditions is a finite set.
\begin{enumerate}
\item
$v$ is a homeomorphism and is holomorphic on each of the irreducible components.
\item
$v$ is the identity map on $(0,1] \times \R \subseteq \Omega$.
\item $u \circ v = u$.
\item $v(z_i) = z_i$, $i=1,\dots,k$. 
\end{enumerate}
\item Definition \ref{defn2929} (6)(7) hold.
\item 
Let $a_{-}$ and $a_+$ be as in the conclusion of Theorem \ref{asymptotic}.
Then
\begin{equation}
([a_-],\lim_{z \downarrow -\infty}\gamma(z)) \in R_-,
\qquad 
([a_+],\lim_{z \uparrow + \infty}\gamma(z)) \in R_+.
\end{equation}

\end{enumerate}
\par
\begin{center}
\includegraphics[scale=0.25]
{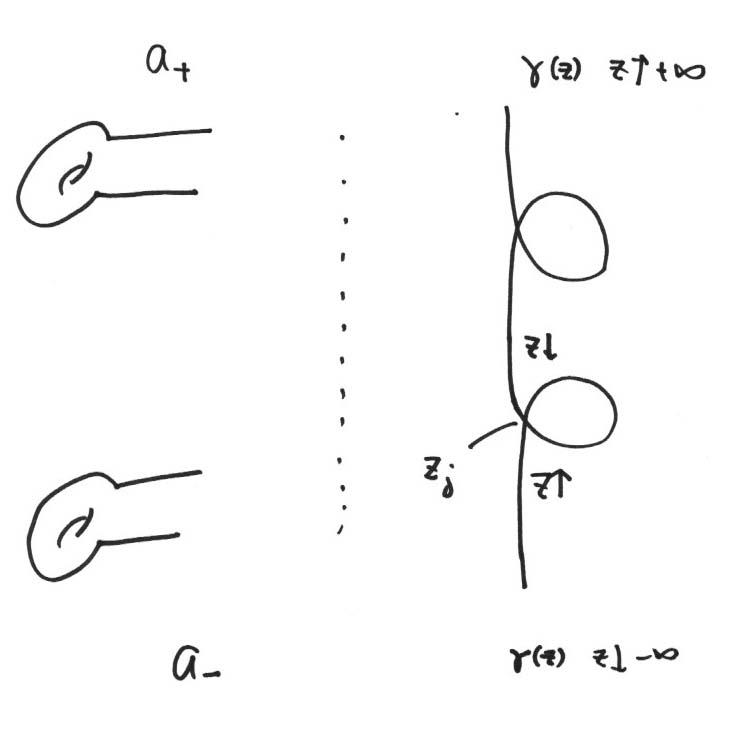}
\end{center}\centerline{\bf Figure 2.9}
\par\medskip
We can define equivalence among such objects by modifying 
Definition \ref{equivlimiobj} in an obvious way.
\par
Let $\overset{\circ}{\widetilde{\mathcal M}}((M,\mathcal E),L;R_-,R_+;\vec i;E)$
be the set of equivalence classes of such objects.
We can define a topology on it by modifying Definition \ref{defn1222}
in an obvious way.
\par
We put
$$
\overset{\circ}{\widetilde{\mathcal M}}_{k}((M,\mathcal E),L;R_-,R_+;E)
=
\bigcup_{\vec i; \vert \vec i \vert = k}\overset{\circ}{\widetilde{\mathcal M}}((M,\mathcal E),L;R_-,R_+;\vec i;E)
$$
\par
Let 
$\overset{\circ}{{\mathcal M}}_k((M,\mathcal E),L;R_-,R_+;E)$
be its quotient space by $\R$ action.
\par
By taking the union (\ref{fibercompactka2}) 
we obtain ${{\mathcal M}}_k((M,\mathcal E),L;R_-,R_+;E)$.
\par
 ${{\mathcal M}}_k((M,\mathcal E),L;R_-,R_+;E)$ is a compact Hausdorff space.
\par
There exists an evaluation map
\begin{equation}
{\rm ev} = ({\rm ev}_1,\dots,{\rm ev}_k) :
{{\mathcal M}}_k((M,\mathcal E),L;R_-,R_+;E) 
\to (\tilde L \times_{R(M)} \tilde L)^k.
\end{equation}
If 
$(\frak A,{
\frak z},{\frak w},\Omega,u,\vec z)$
is an element of $\overset{\circ}{\widetilde{\mathcal M}}((M,\mathcal E),L;R_-,R_+;\vec i;E)$
then, by definition  
\begin{equation}
{\rm ev}_j ([\frak A,{\frak z},{\frak w},\Omega,u,\vec z]) 
=
(\lim_{z \uparrow z_j}\gamma(z),\lim_{z\downarrow z_j}\gamma(z))
\in R_{i(j)}.
\end{equation}
We also define the evaluation maps
\begin{equation}
\aligned
&{\rm ev}_{-}(\frak A,{
\frak z},{\frak w},\Omega,u,\vec z) = ([a_-],\lim_{z \downarrow -\infty}\gamma(z)) \in R_-, \\
&
{\rm ev}_{+}(\frak A,{\frak z},{\frak w},\Omega,u,\vec z) =([a_+],\lim_{z \uparrow + \infty}\gamma(z)) \in R_+.
\endaligned
\end{equation}
\end{defn}
Now in the same way as (\ref{form230}) 
we obtain
\begin{equation}\label{formfffff}
\frak n_{k,E} : C(R(M) \times_{R(\Sigma)} L;\Z_2)
\otimes C(\tilde L \times_{R(M)} \tilde L;\Z_2)^{k\otimes}
\to C(R(M) \times_{R(\Sigma)} \tilde L;\Z_2).
\end{equation}
Then we apply Definition \ref{defn21919} to obtain 
\begin{equation}
\frak n_k : CF(R(M),L)
\otimes CF(L)^{k\otimes}
\to 
CF(R(M),L).
\end{equation}
Theorem \ref{220} still holds in our immersed case.
\par
We thus described the construction of right filtered $A_{\infty}$
module $(CF(R(M),L),\{\frak n_k\})$.
It is straight forward to generalize this construction to the 
construction of filtered $A_{\infty}$ functor
$\mathcal{HF}_{(M,\mathcal E_{M})} : 
\mathscr{FUK}(R(\Sigma)) \to {\mathscr{CH}}$.
Namely we can construct the operations
\begin{equation}
\frak n_k : 
CF(R(M),L_0)
\otimes \bigotimes_{i=1}^{k}CF(L_{i-1},L_{i})
\to 
CF(R(M),L_k)
\end{equation}
which satisfies the same relation (\ref{rightmodule}).
In fact we may regard the union $L_0 \cup \dots \cup L_{k}$
as a single immersed Lagrangian submanifold and can
apply Definition \ref{defn23} etc.
\par
We have thus completed the sketch of the proof of Theorem \ref{functordef}.
\begin{rem}
In case $L$ is an embedded and monotone Lagrangian submanifold of $R(\Sigma)$
the construction of $HF((M,\mathcal E);L)$ is carried out in detail 
by \cite{Sawe}.
More precisely they did the case when the restriciton of $\mathcal E$ to $\Sigma$ is a 
trivial $SU(2)$ bundle. This case is harder than the case we study here. 
\end{rem}

\section{Floer theoretical unobstructed-ness of the moduli space of flat connections on 3-manifolds
with boundary}\label{sec:unbolt}

In this section we prove Theorem \ref{mainthm} (1).
The main idea of its proof is an algebraic result, Proposition \ref{thm35} below.
To state it we introduce certain notions.

\begin{defn}{\rm (Compare \cite[Condition 3.1.6]{fooobook})}
A discrete submonoid $G$ of $\R_{\ge 0}$ is a discrete subset 
$G \subset \R_{\ge 0}$ containing $0$ such that $\lambda_1, \lambda_2 \in G$ 
implies $\lambda_1 + \lambda_2 \in G$.
Hereafter we say {\it discrete monoid} in place of discrete submonoid of $\R_{\ge 0}$
for simplicity.
\end{defn}
\begin{defn}{\rm (\cite[Definition 3.2.26 etc.]{fooobook})}
Let $\overline C$, $\overline C_i$ ($i=1,2$) be $\Z_2$ vector spaces and 
$C$ (resp. $C_i$)  the completion of $\overline C \otimes \Lambda_{0}^{\Z_2}$
(resp. of $\overline C_i \otimes \Lambda_{0}^{\Z_2}$) with respect to the $T$-adic topology.
Let $G$ be a discrete monoid.
\begin{enumerate}
\item An element $x$ of $C$ is said to be $G$-{\it gapped} if 
there exists $x_{\lambda} \in \overline C$ for each $\lambda \in G$ such that
$$
x = \sum_{\lambda \in G}T^{\lambda}x_{\lambda}.
$$
\item
A $\Lambda_0^{\Z_2}$ module homomorpism $\varphi : C_1 \to C_2$
is said to be $G$-{\it gapped} if there exists a $\Z_2$ linear maps
$\varphi_{\lambda} : \overline C_1 \to \overline C_2$ such that
$$
\varphi = \sum_{\lambda \in G}T^{\lambda}\varphi_{\lambda}.
$$
\item
A filtered $A_{\infty}$ algebra (resp. a filtered $A_{\infty}$ module) 
is said to be $G$-{\it gapped} if its structure maps $\frak m_k$ 
(resp. $\frak n_k$) are all $G$-gapped.
\end{enumerate}
\end{defn}
\begin{defn}\label{defn33}
Let $(C,\{\frak m_k\})$ be a $G$-gapped filtered $A_{\infty}$ algebra and 
$(D,\{\frak n_k\})$ be a $G$-gapped right filtered $A_{\infty}$ module over $(C,\{\frak m_k\})$.
\par
We say ${\bf 1} \in D$ is a {\it cyclic element}\footnote{The word 
cyclic element seems to be a standard one for an object satisfying a 
condition such as (1). We remark that the notion of cyclic element has no 
relation to the cyclic symmetry of the filtered $A_{\infty}$ algebra associated 
to a Lagrangian submanifold.} if the following holds
\begin{enumerate}
\item
The map 
$C \to D$ which sends $x$ to $\frak n_2({\bf 1};x)$ is an 
$\Lambda_{0}^{\Z_2}$ module isomorphism 
$C \to D$.
\item
$\frak n_0({\bf 1}) \equiv 0 \mod \Lambda_+^{\Z_2}$.
\end{enumerate}
\end{defn}
We also recall:
\begin{defn}
Let $C$ be a filtered $A_{\infty}$ algebra.
An element $b \in C \otimes_{\Lambda_0^{\Z_2}}\Lambda_+^{\Z_2}$ is said to be 
a {\it bounding cochain}  if
\begin{equation}\label{MCeq}
\sum_{k=0}^{\infty}\frak m_k(b,\dots,b) = 0.
\end{equation}
Note the left hand side converges in $T$-adic topology, since $b \equiv 0 \mod \Lambda_+^{\Z_2}$.
\end{defn}
\begin{prop}\label{thm35}
Let $(C,\{\frak m_k\})$ be a $G$-gapped filtered $A_{\infty}$ algebra and 
$(D,\{\frak n_k\})$  a $G$-gapped right filtered $A_{\infty}$ module over $(C,\{\frak m_k\})$.
Suppose ${\bf 1} \in D$ is a  cyclic element, which is $G$-gapped.
\par
Then there exists a unique $G$-gapped bounding cochain $b$ of $(C,\{\frak m_k\})$
such that 
\begin{equation}\label{eq32}
d^b({\bf 1}) = 0.
\end{equation}
\end{prop}
Note we defined $d^b$ by
$$
d^b(y) = \sum_{k=0}^{\infty} \frak n_k(y;b,\dots,b).
$$
\begin{proof}
We first prove the uniqueness.
Let $G = \{\lambda_i \mid i=0,1,2,\dots\}$
where $0 = \lambda_0 < \lambda_1 < \lambda_2 < \dots.$
We put
$$
\aligned
{\bf 1} &=& \sum_{i=0}^{\infty} T^{\lambda_i}{\bf 1}_i,
\qquad
b &=&  \sum_{i=1}^{\infty} T^{\lambda_i}b_i \\
\frak m_k & =& \sum_{i=0}^{\infty} T^{\lambda_i}\frak m_{k,i},
\qquad
\frak n_k  &=& \sum_{i=0}^{\infty} T^{\lambda_i}\frak n_{k,i}.
\endaligned
$$
according to the definition of $G$-gapped-ness.
(Note the coefficient of $T^{\lambda_0}$ ($\lambda_0 = 0$) of $b$ is $0$ since 
$b \in  C \otimes_{\Lambda_0^{\Z_2}}\Lambda_+^{\Z_2}$.)
\par
We calculate the coefficient of $T^{\lambda_n}$
of the equation (\ref{eq32}) and obtain
\begin{equation}\label{cond33}
\frak n_{1,0}({\bf 1}_0;b_{n})
+ 
\sum \frak n_{k,m}({\bf 1}_{n_0};b_{n_1},\dots,b_{n_k}) = 0.
\end{equation}
Here the second term is the sum of all $k$, $m$, $n_0,n_1,\dots,n_k$ such that
\begin{equation}\label{cond34}
\lambda_n = \lambda_m + \lambda_{n_0} + \sum_{i=1}^k \lambda_{n_i}
\end{equation}
except the case $k=1$, $m=0$, $n_0 = 0$, $n_1 = n$.
(The case we exclude here corresponds to the first term.)
Note if $k$, $m$, $n_0,n_1,\dots,n_k$  satisfy
(\ref{cond34}) then $n_i \le n$ for $i=0,\dots,k$.
Moreover $n_i < n$ unless $k=1$, $m=0$, $n_0 = 0$, $n_1 = n$.
\par
Therefore we can solve  (\ref{cond33}) and obtain $b_n$ uniquely 
by induction on $n$.
(Here we use Definition \ref{defn33} (1).)
\par
Thus we proved that there exists a unique $G$-Gapped element
$b \in C \otimes_{\Lambda_0^{\Z_2}} \Lambda_+^{\Z_2}$ satisfying (\ref{eq32}).
\par
It remains to prove that this element $b$ satisfies the 
Maurer-Cartin equation (\ref{MCeq}).
We will prove
\begin{equation}\label{MCmod}
\sum_{k=0}^{\infty}\frak m_k(b,\dots,b) \equiv 0 \mod T^{\lambda_c}
\end{equation}
by induction on $c \in \Z_+$.
We assume (\ref{MCmod}) for $c \le n-1$ and will prove  the case $c = n$ below.
\par
We put $\partial = \frak n_{1,0}$.
(By (\ref{form231}) we have $\partial \circ \partial = 0$.)
Using (\ref{form231}) and Definition \ref{defn33} (2) we have
\begin{equation}
\partial\frak n_{1,0}({\bf 1}_0;x) + \frak n_{1,0}({\bf 1}_0;\partial x) = 0
\end{equation}
for $x \in \overline C$.
\par
We next consider $\partial \frak n_{1,0}({\bf 1}_0;b_{n})$.
Using (\ref{cond33}) we find
$$
\partial (\frak n_{1,0}({\bf 1}_0;b_{n}))
= 
\sum \partial (\frak n_{k,m}({\bf 1}_{n_0};b_{n_1},\dots,b_{n_k})).
$$
We calculate the right hand side using (\ref{form231}) to obtain:
\begin{equation}\label{eq3737}
\aligned
&\sum \frak n_{k_1,m_1}(\frak n_{k_2,m_2}({\bf 1}_{n_0};b_{n_1},\dots,b_{n_{k_2}}),
\dots, b_{n_k}) \\
&+ \sum \frak n_{k_1,m_1}({\bf 1}_{n_0};b_{n_1},\dots,
\frak m_{k_2,m_2}(b_{n_{i+1}},\dots,b_{n_{i+k_2}}),\dots,b_{n_k}) \\
& + \sum \frak n_{k,m}({\bf 1}_{n_0};b_{n_1},\dots,\partial b_{n_j},\dots,b_{n_k}).
\endaligned
\end{equation}
Here the sum in the first line is taken over 
$k_1$, $k_2$, $m_1$, $m_2$, $n_0, \dots, n_k$ such that
$k_1 + k_2 = k$ and
$
\lambda_n =
\lambda_{m_1} + 
 \lambda_{m_2} + \lambda_{n_0} + \sum_{i=1}^k \lambda_{n_i}
$, except $(k_1,m_1) = (0,0)$.
\par
The sum in the second line is taken over 
$k_1$, $k_2$, $m_1$, $m_2$, $n_0, \dots, n_k$ such that
$k_1 + k_2 = k+1$ and
$
\lambda_n =
\lambda_{m_1} + 
 \lambda_{m_2} + \lambda_{n_0} + \sum_{i=1}^k \lambda_{n_i}
$, except $m_2 = 0$, $k_2 =1$.
(The excluded case corresponds to the third line.)
\par
The sum in the third line 
is taken over 
$k$, $m$, $j$, $n_0, \dots, n_k$ such that
$j=1,\dots, k$ and
$\lambda_n = \lambda_m + \lambda_{n_0} + \sum_{i=1}^k \lambda_{n_i}$, 
except $n_0 =0$, $k=1$, $m=0$.
We exclude this case since it is excluded in the second term of 
(\ref{cond33}).
\par
Note the first line of (\ref{eq3737}) vanishes because of the equality 
(\ref{eq32}) and $(k_1,m_1) \ne (0,0)$.
\par
By using induction hypothesis (\ref{MCmod}) for $c \le n-1$, 
the sum of the second and third lines cancel each other
except the sum
$$
\sum \frak n_{0,1}({\bf 1}_{0};
\frak m_{k,m}(b_{n_1},\dots,b_{n_{k}}))
$$
which is taken over $k, m,  n_1,\dots, n_k$
such that 
$\lambda_n = \lambda_m + \sum_{i=1}^k \lambda_{n_i}$.
(In fact this sum could be canceled with 
$\frak n_{0,1}({\bf 1}_{0};
\partial b_{n})$. However this  is the case excluded in the third line.)
\par
Thus we have
$$
\frak n_{1,0}({\bf 1}_0;\partial b_{n})
=
\partial \frak n_{1,0}({\bf 1}_0;b_{n})
=
\sum \frak n_{0,1}({\bf 1}_{0};
\frak m_{k,m}(b_{n_1},\dots,b_{n_{k}})).
$$
Using Definition \ref{defn33} (1) it implies
$$
\partial b_{n} = 
\sum \frak m_{k,m}(b_{n_1},\dots,b_{n_{k}}).
$$
It implies (\ref{MCmod}) for $c=n$.
The proof of Proposition \ref{thm35} is now complete.
\end{proof}
\begin{rem}
Suppose $(C,\{\frak m_k\})$ is a filtered $A_{\infty}$ algebra and has a unit ${\bf e}$.
Then by defining 
$$
\frak n_{k_1,k_2}(x_1,\dots,x_{k_1};y;z_{1},\dots,z_{k_2})
= \frak m_{k_1+k_2+1}(x_1,\dots,x_{k_1},y,z_{1},\dots,z_{k_2})
$$
$C$ can be regarded as a filtered $A_{\infty}$ bimodule over 
$(C,\{\frak m_k\})$-$(C,\{\frak m_k\})$.
(See \cite[Example 3.7.6]{fooobook}.)
Moreover ${\bf e} = {\bf 1}$ satisfies Definition \ref{defn33} (1)(2).
\par
However in case $\frak m_0 \ne 0$ the operations
$$
\frak n_{k}(y;x_{1},\dots,x_{k})
= \frak m_{k+1}(y,x_{1},\dots,x_{k})
$$
do {\it not} define a structure of  filtered $A_{\infty}$ right module on 
$C$ over $(C,\{\frak m_k\})$.
In fact we have
$$
\aligned
&\sum_{\ell=0}^{k}
\frak n_{k-\ell+1}(\frak n_{\ell}(y;x_1,\dots,x_{\ell});x_{\ell+1},\dots,x_k) \\
&+
\sum_{0\le \ell \le m\le k}
\frak n_{k-m+\ell+1}(y;x_1,\dots,\frak m_{m-\ell}(x_{\ell},\dots,x_{m-1}),\dots,x_k) \\
&= \frak n_{1,k}(\frak m_0(1);y,x_1,\dots,x_k).
 \endaligned
$$
Therefore we can {\it not} apply Proposition \ref{thm35} in this situation.
\end{rem}
Now we apply Proposition  \ref{thm35}  to our geometric situation and prove Theorem 
\ref{mainthm} (1).
Let $M$ be a 3 manifold with boundary $\Sigma$ and $\mathcal E_{M}$  an 
$SO(3)$ bundle on $M$ such that the 2nd Stiefel-Whitney class $w_2(\mathcal E_{\Sigma})$
of the restriction $\mathcal E_{\Sigma}$ of $\mathcal E_{M}$ to $\Sigma$ 
is the fundamental class $[\Sigma]$.
Let $i_{R(M)} : R(M) \to R(\Sigma)$ be the Lagrangian immersion, where 
$R(M)$ (resp. $R(\Sigma)$) is the set of gauge equivalence classes of 
flat connections of $\mathcal E_M$ (resp. $\mathcal E_{\Sigma}$).
\par
In the previous section, we associate a right filtered $A_{\infty}$ module 
$CF(M,L)$
over $(CF(L),\{\frak m_k\})$
to a Lagrangian immersion $i_L : \tilde L \to L \subseteq R(\Sigma)$.
These filtered $A_{\infty}$ algebra and filtered $A_{\infty}$ module are 
$G$-gapped for certain discrete monoid $G$ by Gromov-Uhlenbeck compactness and Theorem \ref{thm215prime}. 
\par
We consider its special case where $\tilde L = R(M)$.
In this case the underlying $\Lambda_{0}^{\Z_2}$
module of 
$(CF(M,L),\{\frak n_k\})$
coincides with 
$(CF(R(M),R(M)),\{\frak m_k\})$.
Namely both are
$C(R(M) \times_{R(\Sigma)} R(M);\Lambda_0^{\Z_2})$.
We remark
\begin{equation}
R(M) \times_{R(\Sigma)} R(M)
\cong R(M) \sqcup \{(p,q) \in R(M)^2 \mid p \ne q, i_{R(M)}(p) = i_{R(M)}(q)\}.
\end{equation}
Therefore
$$
C(R(M) \times_{R(\Sigma)} R(M);\Lambda_0^{\Z_2})
=
C(R(M);\Lambda_0^{\Z_2})
\oplus \bigoplus_{(p,q)} \Lambda_0^{\Z_2}[(p,q)]
$$
where the direct sum in the second component is taken over 
$\{(p,q) \in R(M)^2 \mid p \ne q, i(p) = i(q)\}$.
\begin{defn}\label{def37}
We take ${\bf 1}_M \in C(R(M) \times_{R(\Sigma)} R(M);\Lambda_0^{\Z_2})$
to be the fundamental cycle of $R(M)$ in $C(R(M);\Lambda_0^{\Z_2})$.
\end{defn}
Now we observe:
\begin{lem}\label{lem38}
${\bf 1}_M$ is a cyclic element of the right filtered $A_{\infty}$ module
$(C(R(M) \times_{R(\Sigma)} R(M);\Lambda_0^{\Z_2}),\{\frak n_k\})$
over $(C(R(M) \times_{R(\Sigma)} R(M);\Lambda_0^{\Z_2}),\{\frak m_k\})$.
\end{lem}
\begin{proof}
We decompose
$$
\frak n_k = \sum T^{\lambda_i} \frak n_{k,i}
$$
accoding to the definition of $G$-gapped-ness.
Then by definition $\frak n_{k,0}$ is defined by using the moduli spaces
${\mathcal M}_k((M,\mathcal E),L;R_-,R_+;0)$ consisting of the solutions of 
(\ref{ASDeq}), (\ref{ASDprod}) with zero energy.
Such solution is necessary of the form $(\frak A,{
\frak z},{\frak w},\Omega,u)$
where $\frak A$ is a flat connection on $M \times \R$, $u$ is a constant map and 
$\Omega = \Sigma \times (0,1] \times \R$.
It follows that
$$
\frak n_{1,0}({\bf 1}_M;x) = x
$$
for any $x \in C(R(M) \times_{R(\Sigma)} R(M);\Lambda_0^{\Z_2})$.
Definition \ref{defn33} (1) follows immediately.
\par
Definition \ref{defn33} (2) follows from the fact that $\frak n_0$ is congruent 
to the classical boundary operator modulo $\Lambda_{+}^{\Z_2}$.
\end{proof}
Proposition \ref{thm35} and Lemma \ref{lem38} immediately imply the next theorem.
\begin{thm}\label{thm39}
The immersed Lagrangian submanifold $R(M) \to R(\Sigma)$ is unobstructed.
Moreover we can chose the bounding cochain $b_M$ uniquely so that
\begin{equation}\label{1isccle}
d^{b_M}({\bf 1}_M) = 0.
\end{equation}
\end{thm}
Theorem  \ref{mainthm} (1) follows from Theorem \ref{thm39} except the statement 
that the gauge equivalence class of $b_{M}$ is independent of the choices.
Here the choices are perturbation to define filtered $A_{\infty}$ algebra and 
filtered $A_{\infty}$ module involved in the construction and the metric on $M$ 
etc..
One can prove this independence by using a cobordism argument, which have been used 
extensively in various related situations. 
(The one which is closest to the present situation is \cite[Sections 5,6,7]{fu2}.) 
So we can safely omit it in this paper
and postpone its detail to \cite{fu8}.
\par
Theorem \ref{mainthm} (3) follows from the next proposition.
\begin{prop}\label{prop310}
If $R(M)$ is an embedded Lagrangian submanifold then
$$
\frak n_0({\bf 1}_M) = 0.
$$
\end{prop}
\begin{proof}
The proof is based on the  monotonicity and proceed as follows.
We decompose $R(M)$ to the connected components $R(M) =
\bigcup_{i \in I} R_i$.
We first observe
$$
{\mathcal M}((M,\mathcal E),L;R_i,R_j;E) = \emptyset
$$
for $i\ne j$.
In fact the boundary value of an element of ${\mathcal M}((M,\mathcal E),L;R_i,R_j;E)$ 
is a path joining $R_i$ and $R_j$ in $R(M)$.
We next show:
\begin{lem}\label{lem311}
If $E>0$ then
$$
\dim {\mathcal M}((M,\mathcal E),L;R_i,R_i;E)
> \dim {\mathcal M}((M,\mathcal E),L;R_i,R_i;0).
$$
Here $\dim$ means the virtual dimension.
\end{lem}
\begin{proof}
Let $(\frak A,\Omega,u)$ be an representative of an element of 
${\mathcal M}((M,\mathcal E),L;R_i,R_i;E)$.
The boundary value of $u$ defines a $t\in \R$ parametrized family 
$a(1,t)$ of flat connections of $M$.
We may regard it as a connection on $M \times \R$, which we denote by $\frak A_0$.
We remark that $\frak A$ and $\frak A_0$ has the same boundary value on 
$\Sigma \times \{1\} \times \R$ and the same asymptotic as $\R$ coordinate goes to 
$\pm \infty$. So we can define relative Pontryagin number
$$
\int_{M \times \R} (p_1(\frak A) - p_1(\frak A_0)) \in \Z.
$$
We then have
$$
2\pi^2\int_{M \times \R} (p_1(\frak A) - p_1(\frak A_0)) =  E.
$$
Since $E$ is strictly positive 
$\int_{M \times \R} (p_1(\frak A) - p_1(\frak A_0))$ is strictly positive.
Therefore $\frak A$ is homotopic relative to the boundary 
to a connection obtained by gluing $\frak A_0$ with a 
connection on $S^4$ with positive Pontryagin number.
Lemma \ref{lem311} now follows from index sum formula.
\end{proof}
Proposition \ref{prop310} follows from Lemma \ref{lem311} and 
dimension counting.
\end{proof}
Proposition \ref{prop310} implies that $b_M = 0$ if $R(M)$ is an 
embedded Lagrangian submanifold.
This is Statement (3) of Theorem \ref{mainthm}.
\begin{rem}\label{rem312}
We can use Proposition \ref{thm35} to study Wehrheim-Woodward functoriality
(\cite{WW})
in a similar way. Especially we can prove the next theorem.
Let $(M_i,\omega_i)$ $i=1,2$ be a symplectic manifold, $i_{L_1} : \tilde L_1
\to L_1 \subset M_1$ an 
immersed Lagrangian submanifold of $M_1$ and $i_{L_{12}} : \tilde L_{12}
\to L_{12} \subset M_1 \times M_2$ an immersed Lagrangian  submanifold 
of $(M_1 \times M_2, \omega_1 \oplus - \omega_2)$.
We assume $L_1 \times M_2$ is transversal to $L_{12}$ and put
$$
\tilde L_2 = (L_1 \times M_2) \times_{M_1 \times M_2} \tilde L_{12}.
$$
Then the composition of $\tilde L_2 \subset M_1 \times M_2$ 
with the projection $M_1 \times M_2 \to M_2$ 
is a Lagrangian immersion $i_{L_2} : \tilde L_2 \to L_2 \subset M_2$.
We assume $M_1$, $M_2$, $L_1$, $L_{12}$ are  spin and fix a spin structure of them.
It induces a spin structure of $L_2$.

\begin{thm}\label{WWfunc}
In the above situation 
we assume that $L_1$ and $L_{12}$ are unobstructed, in addition.
\par
Then $L_2$ is unobstructed.
Moreover gauge equivalence classes of the bounding cochains $b_1$, $b_{12}$ of $L_{1}$ and $L_{12}$ 
determine a gauge equivalence class of a bounding cochain $b_{2}$ of $L_2$. 
\end{thm}
For the proof we replace Figure 2.9 by the next Figure 3.1.
\par
\begin{center}
\includegraphics[scale=0.25]
{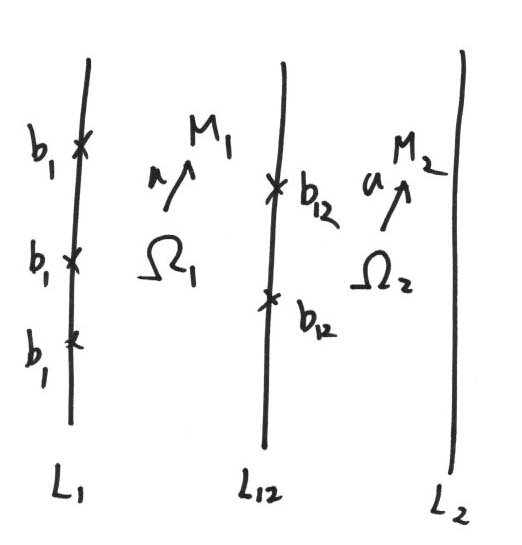}
\end{center}\centerline{\bf Figure 3.1}
\par\medskip
Here $u$ is a combination of maps, to $M_1$ 
(in the domain $\Omega_1$) and $M_2$ (in the domain $\Omega_2$).
$L_1$ and $L_{12}$ are used to define a 
boundary condition on $\partial \Omega_1 \setminus (
\Omega_1\cap \Omega_2)$ and on 
$\Omega_1\cap \Omega_2$, respectively.
We use $b_1$, the bounding cochain of $L_1$ and $b_{12}$, 
the bounding cochain of $L_{12}$ to cancel the (disk)
bubbles on $\partial \Omega_1 \setminus (
\Omega_1\cap \Omega_2)$
and on $\Omega_1\cap \Omega_2$, respectively.
We thus obtain a structure of filtered $A_{\infty}$ right 
module over $CF(L_2)$.
Using the fact that $\tilde L_2 \times_{M_2} \tilde L_2
= \tilde L_1 \times_{M_1} \tilde L_{12} \times_{M_2} \tilde L_2$,
we can show $[\tilde L_2]$ is the cyclic element of this 
filtered $A_{\infty}$ right 
module and can apply Proposition \ref{thm35}.

We will prove Theorem \ref{WWfunc} and explore related topics in 
\cite{fu9}.
See also Remark \ref{rem415}.
\end{rem}

\section{Representativity of the relative Floer homology functor}\label{sec:represent}

In this section we explain a proof of Theorem \ref{represent}.
Let $(M,\mathcal E_M)$ and $(\Sigma,\mathcal E_{\Sigma})$
be as in Situation \ref{situ21}.
In Theorem \ref{thm39} we obtain a bounding cochain $b_M$ of 
the filtered $A_{\infty}$ algebra $(CF(R(M)),\{\frak m_k\})$
satisfying (\ref{1isccle}).
\par
We consider a Lagrangian immersion $i_{L} : \tilde L \to R(\Sigma)$
and its filtered $A_{\infty}$ algebra $(CF(L),\{\frak m_k\})$
where 
$
CF(L) = C(\tilde L \times_{R(\Sigma)} \tilde L;\Lambda_0^{\Z_2}).
$
In Section \ref{HF3bdfunctor} we defined a right filtered $A_{\infty}$
module $(CF(M,L),\{\frak n_k\})$ over $(CF(L),\{\frak m_k\})$
where
$
CF(M,L) = C(R(M) \times_{R(\Sigma)} \tilde L;\Lambda_0^{\Z_2})
$.
\par
We define right filtered $A_{\infty}$ module 
$(CF((R(M),b_M),L),\{{}^{b_M}\frak n_k\})$ over $(CF(L),\{\frak m_k\})$
as follows.
We put
\begin{equation}
{}^{b_M}\frak n_k(y;x_1,\dots,x_k)
=
\sum_{\ell=0}^{\infty}
\frak n_{\ell,k}(\underbrace{b_M,\dots,b_M}_{
\ell};y;x_1,\dots,x_k).
\end{equation}
Here $\{\frak n_{\ell,k}\}$ is the filtered $A_{\infty}$ bimodule structure
of $CF(R(M),L)$ over $CF(R(M))$-$CF(L)$.
We review the construction of this  filtered $A_{\infty}$ bimodule structure
later. See Corollary \ref{cor4949}.

Using the fact that $b_M$ satisfies the Maurer-Cartin equation 
(\ref{MCeq}) we can prove that ${}^{b_M}\frak n_k$ defines the 
structure of right filtered $A_{\infty}$ module.
Namely it satisfies (\ref{rightmodule}).
\par
We recall:
\begin{defn}
Let $(D_i,\{\frak n_{k}^i\})$ ($i=1,2$) be right filtered 
$A_{\infty}$ modules over $(C,\{\frak m_k\})$.
A {\it filtered $A_{\infty}$ homomorphism} 
$: (D_1,\{\frak n_{k}^1\}) \to (D_2,\{\frak n_{k}^2\})$ is 
$\widehat{\varphi} = \{\varphi_{k}\mid k=0,1,2,\dots\}$
where
$$
\varphi_k : D_1 \otimes C^{k\otimes} \to D_2
$$
such that 
\begin{equation}\label{for4242}
\aligned
&\sum \frak n_{k_1}(\varphi_{k_2}(y;x_1,\dots,x_{k_2});\dots,x_k)
\\
= &
\sum \varphi_{k_1}(\frak n_{k_2}(y;x_1,\dots,x_{k_2});\dots,x_k)
\\
&+ \sum \varphi_{k_1}(y;x_1,\dots,\frak m_{k_2}(x_{i},\dots,x_{i+k_2-1}),\dots,x_k).
\endaligned
\end{equation}
Here the sums in the first and the second lines are taken over $k_1,k_2$ with $k_1 +k_2 = k$,
the sum in the third line is taken over $k_1,k_2,i$ such that
$k_1 + k_2 = k+1$ and $i = 1,\dots,k_1$.
\par
We say $\widehat{\varphi}$ is {\it strict} if $\varphi_0 = 0$.
\end{defn}
\par
We will prove:
\begin{thm}\label{them42}
There exists a strict right filtered $A_{\infty}$ module homomorphism
\begin{equation}
\hat\varphi  : (CF((R(M),b_M),L),\{{}^{b_M}\frak n_k\})
\to (CF(M;L),\{\frak n_k\})) 
\end{equation}
over $(CF(L),\{\frak m_k\})$ such that
\begin{equation}\label{form44}
\varphi_1 \equiv {\rm id} \mod \Lambda_+^{\Z_2}.
\end{equation}
Here $\hat\varphi = \{\varphi_k \mid k=1,2,\dots \}$.
\end{thm}
We remark that $CF((R(M),b_M),L)$ and $CF(M,L)$
are both
$C(R(M)\times_{R(\Sigma)} L;\Lambda_0^{\Z_2})$ as 
$\Lambda_0^{\Z_2}$ modules. So the identity map 
in the right hand side of (\ref{form44}) makes sense.
\par
Before proving Theorem \ref{them42} we draw its consequence.
\par
Let $b$ be a bounding cochain of $(C(L),\{\frak m_k\})$
we define $d^b : C(M,L) \to C(M;L)$ by
(\ref{form28}). Then $d^b \circ d^b =0$ and 
the Floer cohomology $HF(M,(L,b))$ is the cohomlogy group 
of $d^b$.
\par
In the same way we define $d^b = ({}^{b_M}\frak n^{b})_0 : CF((R(M),b_M),L) \to CF((R(M),b_M),L)$
by 
$$
d^b(y) 
=
\sum_{\ell=0}^{\infty}\sum_{k=0}^{\infty}
\frak m_{\ell + k + 1}(\underbrace{b_M,\dots,b_M}_{
\ell},y,\underbrace{b,\dots,b}_k).
$$
$d^b \circ d^b =0$ again holds and 
$HF((R(M),b_M),(L,b))$ is its cohomology group.
\begin{cor}\label{cor423}
There exists a canonical isomorphism
\begin{equation}\label{iso4545}
 HF((R(M),b_M),(L,b)) \cong HF(M,(L,b)).
\end{equation}
\end{cor}
\begin{proof}[Proof of Theorem \ref{them42} $\Rightarrow$ Corollary \ref{cor423}]
We define a map 
$$\varphi^b : CF((R(M),b_M),L) \to CF(M,L)$$
by 
$$
\varphi^b(y)
= \sum_{k=0}^{\infty} \varphi_k(y;b,\dots,b).
$$ 
Since $b \equiv 0 \mod \Lambda_+^{\Z_2}$ the right hand side 
converges in $T$-adic topology.
It is easy to check that
(\ref{for4242}) and (\ref{MCeq}) imply
$d^b \circ \varphi^b = \varphi^b \circ d^b$.
Namely $\varphi^b$ is a chain map.
Then (\ref{form44}) and $b \equiv 0 \mod \Lambda_+^{\Z_2}$
implies that 
$\varphi^b \equiv {\rm id} \mod \Lambda_+^{\Z_2}$.
Therefore $\varphi^b$ induces an isomorphism on cohomologies.
\end{proof}
\begin{proof}[Proof of Theorem \ref{them42}]
The main part of the proof is defining the moduli spaces which we use to define 
the operators $\varphi_k$.
\par
We take a domain $W \subset \C$ such that
the following holds. (See Figure 4.1.)
\begin{conds}\label{conds44}
\begin{enumerate}
\item 
The intersection 
$W \cap \{z \in \C \mid {\rm Im} z < -2\}$
is $\{z \in \C \mid \vert{\rm Re}z \vert \le 1, {\rm Im} z < -2 \}$.
\item
The intersection 
$W \cap \{z \in \C \mid {\rm Im} z > +2\}$
is $\{z \in \C \mid \vert{\rm Re} z\vert \le 1, {\rm Im} z > +2 \}$.
\item
The intersection 
$W \cap \{z \in \C \mid {\rm Re} z > 0\}$
is $\{z \in \C \mid 0 < {\rm Re} z \le 1 \}$.
\item
The intersection 
$W \cap \{z \in \C \mid {\rm Im} z < -2\}$
is $\{z \in \C \mid {\rm Im} z < -2,  \vert{\rm Re} z\vert \le 1\}$.
\item
The boundary $\partial W$ has three connected components 
$\partial_i W$ ($i=1,2,3$) each of which is a $C^{\infty}$ submanifold of $\C$ 
and is diffeomorphic to 
$\R$.
Moreover
$
\partial_1W = \{z \in \C \mid {\rm Re} z = 1\}
$,
$
\partial_2W 
\subset \{z \in \C \mid {\rm Re} z <0,  {\rm Im} z > 0\}
$, 
$
\partial_3W 
\subset \{z \in \C \mid {\rm Re} z <0,  {\rm Im} z < 0\}
$.
\end{enumerate}
\end{conds}
\par
\begin{center}
\includegraphics[scale=0.25]
{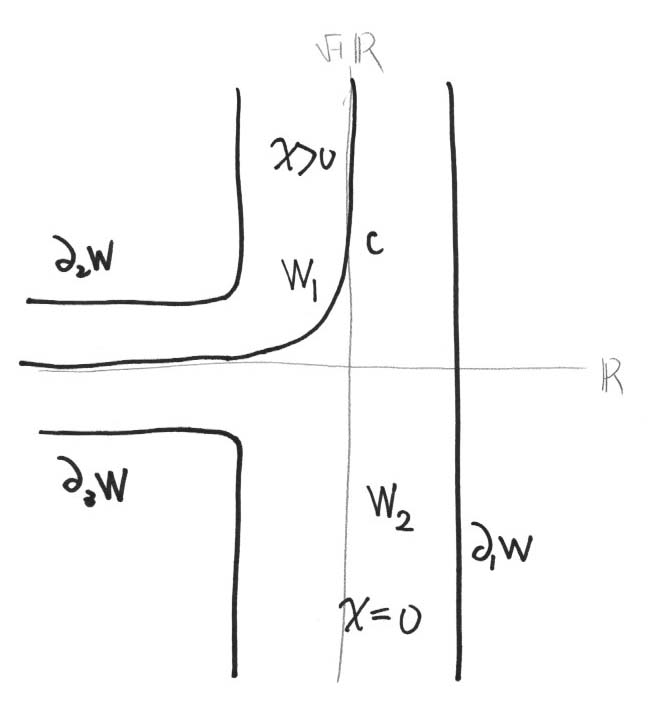}
\end{center}\centerline{\bf Figure 4.1}
\par\medskip
We take $\chi : W \to [0,1]$, a submanifold $C \subset W$ diffeomorphic to $\R$, and a Riemannian metric ${\bf g}$ on $\Sigma \times W_1$ with the following properties: 
(See Figure 4.1.)
\begin{conds}\label{cond45}
\begin{enumerate}
\item
$W \setminus C$ consists of two connected components $W_1$ and $W_2$.
Moreover
$\{z \in \C \mid \chi(z) > 0\} = W_1$.
\item
On $\{z \in W \mid {\rm Re} z < -3\}$, 
$\chi(z) = \chi(-{\rm Im} z)$, where $\chi$ in the right hand side is the same 
function as one appeared in (\ref{form211}).
${\bf g} = \chi^2 g_{\Sigma} + ds^2 + dt^2$ 
on $\Sigma \times \{z \in W \mid {\rm Re} z < -3\}$, 
where we put $z = t - \sqrt{-1}s$.
\item
On  $\{z \in W \mid {\rm Im} z < -3\}$, 
$\chi(z) = \chi({\rm Re} z)$, where $\chi$ in the right hand side is the same 
function as one appeared in (\ref{form211}).
${\bf g} = \chi^2 g_{\Sigma} + ds^2 + dt^2$ 
on $\Sigma \times \{z \in W \mid {\rm Im} z < -3\}$, 
where we put $z = s + \sqrt{-1}t$.
\item
In a neighborhood of $\Sigma \times \partial_2 W$,
the space $\Sigma \times W_1$ with metric ${\bf g}$ is isometric to the direct product
$g_{\Sigma} \times (0,\epsilon) \times \R$.
Here $g_{\Sigma} \times \{0\} \times \R$
corresponds to $\Sigma \times \partial_2 W$.
This isometry is compatible with the isometry obtained by 
items (2)(3) in the domain described by those items.
\item
Let $U(C)$ be  a neighborhood of $C$ in $\C$.
Then on $\Sigma \times (W_1\cap U(C))$, 
the metric ${\bf g}$ becomes $\chi^2(s,t) g_{\Sigma} + ds^2 + dt^2$.
where $s+\sqrt{-1}t$ is the standard coordinate of $\C$ and 
$\chi$ satisfies the condition of \cite[Lemma 4.7]{fu3}.
\end{enumerate}
\end{conds}
We extend the metric ${\bf g}$ on $\Sigma \times W_1$ to 
a `singular metric' on $\Sigma \times W$ by putting 
${\bf g} = 0 g_{\Sigma} + ds^2 + dt^2$ outside $\Sigma \times W_1$.
\par
By Condition \ref{cond45} (4), $(\Sigma \times W_1,{\bf g})$
is isometric to $(\Sigma \times (0,\epsilon) \times \R,g_{\Sigma} + ds^2 + dt^2)$ 
in a neighborhood of $\Sigma \times \partial_2W$.
\par
We remark that $M_0\times \R$ is isometric to 
$(\Sigma \times (-\epsilon,0) \times \R,g_{\Sigma} + ds^2 + dt^2)$
in an neighborhood of its boundary.
Therefore we can glue them together to obtain 
$(X,{\bf g})$. Here $\bf g$ is a `Riemannian metric' which is degenerate
on $\Sigma \times \overline{W_2}$.
The $SO(3)$ bundle $\mathcal E_M$ on $M$ induces an
$SO(3)$ bundle on $X$  in an obvious way, which we denote by $\mathcal E_X$.
\par
Note that $X$ has 3 ends and 2 boundary components. 
The 3 ends appear as ${\rm Im} z \to +\infty$, 
${\rm Re} z \to -\infty$,
${\rm Im} z \to -\infty$.
\par
The end corresponding to  ${\rm Im} z \to +\infty$ is 
$M \times [c,\infty)$.
The end corresponding to ${\rm Re} z \to -\infty$ is 
$M \times (-\infty,-c]$.
The end corresponding to ${\rm Im} z \to -\infty$ 
is $\Sigma \times [-1,1] \times (-\infty,-c]$.
\par
The boundaries are $\Sigma \times \partial_1W$ and 
$\Sigma \times \partial_3W$. Note 
$\Sigma \times \partial_2W$ is glued with 
$\partial M_0 \times \R$ and so is not a boundary of $X$.
\par
For a smooth connection $\frak A$ of $\mathcal E_X$ we can 
consider the `ASD-equation'. Namely we require (\ref{ASDeq})
on $X \setminus (\Sigma \times \overline{W_2})$
and (\ref{ASDprod}) on $\Sigma \times W \subset X$.
(Note (\ref{ASDeq}) coincides with (\ref{ASDprod}) on the overlapped part.)
We say $\frak A$ is an ASD-connection by an abuse of notation 
if it satisfies (\ref{ASDeq})
on $X \setminus (\Sigma \times \overline{W_2})$
and (\ref{ASDprod}) on $\Sigma \times W \subset X$.
\par\newpage
\begin{center}
\includegraphics[scale=0.25]
{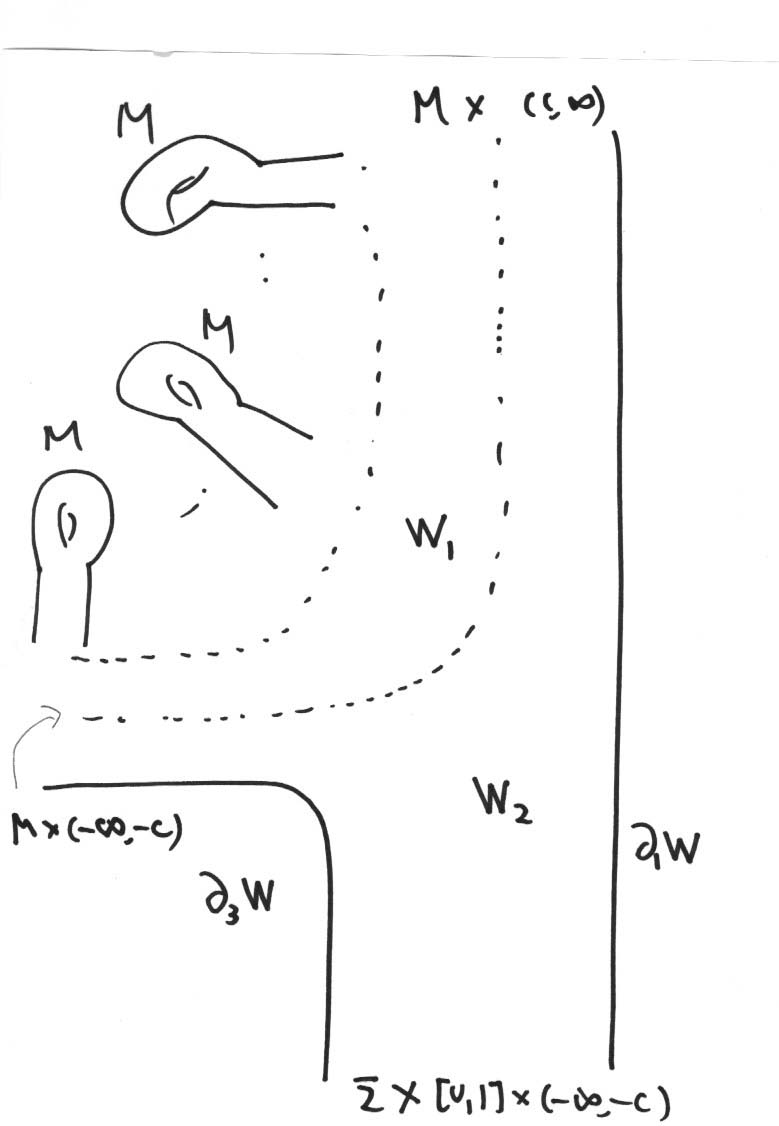}
\end{center}
\centerline{\bf Figure 4.2}
\par\medskip
We consider also $\Omega$ which contains $W_2$.
Namely $\Omega$ is a union of $W_2$
and tree of disk and sphere components attached 
to $\partial W_2$ and ${\rm Int} W_2$, respectively.
We consider the pair $(\Omega,u)$
which satisfies Condition \ref{conds27}, except we replace 
$(0,1)\times \R$ and $\{0\} \times \R$  by
$W_2$ and $C$, respectively.
We call this condition Condition \ref{conds27}'.
\par
Let $\partial_1\Omega$ 
(resp. $\partial_3\Omega$) be the union of $\partial_1W$ 
(resp.  $\partial_3W$)
and the boundary of the disk components attached to $\partial_1W$
(resp.  $\partial_3W$).
\par
Now we modify Definitions \ref{defn2929} and \ref{defn23}
as follows.
We consider the decompositions
(\ref{form236}) and  (\ref{form235}).  Let $I(L)$ and $I(R(M),L)$ be 
the index sets as in there.
We consider the decomposition
 (\ref{form235})
 in case $L=R(M)$ and let
$I(R(M))$ be its index set.
Namely
$$
\aligned
\tilde L \times_{R(\Sigma)} \tilde L &= \bigcup_{j\in I(L)} \hat L_j \\
R(M) \times_{R(\Sigma)} R(M) &= \bigcup_{j\in I(R(M))} \widehat{R(M)}_j \\
R(M) \times_{R(\Sigma)} \tilde L &= \bigcup_{j\in I(R(M),L)} R_i.
\endaligned
$$
Let $R_1$ be a connected component of 
$R(M) \times_{R(\Sigma)}R(M)$ and $R_2$, $R_3$ be connected components 
of  $R(M) \times_{R(\Sigma)} \tilde L$.
In other words $R_1 = \widehat{R(M)}_{j_1}$ for some $j_1\in I(R(M))$.
$R_2 = R_{j_2}$, $R_3 = R_{j_3}$ for some $j_2, j_3\in I(R(M),L)$.
\par
We also take $\vec i^{(1)}$,  $\vec i^{(3)}$
as in Definition \ref{defn46640} (5) below.
\begin{defn}\label{defn46640}
We define the set $\overset{\circ}{{\mathcal M}}((X,\mathcal E_X),R(M),L;R_1,R_2,R_3;
\vec i^{(1)},\vec i^{(3)};E)$
as the set of all equivalence classes of $(\frak A,{
\frak z},{\frak w},\Omega,u,\vec z^{(1)},\vec z^{(3)})$ satisfying the following conditions.
(See Figure 4.3.)
\begin{enumerate}
\item
$\frak A$ is a connection of $\mathcal E_{X}$ satisfying equations  
 (\ref{ASDeq}), (\ref{ASDprod}).
\item
$\frak z = (\frak z_1,\dots,\frak z_{m_1})$ is an {\it unordered} 
$m_1$-tuple of points of $X \setminus (\Sigma \times \overline{W_2})$.
We put $\Vert \frak z\Vert = m_1$.
We say subset $\{\frak z_1,\dots,\frak z_{m_1}\}
\subset X \setminus (\Sigma \times \overline{W_2})$ the {\it support} of $\frak z$ 
and denote it by $\vert\frak z\vert$.
We define ${\rm multi} : \vert\frak z\vert \to \Z_{>0}$ by 
$
{\rm multi}(x) = \#\{i \mid z_i = x\}$
and call it the {\it multiplicity function}.
\item
$\frak w = (\frak w_1,\dots,\frak w_{m_2})$ is an {\it unordered} 
$m_2$-tuple of points of $C$.
We put $\Vert \frak w\Vert = m_2$.
We say the subset $\{\frak w_1,\dots,\frak w_{m_2}\} 
\subset C$ the {\it support} of $\frak w$.
We define ${\rm multi} : \vert\frak w\vert \to \Z_{>0}$ by 
$
{\rm multi}(x) = \#\{i \mid w_i = x\}$
and call it the {\it multiplicity function}.
\item
$\Omega$ satisfies Condition \ref{conds27}'.
\item
$\vec i^{(1)} = ( i^{(1)}(1),\dots,i^{(1)}(k_1)) 
\in I(L)^{k_1}$ and  $\vec i^{(3)} = ( i^{(3)}(1),\dots,i^{(3)}(k_3)) 
\in I(R(M))^{k_3}$
\item
$\vec z^{(1)} = (z_1^{(1)},\dots,z_{k_1}^{(1)})$
(resp. $\vec z^{(3)} = (z_1^{(3)},\dots,z_{k_3}^{(3)})$)
$z_i^{(1)}$ lies on $\partial_1 \Omega$,
(resp. $z_i^{(3)}$ lies on $\partial_3 \Omega$).
None of $z^{(1)}_i$ or $z^{(3)}_i$ is a nodal point. 
$z^{(1)}_i \ne z^{(1)}_j$, $z^{(3)}_i \ne z^{(3)}_j$  if  $i\ne j$.
$(z^{(1)}_1,\dots,z^{(1)}_{k_1})$ 
(resp. $(z^{(3)}_1,\dots,z^{(3)}_{k_3})$ ) 
respects counter clockwise orientation of 
$\partial_1 \Omega$ (resp. $\partial_3 \Omega$).
\item
There exists $\gamma^{(1)} : \partial_1 \Omega \setminus \{z^{(1)}_1,
\dots,z^{(1)}_{k_1}\}
\to 
\tilde L$
such that $u(z) = i_L(\gamma^{(1)}(z))$ on $\partial_1 \Omega \setminus \{z^{(1)}_1,
\dots,z^{(1)}_{k_1}\}$.
\par
There exists $\gamma^{(3)} : \partial_3 \Omega \setminus \{z^{(3)}_1,
\dots,z^{(3)}_{k_3}\}
\to 
R(M)$
such that $u(z) = i_{R(M)}(\gamma^{(3)}(z))$ on $\partial_3 \Omega \setminus \{z^{(3)}_1,
\dots,z^{(3)}_{k_3}\}$.
\item
For $j=1,\dots,k_1$ the following holds. 
\begin{equation}\label{form23777rev}
(\lim_{z \uparrow z^{(1)}_j}\gamma^{(1)}(z),\lim_{z\downarrow z^{(1)}_j}\gamma^{(1)}(z))
\in \hat L_{i^{(1)}(j)}.
\end{equation}
Here the notation $z \uparrow z_j$, $z \downarrow z_j$ is 
defined in the same way as (\ref{form23777})\par
For $j=1,\dots,k_3$ the following holds. 
\begin{equation}\label{form23777revrev}
(\lim_{z \uparrow z^{(3)}_j}\gamma^{(3)}(z),\lim_{z\downarrow z^{(3)}_j}\gamma^{(3)}(z))
\in \widehat{R(M)}_{i^{(3)}(j)}.
\end{equation}
Here the notation $z \uparrow z_j$, $z \downarrow z_j$ is 
defined in the same way as (\ref{form23777})\footnote{We remark that 
$z \uparrow z_j$ here means that $z$ moves to the counter-clock-wise 
way towards $z$. So ${\rm Im} z > {\rm Im} z_j$.}.
\item
We replace Condition \ref{conds28} (3)
by the stability of 
$(\Omega,u,\vec z^{(1)},\vec z^{(3)})$. Namely the set of all maps $v : \Omega \to \Omega$ 
satisfying the next three conditions is a finite set.
\begin{enumerate}
\item
$v$ is a homeomorphism and is holomorphic on each of the irreducible components.
\item
$v$ is the identity map on $(0,1] \times \R \subseteq \Omega$.
\item $u \circ v = u$.
\item $v(z^{(1)}_j) = z^{(1)}_j$, $j=1,\dots,k_1$ 
and $v(z^{(3)}_j) = z^{(3)}_j$, $j=1,\dots,k_3$. 
\end{enumerate}
\item
For $(s,t) \in W_2$
we have
$$
[A(s,t)] = u(s,t). 
$$
Here $A(s,t)$ is obtained from $\frak A$ by  (\ref{AnadfraA}).
\item
The energy of $(\frak A,{
\frak z},{\frak w},\Omega,u)$
which is defined in the same way as Definition \ref{defnenergy2} is $E$.
\item
We assume the following 
asymptotic boundary conditions, which are defined by using $R_1,R_2,R_3$.
\begin{enumerate}
\item
\begin{equation}\label{form23777revrevrev}
(\lim_{z \to -1-\infty\sqrt{-1}}\gamma^{(3)}(z),\lim_{z\to 
+1-\infty\sqrt{-1}}\gamma^{(1)}(z))
\in R_1.
\end{equation}
Here $\lim_{z \to -1-\infty\sqrt{-1}}$ is the limit 
when the imaginary part of $z \in \partial_3 W$ goes to $-\infty$.
(We remark that then the ${\rm Re}z = -1$ by 
Condition \ref{conds44} (4)(5).)
The meaning of 
$\lim_{z\to 
+1-\infty\sqrt{-1}}$ is similar.
\item
We consider the restriction of $\frak A$ to $\Sigma \times \{z \in W\mid {\rm Im} z = c\}$
for $c >3$. We glue it with the restriction of $\frak A$ to 
$M_0 = M \setminus (\Sigma \times (-1,1])$, which is attached to $-1 + c\sqrt{-1}$.
(See Figure 4.4.) We call it $\frak A\vert_{{\rm Im} z = c}$.
It is a connection of the bundle $\mathcal E_M$ on $M$.
We assume that $\frak A\vert_{{\rm Im} z = c}$ converges to a 
flat connection as $c \to + \infty$. We write its limit
$\lim_{c\to+\infty} \frak A\vert_{{\rm Im} z = c}$.
\par
Then we also assume
\begin{equation}\label{form23777revrevrev2}
(\lim_{c\to+\infty} \frak A\vert_{{\rm Im} z = c},\lim_{z\to 
+1+\infty\sqrt{-1}}\gamma^{(1)}(z))
\in R_2.
\end{equation}
Here the meaning of $\lim_{z\to
+1+\infty\sqrt{-1}}$ is similar to 
(\ref{form23777revrevrev}).
\item
We consider the restriction of $\frak A$ to $\Sigma \times \{z 
\in W\mid {\rm Re} z = c\}$
for $c < -3$. We glue it with the restriction $\frak A$ to 
$M_0$ which is attached to $c + \sqrt{-1}$.
(See Figure 4.5.) We call it $\frak A\vert_{{\rm Re} z = c}$.
It is a connection of the bundle $\mathcal E_M$ on $M$.
We assume that $\frak A\vert_{{\rm Re} z = c}$ converges to a 
flat connection as $c \to - \infty$. We write its limit
$\lim_{c\to-\infty} \frak A\vert_{{\rm Re} z = c}$.
\par
Then we also assume
\begin{equation}\label{form23777revrevrev3}
(\lim_{c\to-\infty} \frak A\vert_{{\rm Re} z = c},\lim_{z\to 
-\infty -\sqrt{-1}}\gamma^{(3)}(z))
\in R_3.
\end{equation}
Here the meaning of $\lim_{z\to 
+1+\infty\sqrt{-1}}$ is similar to 
(\ref{form23777revrevrev}).
\end{enumerate}
\end{enumerate}
The equivalence relation is defined in the same way as Definition \ref{equivlimiobj}.
\end{defn}
\par\newpage
\begin{center}
\includegraphics[scale=0.25]
{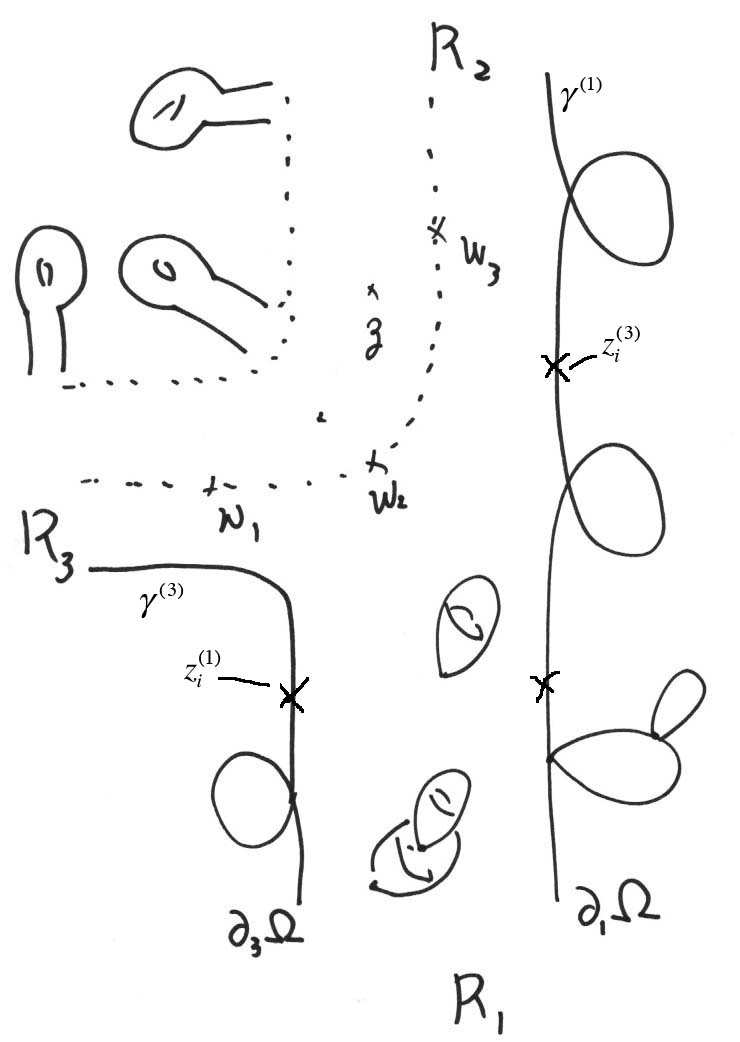}
\end{center}
\centerline{\bf Figure 4.3}
\par
\begin{center}
\includegraphics[scale=0.25]
{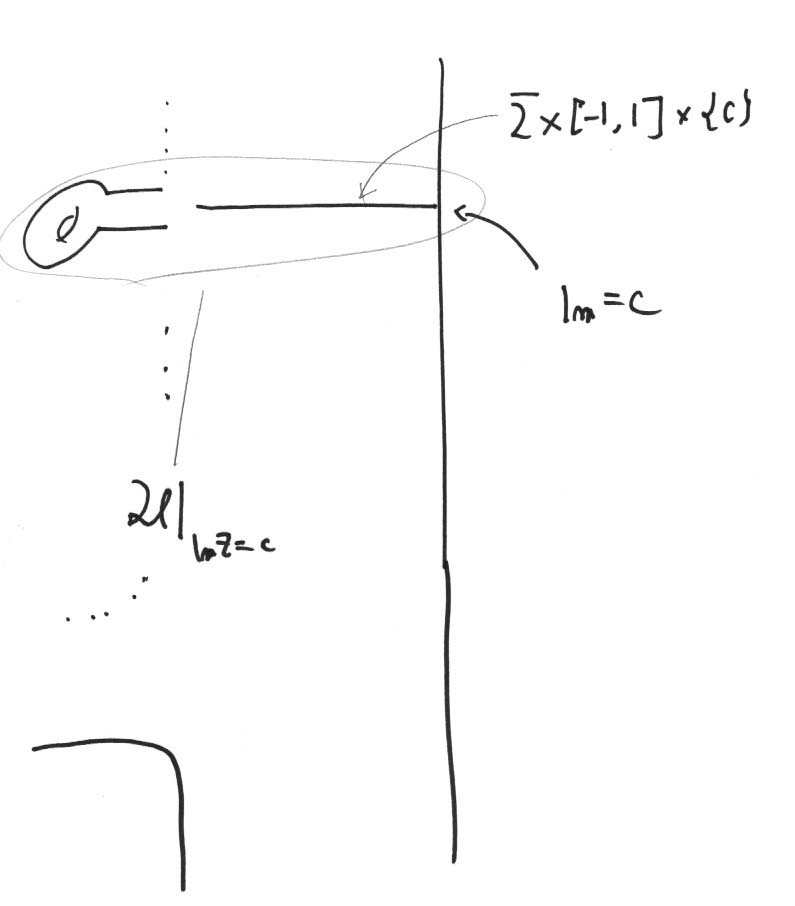}
\end{center}
\centerline{\bf Figure 4.4}
\par
\begin{center}
\includegraphics[scale=0.25]
{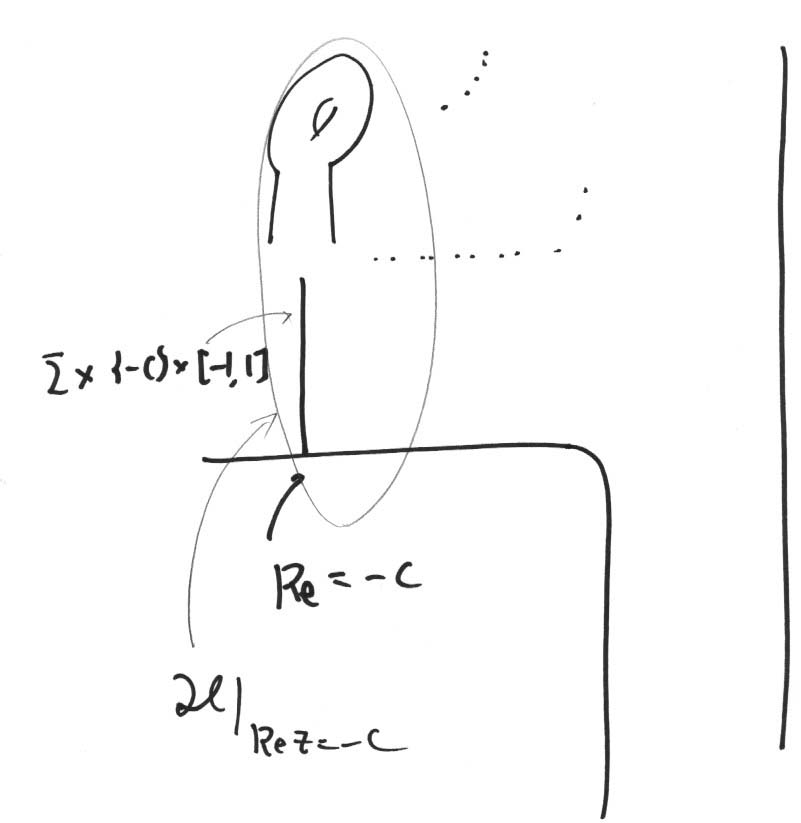}
\end{center}
\centerline{\bf Figure 4.5}
\par

We can define a topology on
$\overset{\circ}{{\mathcal M}}((X,\mathcal E_X),R(M),L;R_1,R_2,R_3;
\vec i^{(1)},\vec i^{(3)};E)$ by modifying Definition \ref{defn1222}
in an obvious way.
\par
We put
$$
\aligned
&\overset{\circ}{{\mathcal M}}_{k_1,k_3}((X,\mathcal E_X),L;R_1,R_2,R_3;E) \\
&=
\bigcup_{\vec i^{(1)}; \vert \vec i^{(1)} \vert = k_1}
\bigcup_{\vec i^{(3)}; \vert \vec i^{(3)} \vert = k_3}
\overset{\circ}{\widetilde{\mathcal M}}((X,\mathcal E_X),L;R_1,R_2,R_3;\vec i^{(1)},\vec i^{(3)};E)
\endaligned
$$
$\overset{\circ}{{\mathcal M}}_{k_1,k_3}((X,\mathcal E_X),L;R_1,R_2,R_3;E)$ is a Hausdorff space.
\par
We define evaluation maps
\begin{equation}
\aligned
&{\rm ev}^{(1)}  :
\overset{\circ}{{\mathcal M}}_{k_1,k_3}((X,\mathcal E_X),L;R_1,R_2,R_3;E)
\to (\tilde L \times_{R(\Sigma)} \tilde L)^{k_1}
\\
&{\rm ev}^{(3)}  :
\overset{\circ}{{\mathcal M}}_{k_1,k_3}((X,\mathcal E_X),L;R_1,R_2,R_3;E)
\to (R(M) \times_{R(\Sigma)} R(M))^{k_3}
\endaligned
\end{equation}
They are defined by (\ref{form23777rev}) and (\ref{form23777revrev}).
\par
We also define the evaluation maps
\begin{equation}
{\rm ev}^{\infty}_i  :
\overset{\circ}{{\mathcal M}}_{k_1,k_3}((X,\mathcal E_X),L;R_1,R_2,R_3;E)
\to R_i
\end{equation}
for $i=1,2,3$. They are defined by 
(\ref{form23777revrevrev}), (\ref{form23777revrevrev2}) and
(\ref{form23777revrevrev3}).
\par
Note $\overset{\circ}{{\mathcal M}}_{k_1,k_3}((X,\mathcal E_X),L;R_1,R_2,R_3;E)$
is not yet compact. There are still three types of ends, which are:
\begin{enumerate}
\item[(I)]
An ASD-connection escape to the direction ${\rm Im}(z) \to +\infty$.
\item[(II)]
An ASD-connection escape to the direction ${\rm Re}(z) \to -\infty$.
\item[(III)]
A pseudo-holomorphic strip escape to the direction ${\rm Im}(z) \to -\infty$.
\end{enumerate}
The end (I) is described by the union of the fiber products:
\begin{equation}\label{form4133}
\overset{\circ}{{\mathcal M}}_{k'_1,k_3}((X,\mathcal E_X),L;R_1,R'_2,R_3;E_1) 
{}_{{\rm ev}^{\infty}_2}\times_{{\rm ev}_-}
{\mathcal M}_{k''_1}((M,\mathcal E),L;R'_2,R_2;E_2)
\end{equation}
Here the union is taken over all $k'_1$, $k''_1$, $E_1$, $E_2$, $R'_2$
such that $k'_1 + k''_1 = k_1$, $E_1 + E_2 = E$ and 
$R'_2$ is a connected component of 
$R(M) \times_{R(\Sigma)} \tilde L$.
Note ${\mathcal M}_{k''_1}((M,\mathcal E),L;R'_2,R_2;E_2)$ 
is defined in Definition \ref{defn23}.
See Figure 4.6.
\par
\begin{center}
\includegraphics[scale=0.25]
{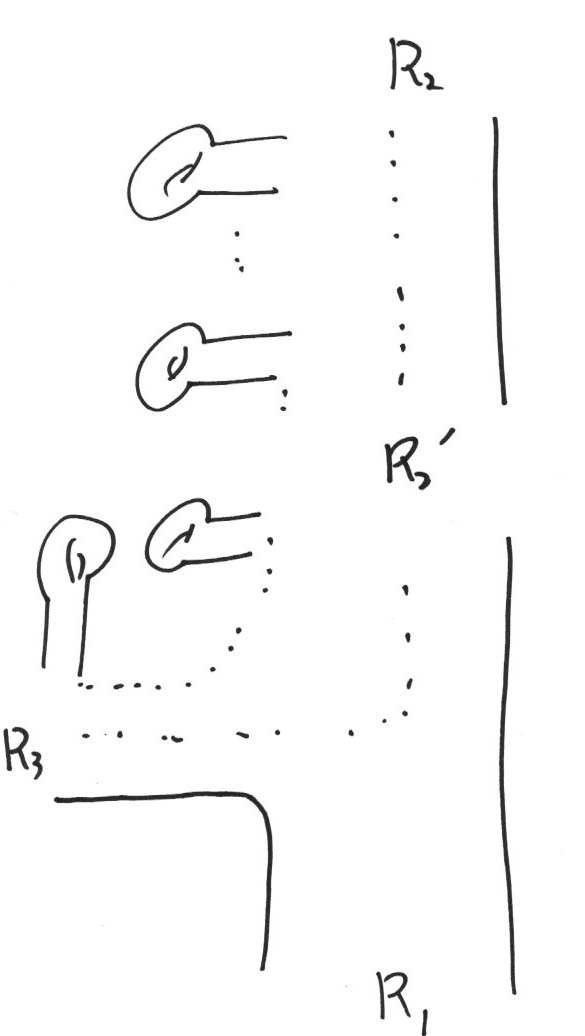}
\end{center}
\centerline{\bf Figure 4.6}
\par
The end (II) is described by the union of the fiber products:
\begin{equation}\label{form414}
\aligned
&\overset{\circ}{{\mathcal M}}_{k_1,k'_3}((X,\mathcal E_X),L;R_1,R_2,R'_3;E_1)  \\
&{}_{{\rm ev}^{\infty}_3}\times_{{\rm ev}_+}
{\mathcal M}_{k''_3}((M,\mathcal E),R(M);R_3,R'_3;E_2)
\endaligned
\end{equation}
Here the union is taken over all $k'_3$, $k''_3$, $E_1$, $E_2$, $R'_3$
such that $k'_3 + k''_3 = k_3$, $E_1 + E_2 = E$ and 
$R'_3$ is a connected component of 
$R(M) \times_{R(\Sigma)} R(M)$.
See Figure 4.7.
We remark that in the second line $R_3$ appears first and $R'_3$ next.
(Namely $R_3,R'_3$ and not $R'_3,R_3$.)
The reason is as follows. 
In Figure 4.7 `the bubble'  component lies in the left. 
We need to rotate it by 90 degree counter-clock-wise direction 
to put it in the same way as Figure 2.2.
Then after rotation, $R_3$ will lie in the direction ${\rm Im}z \to 
-\infty$.
\par\newpage
\begin{center}
\includegraphics[scale=0.25]
{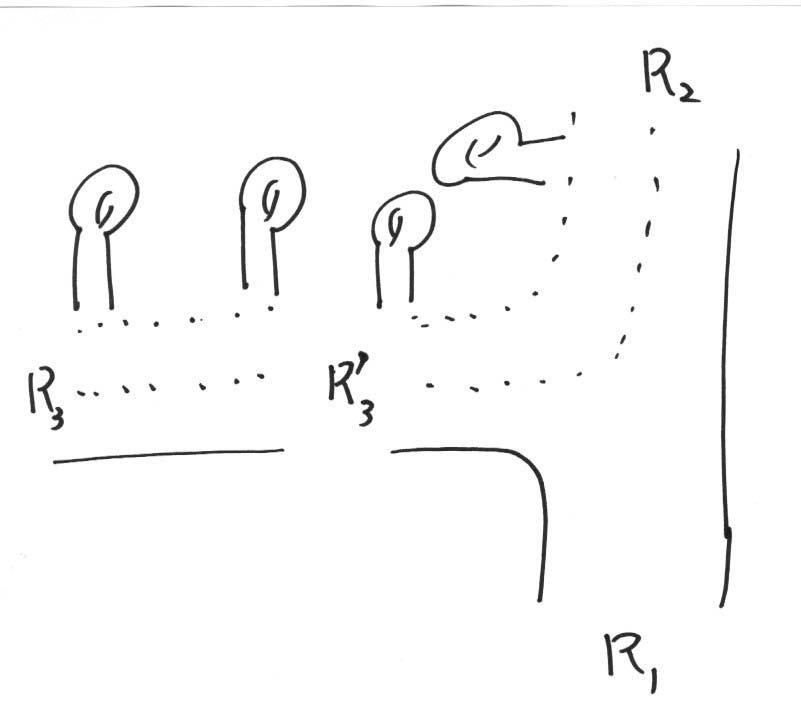}
\end{center}
\centerline{\bf Figure 4.7}
\par
To describe the end (III) we use the moduli space of pseudo-holomorphic strips 
which is used to define Lagrangian Floer homology
$HF(R(M),L)$. We here need a digression and review the definition of $HF(R(M),L)$.
The discussion below is a generalization of \cite[Section 3.7.4]{fooobook}
to the case of a pair of {\it immersed} Lagrangian submanifolds.
\cite{AJ} did not discuss the case of pairs of 
immersed Lagrangian submanifolds since we may regard 
the union as a single immersed Lagrangian submanifold
and so it is actually included in the case of  single immersed Lagrangian submanifold.
Therefore the only point which is not literally written 
in \cite{AJ} is that we include the case when the intersection of 
two immersed Lagrangian submanifolds are clean but not 
transversal. 
(This generalization is not a big deal and one can handle it in the same way
as \cite[Section 3.7.4]{fooobook}.)
\par
It seems more natural to explain it in a general situation rather than
a special case we use here.
Let $(Y,\omega)$ be a monotone symplectic manifold,
(We assume monotonicity here since we use $\Z_2$ coefficient.)
and $i_{L_j} : \tilde L_j \to Y$ be Lagrangian immersion for $j=1,2$.
We fix a compatible almost complex structure $J_Y$ on $Y$.
We assume that the self-intersection of $L_j$ are transversal for $j=1,2$ and 
the fiber product $\tilde L_1 \times_Y \tilde L_2$ is clean.
We decompose the fiber products into connected components and put:
$$
\aligned
\tilde L_j \times_{Y} \tilde  L_j  &= \bigcup_{k\in I(L_j)} \hat L_{j,k} \quad \text{$j =1,2$}, \\
\tilde L_1 \times_{Y} \tilde L_2 &= \bigcup_{k\in I(L_1,L_2))} \widehat{R}_k \endaligned
$$
Let $R_-$, $R_+$ be connected components of $\tilde L_1 \times_{Y} \tilde L_2$.
Let $\vec i^{(j)} = (i^{(j)}(1),\dots,i^{(j)}(k_j))$
where $i^{(j)}(1),\dots,i^{(j)}(k_j) \in I(L_j)$.
\par
We define the moduli space $\overset{\circ}{\widetilde{\mathcal M}}(L_1,L_2;R_-,R_+;\vec i^{(1)},\vec i^{(2)};E)$ as follows.
\begin{defn}\label{defn47}
The moduli space 
$\overset{\circ}{{\mathcal M}}(L_1,L_2;R_-,R_+;\vec i^{(1)},\vec i^{(2)};E)$
is the set of all the equivalence classes of $(\Omega,u,\vec z^{(1)},\vec z^{(2)})$ 
with the following properties.
(See Figure 4.8.)
\begin{enumerate}
\item
$\Omega$ is a union of the domain
\begin{equation}\label{form45}
\Omega_0 = \{z \in \C \mid \vert{\rm Re} z\vert \le 1\},
\end{equation}
trees of spheres attached to the interior of $\Omega_0$
and trees of disk components attached to the boundary of $\Omega_0$.
(The disk components may contain a tree of sphere components 
attached to its interior.)
We denote by $\partial_1\Omega$ (resp. $\partial_2\Omega$)
the union of 
$\{z \in \Omega \mid {\rm Re z} = -1\}$ 
and the boundaries of the tree of disks attached to it
(resp. the union of 
$\{z \in \Omega \mid {\rm Re z} = 1\}$ 
and the boundaries of the tree of disks attached to it.)
\par
We remark $\partial \Omega = \partial_1\Omega \cup \partial_2\Omega$.
\item
$u : \Omega \to R(\Sigma)$ is a pseudo-holomorphic map.
\item
$\vec z^{(1)} = (z_1^{(1)},\dots,z_{k_1}^{(1)})$
(resp. $\vec z^{(2)} = (z_1^{(2)},\dots,z_{k_2}^{(2)})$)
$z_i^{(1)}$ lies on $\partial_1 \Omega$,
(resp. $z_i^{(2)}$ lies on $\partial_2 \Omega$).
None of $z^{(1)}_i$ or $z^{(2)}_i$ is a nodal point. 
$z^{(1)}_i \ne z^{(1)}_j$, $z^{(2)}_i \ne z^{(2)}_j$  if  $i\ne j$.
$(z^{(1)}_1,\dots,z^{(1)}_{k_1})$ 
$(z^{(1)}_1,\dots,z^{(1)}_{k_1})$ 
(resp. $(z^{(2)}_1,\dots,z^{(2)}_{k_2})$ ) 
respects counter clockwise orientation of 
$\partial_1 \Omega$ (resp. $\partial_2 \Omega$).
\item
There exists $\gamma^{(1)} : \partial_1 \Omega \setminus \{z^{(1)}_1,
\dots,z^{(1)}_{k_1}\}
\to 
\tilde L_1$
such that $u(z) = i_{L_1}(\gamma(z))$ on $\partial_1 \Omega \setminus \{z^{(1)}_1,
\dots,z^{(1)}_{k_1}\}$.
\par
There exists $\gamma^{(2)} : \partial_2 \Omega \setminus \{z^{(2)}_1,
\dots,z^{(2)}_{k_2}\}
\to 
\tilde L_2$
such that $u(z) = i_{L_2}(\gamma(z))$ on $\partial_2 \Omega \setminus \{z^{(2)}_1,
\dots,z^{(2)}_{k_2}\}$.
\item
For $j=1,\dots,k_1$ the following holds. 
\begin{equation}\label{form47}
(\lim_{z \uparrow z^{(1)}_j}\gamma^{(1)}(z),\lim_{z\downarrow z^{(1)}_j}\gamma^{(1)}(z))
\in \hat L_{1,i(j)}.
\end{equation}
Here the notation $z \uparrow z_j$, $z \downarrow z_j$ is 
defined in the same way as (\ref{form23777})\par
For $j=1,\dots,k_2$ the following holds. 
\begin{equation}\label{form47rev}
(\lim_{z \uparrow z^{(2)}_j}\gamma^{(2)}(z),\lim_{z\downarrow z^{(2)}_j}\gamma^{(2)}(z))
\in \hat L_{2,i(j)}.
\end{equation}
Here the notation $z \uparrow z_j$, $z \downarrow z_j$ is 
defined in the same way as (\ref{form23777}).
\item
We asuume the stability of 
$(\Omega,u,\vec z^{(1)},\vec z^{(2)})$. Namely the set of all maps $v : \Omega \to \Omega$ 
satisfying the next three conditions is a finite set.
\begin{enumerate}
\item
$v$ is a homeomorphism and is holomorphic on each of the irreducible components.
\item
$v$ is the identity map on $\Omega_0$.
\item $u \circ v = u$.
\item $v(z^{(1)}_j) = z^{(1)}_j$, $j=1,\dots,k_1$ 
and $v(z^{(2)}_j) = z^{(2)}_j$, $j=1,\dots,k_2$. 
\end{enumerate}
\item
The energy of $u$ is $E$. Namely
$$
\int_{\Omega} u^*\omega = E.
$$
\item
We assume the following 
asymptotic boundary conditions, which are defined by using $R_-,R_+$.
\begin{enumerate}
\item
\begin{equation}\label{form418}
(\lim_{z \to -1-\infty\sqrt{-1}}\gamma^{(1)}(z),\lim_{z\to 
+1-\infty\sqrt{-1}}\gamma^{(2)}(z))
\in R_-.
\end{equation}
Here $\lim_{z \to -1-\infty\sqrt{-1}}$ is the limit 
when the imaginary part of $z$ goes to $-\infty$ and 
$z \in \partial_1 \Omega$.
The meaning of 
$\lim_{z\to 
+1-\infty\sqrt{-1}}$ is similar.
\item
\begin{equation}\label{form418rev}
(\lim_{z \to -1+\infty\sqrt{-1}}\gamma^{(1)}(z),\lim_{z\to 
+1+\infty\sqrt{-1}}\gamma^{(2)}(z))
\in R_+.
\end{equation}
Here $\lim_{z \to -1+\infty\sqrt{-1}}$ is the limit 
when the imaginary part of $z$ goes to $-\infty$ and 
$z \in \partial_1 \Omega$.
The meaning of 
$\lim_{z\to 
+1-\infty\sqrt{-1}}$ is similar.
\end{enumerate}\end{enumerate}
\par
\begin{center}
\includegraphics[scale=0.25]
{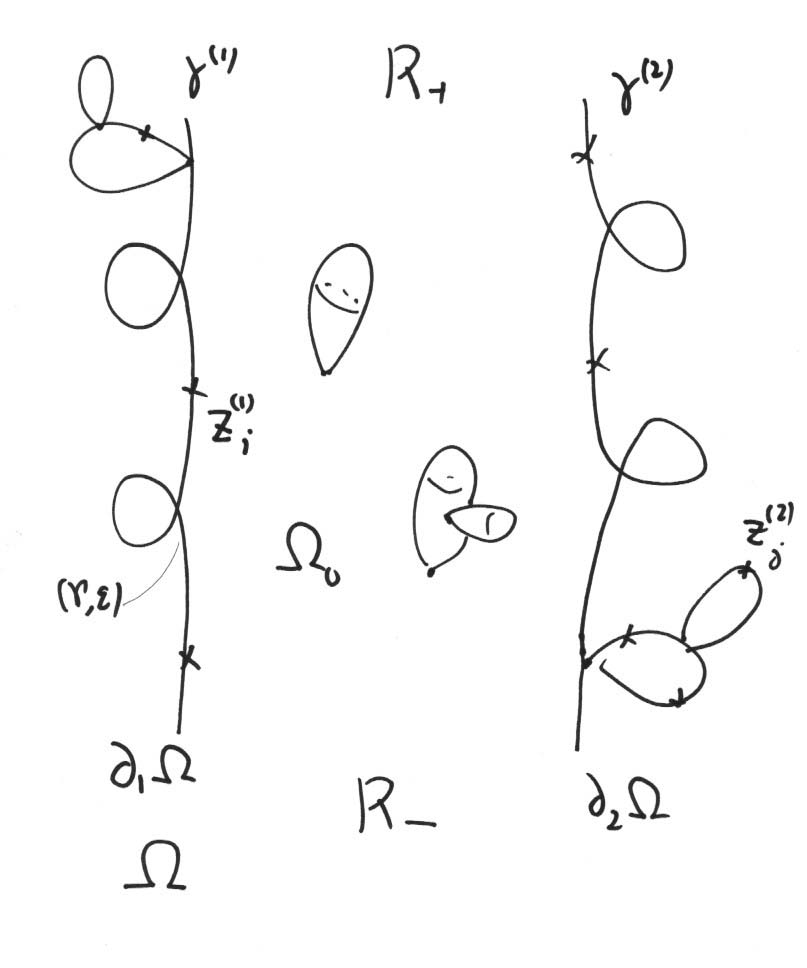}
\end{center}
\par
\centerline{\bf Figure 4.8}
\par

We say 
$(\Omega^1,u^1,\vec z^{1,(1)},\vec z^{1,(2)})$
is equivalent to 
$(\Omega^2,u^2,\vec z^{2,(1)},\vec z^{2,(2)})$
if there exists a homeomorphism 
$v : \Omega^1 \to \Omega^2$
such that
\begin{enumerate}
\item
$v$ is biholomorphic on each irreducible component.
\item
$v(\partial_1\Omega^1) = \partial_1\Omega^2$.
\item
$v(z^{1,(1)}_j) = z^{2,(1)}_j$, $j=1,\dots,k_1$ 
and $v(z^{1,(2)}_j) = z^{2,(2)}_j$, $j=1,\dots,k_2$. 
\item 
$u^2\circ v = u^1$.
\end{enumerate}
The $\R$ action 
which translate the ${\rm Im}z$ direction of $\Omega$
is actually included in the definition of equivalence relation above. 
\end{defn}
We put
$$
\aligned
&\overset{\circ}{{\mathcal M}}_{k_1,k_2}(L_1,L_2;R_-,R_+;E)\\
&=
\bigcup_{\vec i^{(1)}; \vert \vec i^{(1)} \vert = k_1}
\bigcup_{\vec i^{(2)}; \vert \vec i^{(2)} \vert = k_2}
\overset{\circ}{{\mathcal M}}(L_1,L_2;R_-,R_+;\vec i^{(1)},\vec i^{(2)};E)
\endaligned
$$

We define evaluation maps
\begin{equation}\label{evmap420}
\aligned
&{\rm ev}^{(1)}  :
\overset{\circ}{{\mathcal M}}_{k_1,k_2}(L_1,L_2;R_-,R_+;E)
\to (\tilde L_1 \times_{Y} \tilde L_1)^{k_1}
\\
&{\rm ev}^{(2)}  :
\overset{\circ}{{\mathcal M}}_{k_1,k_2}(L_1,L_2;R_-,R_+;E)
\to (\tilde L_2 \times_{Y} \tilde L_2)^{k_2}.
\endaligned
\end{equation}
They are defined by (\ref{form47}) and (\ref{form47rev}).
\par
We also define the evaluation maps
\begin{equation}\label{evmap421}
{\rm ev}^{\infty}_i  :
\overset{\circ}{{\mathcal M}}_{k_1,k_2}(L_1,L_2;R_-,R_+;E)
\to R_i
\end{equation}
for $i=+, -$. They are defined by 
(\ref{form418}) and (\ref{form418rev}).
\par
We consider the union of the fiber products:
\begin{equation}\label{fibercompactkarefref}
\aligned
&\overset{\circ}{{\mathcal M}}_{k_1^0,k_2^0}(L_1,L_2;R_-,R_1;E_0)
\times_{R_1} \\
&\overset{\circ}{{\mathcal M}}_{k_1^1,k_2^1}(L_1,L_2;R_1,R_2;E_1)
\times_{R_2} \dots \\
& \dots  \times_{R_{\ell-1}}\overset{\circ}{{\mathcal M}}_{k_1^{\ell-1},k_2^{\ell-1}}(L_1,L_2;R_{\ell-1},R_{\ell};E_{\ell-1}) \\
&\times_{R_{\ell}} {{\mathcal M}}_{k_1^{\ell},k_2^{\ell}}(L_1,L_2;R_{\ell},R_+;E_{\ell}),
\endaligned
\end{equation}
where $k^{(j)}_0 + k^{(j)}_1 + \dots +k^{(j)}_{\ell} = k_j$ for $j=1,2$, 
$E_0+\dots + E_{\ell} = E$, and $R_i$ for $i=1,\dots,\ell$
are connected components of $\tilde L_1 \times_{Y} \tilde L_2$.
This union is by definition 
$
{{\mathcal M}}_{k_1,k_2}(L_1,L_2;R_-,R_+;E)
$.
The evaluation maps (\ref{evmap420}) and (\ref{evmap421}) 
extends there.

\begin{prop}\label{prop48}
We can define a topology on $
{{\mathcal M}}_{k_1,k_2}(L_1,L_2;R_-,R_+;E)
$
by which it becomes compact and Hausdorff.
\par
It has a Kuranishi structure with corner such that 
its boundary is a union of the following three types of 
fiber products.
\begin{enumerate}
\item
$$
{{\mathcal M}}_{k'_1,k'_2}(L_1,L_2;R_-,R;E')
\times_R {{\mathcal M}}_{k''_1,k''_2}(L_1,L_2;R,R_+;E'')
$$
where $k'_1 + k''_1 = k_1$, $k'_2 + k''_2 = k_2$, $E' + E'' = E$ and 
$R$ is a connected component of $\tilde L_1 \times_{Y} \tilde L_2$.
See Figure 4.9.
\item
$$
{{\mathcal M}}_{k'_1,k_2}(L_1,L_2;R_-,R_+;E')
{}_{{\rm ev}^{(1)}_{i}}\times_{{\rm ev}_0} {{\mathcal M}}_{k''_1}(L_1;E'')
$$
where $k'_1 + k''+1 = k_1$, $E' + E'' = E$ and 
$i = 1,\dots,k'_1$.
The second factor is (\ref{diskmoduli}).
(More precisely its analogue in immersed case.)
See Figure 4.10.
\item
$$
{{\mathcal M}}_{k_1,k'_2}(L_1,L_2;R_-,R_+;E')
{}_{{\rm ev}^{(2)}_{i}}\times_{{\rm ev}_0} {{\mathcal M}}_{k''_2}(L_2;E'')
$$
where $k'_2 + k''+2 = k_2$, $E' + E'' = E$ and 
$i = 1,\dots,k'_2$.
The second factor is (\ref{diskmoduli}).
(More precisely its analogue in immersed case.)
See Figure 4.11.
\end{enumerate}
\end{prop}
\par
\begin{center}
\includegraphics[scale=0.25]
{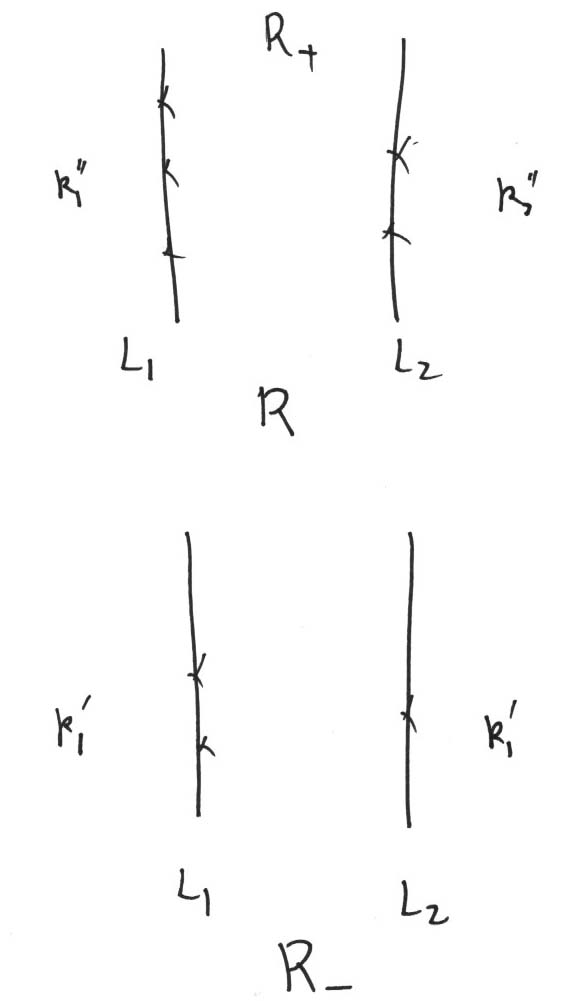}
\end{center}
\centerline{\bf Figure 4.9}
\par
\begin{center}
\includegraphics[scale=0.25]
{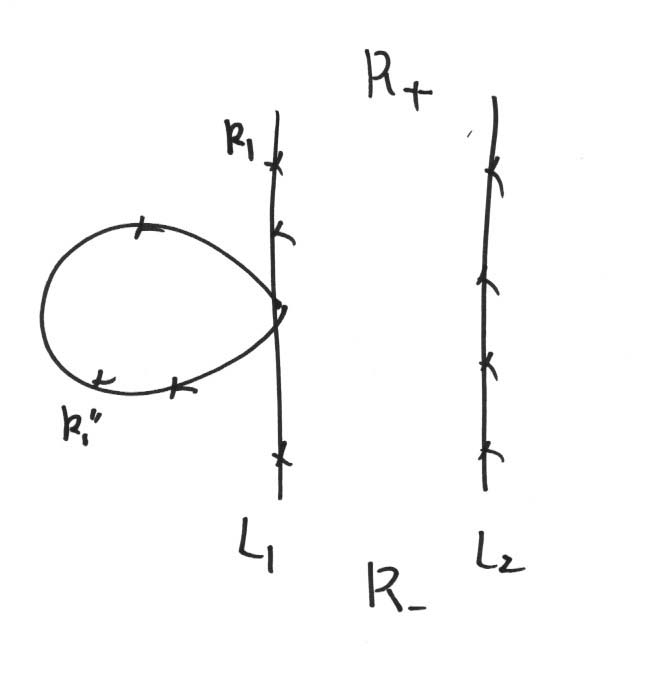}
\end{center}
\centerline{\bf Figure 4.10}
\par
\begin{center}
\includegraphics[scale=0.25]
{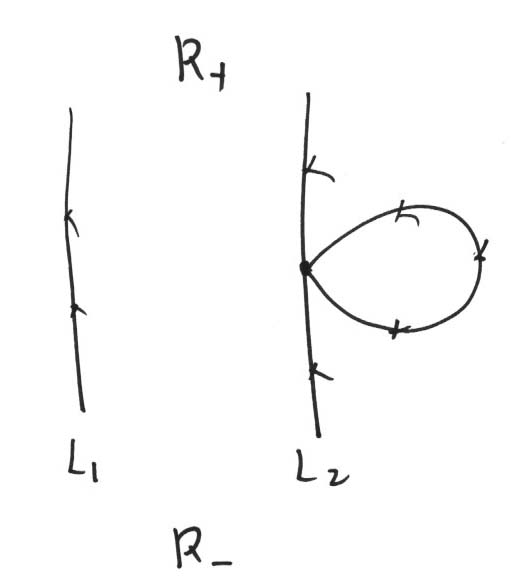}
\end{center}
\centerline{\bf Figure 4.11}
\par
The proof of Proposition \ref{prop48} is entirely similar to the proof of 
\cite[Propositions 7.1.1, 7.1.2]{fooobook2}.
\par
Let $x^1_1,\dots,x^1_{k_1} \in C(\tilde L_1 \times_{Y} \tilde L_1)$, 
$y_-, y_+ \in C(\tilde L_1 \times_{Y} \tilde L_2)$
and $x^2_1,\dots,x^2_{k_2} \in C(\tilde L_2 \times_{Y} \tilde L_2)$.
We define:
\begin{equation}\label{form423}
\aligned
&\langle
\frak n_{k_1,k_2;E}(x^1_1,\dots,x^1_{k_1};y_-;x^2_1,\dots,x^2_{k_2}),y_+
\rangle
\\
&= \#
\big(
{{\mathcal M}}_{k_1,k_2}(L_1,L_2;R_-,R_+;E)\,\,
{}_{{\rm ev}^{(1)},{\rm ev}^{\infty}_-,{\rm ev}^{(2)},{\rm ev}^{\infty}_+}\times\\
&\qquad\qquad (x^1_1 \times \dots \times x^1_{k_1}\times y_- \times x^2_1 \times \dots \times x^2_{k_2} \times y_+)
\big).
\endaligned
\end{equation}
There are various chain models by which (\ref{form423}) becomes rigorous.
(See \cite{fooobook2}, \cite{fooo:overZ}, \cite{foootech2}, \cite{foootech22}.)
We will introduce a chain model which is also suitable to handle the moduli space appearing 
in gauge theory case 
in \cite{fu8}, or in a separate paper.
(We remark again we can work over $\Z_2$ because $Y$ is assumed to 
be monotone. See \cite{fooo:overZ}.)
\par
We then put
\begin{equation}
\aligned
\frak n_{k_1,k_2} = &\sum T^E \frak n_{k_1,k_2;E} 
\\&: CF(L_1)^{k_1\otimes} \otimes_{\Lambda_0^{\Z_2}}
CF(L_1,L_2) \otimes_{\Lambda_0^{\Z_2}} CF(L_2)^{k_2\otimes} 
\to CF(L_1,L_2).
\endaligned
\end{equation}
Here 
$$
CF(L_i) = C(\tilde L_i \times_{Y} \tilde L_i;\Lambda_0^{\Z_2}), 
\qquad 
CF(L_1,L_2) = C(\tilde L_1 \times_{Y} \tilde L_2;\Lambda_0^{\Z_2}).
$$
\par
It is a part of the general theory of Kuranishi structure and 
virtual fundamental chain 
(See \cite{foootech22} for its thorough detail) that Proposition \ref{prop48} and our definitions imply the following:
\begin{cor}\label{cor4949}
$\{\frak n_{k_1,k_2}\}$ defines a structure of filtered $A_{\infty}$ bimodule on $CF(L_1,L_2)$ over 
$(CF(L_1),\{\frak m_k\})$-$(CF(L_2),\{\frak m_k\})$.
\par
Namely we have
\begin{equation}\label{form425}
\aligned
0 = &\sum \frak n_{k'_1,k'_2}(x^1_1,\dots,\frak n_{k''_1,k''_2}(x^1_{k'_1+1},\dots x^1_{k_1} ; y ; x^2_1,\dots,x^2_{k''_2});
\dots,x^2_{k_2}) \\
&+\sum \frak n_{k'_1,k_2}(x^1_1,\dots,\frak m_{k''_1}(x_i,\dots,x_{i+k''_1-1}),\dots,x^1_{k_1};y;x^2_1,\dots,x^2_{k_2})
\\
&+\sum \frak n_{k_1,k'_2}(x^1_1,\dots,x^1_{k_1};y;x^2_1,\dots,\frak m_{k''_2}(x_i,\dots,x_{i+k''_2-1}),\dots,x^2_{k_2}).
\endaligned
\end{equation}
Here the sum in the first line is taken over $k'_1, k''_1$, $k'_2, k''_2$ with 
$k_1 = k'_1 + k''_1$, $k_2 = k'_2 + k''_2$.
The sum in the second line is taken over $k'_1, k''_1, i$ with $k_1+1 = k'_1 + k''_1$, 
$i=1,\dots,k'_1+1$.
The sum in the third line is taken over $k'_2, k''_2, i$ with $k_2+1 = k'_2 + k''_2$, 
$i=1,\dots,k'_2+1$.
\end{cor}
See \cite[Formula (3.7.2)]{fooobook}.
We remark that the first, second, third lines of Formula (\ref{form425}) corresponds to 
(1), (2), (3) in Proposition \ref{prop48}, respectively.
\par
Let $b_1$, $b_2$ be bounding cochains of $CF(L_1)$, $CF(L_2)$, respectively.
Following \cite[Definition-Lemma 3.7.13]{fooobook} we define
$$
\delta_{b_1,b_2} : CF(L_1,L_2) \to CF(L_1,L_2)
$$
by 
\begin{equation}
\delta_{b_1,b_2}(y) 
= 
\sum_{k_1,k_2} \frak n_{k_1,k_2}(b_1,\dots,b_1;y;b_2,\dots,b_2).
\end{equation}
(\ref{form425}) implies $\delta_{b_1,b_2}\circ \delta_{b_1,b_2} = 0$.
(\cite[Lemma 3.7.14]{fooobook}.)
\begin{defn}
Floer homology of the pair $((L_1,b_1),(L_2,b_2))$ is 
$$
HF((L_1,b_1),(L_2,b_2)) = \frac{{\rm Ker}\, \delta_{b_1,b_2}}{{\rm Im}\, \delta_{b_1,b_2}}.
$$
\end{defn}
This is the Floer homology appearing in (\ref{glueiso}), (\ref{formula13}), (\ref{iso4545}).
\par\smallskip
We finish digression and go back to the description of the compactifiation of the moduli space
$\overset{\circ}{{\mathcal M}}_{k_1,k_3}((X,\mathcal E_X),L;R_1,R_2,R_3;E)$.
By definition, ends of type (III) correspond to the following fiber product.
\begin{equation}\label{427427}
\overset{\circ}{{\mathcal M}}_{k'_1,k'_3}((X,\mathcal E_X),L;R'_1,R_2,R_3;E_1)
{}_{{\rm ev}^{\infty}_1}\times_{{\rm ev}_-} 
{{\mathcal M}}_{k''_1,k''_2}(R(M),L;R'_1,R_1;E_2).
\end{equation}
Here $k'_1 + k''_1 = k_1$, $E_1 + E_2 = E$ and 
$R'_1$ is a connected component of $R(M) \times_{R(\Sigma)} \tilde L$.
\par\newpage
\begin{center}
\includegraphics[scale=0.25]
{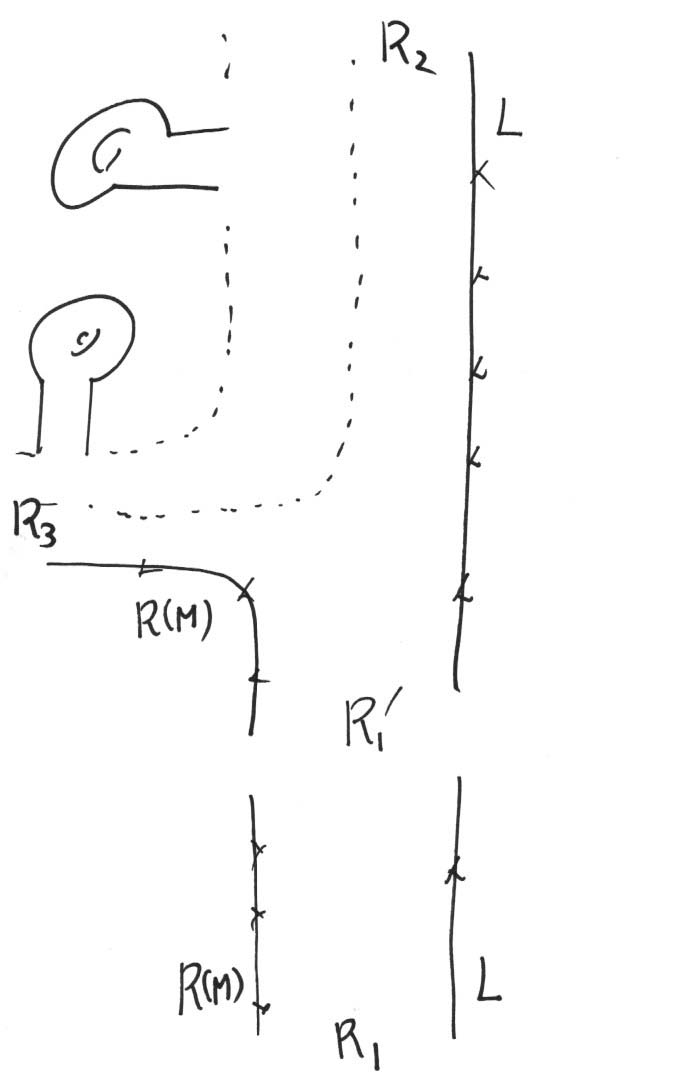}
\end{center}
\centerline{\bf Figure 4.12}
\par
\begin{defn}\label{defn411}
The set ${{\mathcal M}}_{k_1,k_3}((X,\mathcal E_X),L;R_1,R_2,R_3;E)$
is the union of the following fiber products:
\begin{equation}\label{form428}
\aligned
&\overset{\circ}{{\mathcal M}}_{k^1_1,k^1_3}((X,\mathcal E_X),L;R'_1,R'_2,R'_3;E_1)\\
&{}_{{\rm ev}^{\infty}_2}\times_{{\rm ev}_-}
{\mathcal M}_{k^2_1}((M,\mathcal E),L;R'_2,R_2;E_2) \\
&{}_{{\rm ev}^{\infty}_3}\times_{{\rm ev}_+}
{\mathcal M}_{k^2_3}((M,\mathcal E),R(M);R_3,R'_3;E_3)\\
&{}_{{\rm ev}^{\infty}_1}\times_{{\rm ev}_+}  
{{\mathcal M}}_{k^3_1,k^3_2}(R(M),L;R_1,R'_1;E_4).
\endaligned
\end{equation}
Here $k_1^1 + k_1^2 + k_1^3 = k_1$, $k_2^1 + k_2^2 + k_2^3 = k_2$,
$E_1 + E_2 + E_3 + E_4 = E$ and 
$R'_2$, $R'_3$, $R'_1$ are connected components of 
$R(M) \times_{R(\Sigma)} R(M)$, $R(M) \times_{R(\Sigma)} R(M)$, $R(M) \times_{R(\Sigma)} \tilde L$, 
respectively.
\par
We remark that we include the case when some of the 2nd, 3rd, 4th factors of the 
fiber product (\ref{form428}) is trivial.
For example, if $E_2 = 0$, $R'_2 = R_2$ and $k_1^2 = 0$ then 
the factor ${\mathcal M}_{k^2_1}((M,\mathcal E),L;R'_2,R_2;E_2)$
drops.
\end{defn}
\begin{prop}\label{prop412}
We can define a topology on 
${{\mathcal M}}_{k_1,k_3}((X,\mathcal E_X),L;R_1,R_2,R_3;E)$
so that it becomes compact and Hausdorff.
\par
The space ${{\mathcal M}}_{k_1,k_3}((X,\mathcal E_X),L;R_1,R_2,R_3;E)$
carries a virtual fundamental chain such that 
its boundary is the sum of the virtual fundamental chains of the 
following 5 types of fiber products.
\begin{enumerate}
\item The compactification of (\ref{form4133}), which is:
$$
{{\mathcal M}}_{k'_1,k_3}((X,\mathcal E_X),L;R_1,R'_2,R_3;E_1) 
{}_{{\rm ev}^{\infty}_2}\times_{{\rm ev}_-}
{\mathcal M}_{k''_1}((M,\mathcal E),L;R'_2,R_2;E_2).
$$
\item
The compactification of (\ref{form414}), which is:
$$
{{\mathcal M}}_{k_1,k'_3}((X,\mathcal E_X),L;R_1,R_2,R'_3;E_1) 
{}_{{\rm ev}^{\infty}_3}\times_{{\rm ev}_+}
{\mathcal M}_{k''_3}((M,\mathcal E),R(M);R_3,R'_3;E_2).
$$
\item
The compactification of (\ref{427427}), which is:
$$
{{\mathcal M}}_{k'_1,k'_3}((X,\mathcal E_X),L;R'_1,R_2,R_3;E_1)
{}_{{\rm ev}^{\infty}_1}\times_{{\rm ev}_+} 
{{\mathcal M}}_{k''_1,k''_2}(R(M),L;R_1,R'_1;E_2).
$$
\item
$$
{{\mathcal M}}_{k'_1,k_3}((X,\mathcal E_X),L;R_1,R_2,R_3;E_1)
{}_{{\rm ev}^{(1)}_i}\times_{{\rm ev}_0} {{\mathcal M}}_{k''_1}(L;E_2),
$$
where $k'_1 + k''_1 = k_1 + 1$, $E_1 + E_2 = E$.
\item
$$
{{\mathcal M}}_{k_1,k'_3}((X,\mathcal E_X),L;R_1,R_2,R_3;E_1)
{}_{{\rm ev}^{(3)}_i}\times_{{\rm ev}_0} {{\mathcal M}}_{k''_3}(R(M);E_2),
$$
where $k'_3 + k''_3 = k_3 + 1$, $E_1 + E_2 = E$.
\end{enumerate}
\end{prop}
\begin{proof}
(1)(2)(3) corresponds to the end of Type (I), (II), (III) respectively.
(4) corresponds to the disk bubble on $\partial_1W$.
(See Figure 4.13.)
(5) corresponds to the disk bubble on $\partial_3W$.
All other bubbles occur in codimension 2.
\end{proof}
\begin{center}
\includegraphics[scale=0.25]
{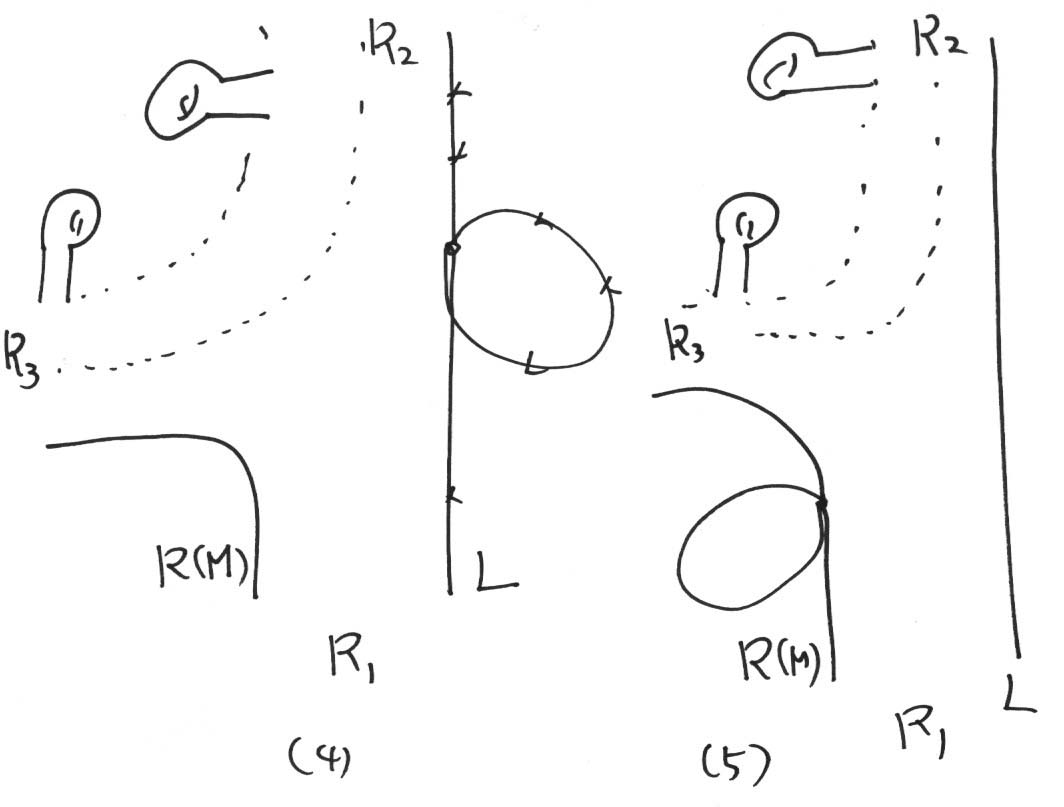}
\end{center}
\centerline{\bf Figure 4.13}
\par\medskip
We remark again that 
${{\mathcal M}}_{k_1,k_3}((X,\mathcal E_X),L;R_1,R_2,R_3;E)$
does not carry Kuranishi structure because Uhlenbeck compactification 
of the moduli space of ASD connections does not carry one.
So we need to generalize the story of Kuranishi structure and 
its virtual fundamental chain.
We will do it in \cite{fu8} or in a separate paper.
\par
Now we consider
$x^{(1)}_1,\dots,x^{(1)}_{k_1} \in C(\tilde L\times_{R(\Sigma)}\tilde L;\Lambda_0^{\Z_1})$, 
$x^{(3)}_1,\dots,x^{(3)}_{k_3} \in C(R(M)\times_{R(\Sigma)}R(M);\Lambda_0^{\Z_1})$, 
$y_1 \in C(R(M)\times_{R(\Sigma)}R(M);\Lambda_0^{\Z_1})$
$y_2, y_3 \in C(R(M)\times_{R(\Sigma)}\tilde L;\Lambda_0^{\Z_1})$.
We define
$$
\psi_{k_1,k_3} :
CF(M, R(M)) 
\otimes CF(R(M))^{k_3\otimes}
\otimes CF(R(M),L)
\otimes  CF(L)^{k_1\otimes}
\to CF(M,L)
$$
by
\begin{equation}\label{form430}
\aligned
&\langle
\psi_{k_3,k_1}(y_2;x^{(3)}_1,\dots,x^{(3)}_{k_3};y_3;x^{(1)}_1,\dots,x^{(1)}_{k_1})
,y_1
\rangle
\\
&= \#\big(
{{\mathcal M}}_{k_3,k_1}((X,\mathcal E_X),L;R_1,R_2,R_3;E) \\
&\quad{}_{{\rm ev}^{\infty}_3,{\rm ev}^{(3)},{\rm ev}^{\infty}_1,{\rm ev}^{(1)},
{\rm ev}^{\infty}_2}\times \\
&\quad (y_2\times x^{(3)}_1 \times \dots
\times x^{(3)}_{k_3}\times  y_3\times  x^{(1)}_1\times \dots\times x^{(1)}_{k_1}
\times y_1)
\big).
\endaligned
\end{equation}
\begin{lem}\label{lem413}
$\psi_{k_1,k_3}$ satisfies the next equality.
\begin{equation}\nonumber
\aligned
&\sum\frak n_{k'_1}(
\psi_{k_3,k''_1}(y_2;x^{(3)}_1,\dots,x^{(3)}_{k''_3};y_3;x^{(1)}_1,\dots,x^{(1)}_{k''_1}),
x^{(1)}_{k''_1+1},\dots,x^{(1)}_{k_1}) \\
+& \sum\psi_{k''_3,k_1}
(\frak n_{k'_3}(y_2;x^{(3)}_1,\dots,x^{(3)}_{k'_3});
x^{(3)}_{k'_3+1},\dots,x^{(3)}_{k_3};y_3;x^{(1)}_1,\dots,x^{(1)}_{k_1})
\\
+&\sum\psi_{k'_3,k'_1}
(y_2;x^{(3)}_1,\dots,x^{(3)}_{k'_3},
\frak n_{k''_3,k''_1}(x^{(3)}_{k'_3+1},\dots,x^{(3)}_{k_3};y_3;
x^{(1)}_1, \dots, x^{(1)}_{k''_1}),
\dots, x^{(1)}_{k_1})\\
+&\sum\psi_{k''_3,k_1}
(y_2;x^{(3)}_1,\dots,\frak m_{k_3'}(x^{(3)}_i,\dots,x^{(3)}_{i+k_3'-1}),
\dots x^{(3)}_{k_3};y_3;x^{(1)}_1,\dots,x^{(1)}_{k_1}) \\
+& \sum \psi_{k_3,k''_1}(y_2;x^{(3)}_1,\dots,x^{(3)}_{k_3};y_3
;x^{(1)}_1,\dots,\frak m_{k'_1}(x^{(1)}_i,\dots,x^{(1)}_{i+k'_1-1}),\dots,x^{(1)}_{k_1})\\
=&0.
\endaligned
\end{equation}
Here the sum in the first line is taken over $k'_1, k''_1$ with $k_1 = k'_1+ k''_1$
and $\frak n_{k'_1}$ is the right filtered $A_{\infty}$ module structure of 
$CF(M, L)$.
\par
The sum in the second line is taken over $k'_3, k''_3$ with $k_3 = k'_3+ k''_3$
and $\frak n_{k'_3}$ in the second line is 
the right filtered $A_{\infty}$ module structure of 
$CF(M, R(M))$.
\par
The sum in the third line is taken over $k'_1, k''_1$, $k'_3, k''_3$
with  $k_1 = k'_1+ k''_1$, $k_3 = k'_3+ k''_3$ and 
$\frak n_{k''_3,k''_1}$ in the third line is the 
filtered $A_{\infty}$ bimodule structure on 
$CF(R(M),L)$.
\par
The sum in the fourth line is taken over $k'_3, k''_3,i$ 
with $k_3 +1 = k'_3+ k''_3$, $i = 1,\dots,k''_3$
and $\frak m_{k_3'}$ in the fourth line is the 
filtered $A_{\infty}$ algebra structure on $CF(R(M))$.
\par
The sum in the fifth line is taken over $k'_1, k''_1,i$ 
with $k_1 +1 = k'_1+ k''_1$, $i = 1,\dots,k''_1$
and $\frak m_{k_1'}$ in the fourth line is the 
filtered $A_{\infty}$ algebra structure on $CF(L)$.
\end{lem}
\begin{proof}
This is a consequence of Proposition \ref{prop412}.
In fact boundary components (1), (2), (3), (4), (5) 
in Proposition \ref{prop412} corresponds to 
1st, 2nd, 3rd, 4th, 5th lines of (\ref{form430}) 
respectively.
\end{proof}
We now define
\begin{equation}
\varphi_{k _1}: CF((R(M),b_M),L) \otimes CF(L)^{\otimes k_1}
\to CF(M;L)
\end{equation}
by
\begin{equation}
\varphi_{k_1}(y_3;x^{(1)}_1,\dots,x^{(1)}_{k_1})
=
\sum_{k_3} \psi_{k_3,k_1}({\bf 1}_M;\underbrace{b_M,\dots,b_M}_{k_3};y_3;x^{(1)}_1,\dots,x^{(1)}_{k_1}),
\end{equation}
where ${\bf 1}_M$ (resp. $b_M$) is as in Definition \ref{def37}
(resp. Theorem \ref{thm39}).
\par
We claim
\begin{equation}\label{modulohomodesu}
\aligned
&\sum\frak n_{k''_1}(\varphi_{k'_1}(y_3;x^{(1)}_1,\dots,x^{(1)}_{k'_1}),\dots,k''_1)
\\
+&\sum\varphi_{k''_1}({}^{b_M}\frak n_{k'_1}(y_3;x^{(1)}_1,\dots,x^{(1)}_{k'_1}),\dots,k''_1)
\\
+& \sum
\frak n_{k'_1}(y_3;x^{(1)}_1,\dots,\frak m_{k''_1}(x_i,\dots,x_{i+k''_1-1}),
\dots,x^{(1)}_{k_1})
=0.
\endaligned
\end{equation}
In fact 1st, 2nd, 3rd term of (\ref{modulohomodesu}) corresponds 
1st, 2nd, 5th lines of the formula in Lemma \ref{lem413}, respectively.
Note the sum of 3rd line of the formula in Lemma \ref{lem413} vanish 
when we put $y_2 = {\bf 1}_M$ and $x^{(3)}_i = b_M$, 
because of (\ref{1isccle}).
Note the sum of 4th line of the formula in Lemma \ref{lem413} vanish 
when we put $x^{(3)}_i = b_M$
since $b_M$ is a bounding cochain.
\par
(\ref{modulohomodesu}) means that 
$\hat\varphi = \{\varphi_k\}$ consists a 
filtered $A_{\infty}$ right module homomorphism
$$
(CF((R(M),b_M),L),\{{}^{b_M}\frak n_k\})
\to (CF(M;L),\{\frak n_k\})).
$$
\par
To prove (\ref{form44}) we consider the moduli space
${{\mathcal M}}_{0,0}((X,\mathcal E_X),L;R_1,R_2,R_3;0)$ in 
Definition \ref{defn47}.
We recall that this space consists of 
$(\frak A,{
\frak z},{\frak w},\Omega,u,\vec z^{(1)},\vec z^{(3)})$
where 
$\frak A$ is a flat connection, $u$ is a constant map,
$\Omega = (0,1] \times \R$ and
$\vec z^{(1)} = \vec z^{(3)} =\emptyset$.
We also consider the case $R_3$ is the fundamental 
class, which is 
${\bf 1}_M$.
Therefore the energy $0$ part of $\varphi_0$ is the identity map.
\par
The proof of Theorem \ref{them42} is complete.
\end{proof}
\begin{rem}\label{rem414}
We remark that $\psi_{k_1,k_3}$ induces a chain map
\begin{equation}\label{map434}
\psi : 
CF(M, (R(M),b_M)) 
\otimes CF((R(M),b_M),(L,b))
\to CF(M,(L,b))
\end{equation}
by using a bounding cochain $b$ of $L$ and 
$b_M$ of $R(M)$.
This map is very similar to the map $\frak m_2$ 
defining the composition of morphisms in 
$\mathscr{FUK}(R(\Sigma))$.
The only difference is: in (\ref{map434}) $M$ plays a role of a `Lagrangian submanifold'
or an object of $\mathscr{FUK}(R(\Sigma))$.\footnote{
To regard $M$ as an object of $\mathscr{FUK}(R(\Sigma))$ 
is an idea which was a starting point of the whole project of 
`Floer homology of 3 manifolds with boundary' and was mentioned in \cite{Don}.}
If we regard $M$ as an `object' of $\mathscr{FUK}(R(\Sigma))$, then  Theorem \ref{functordef}
claims that this object `$M$' is isomorphic to the object $(R(M),b_M)$
in $\mathscr{FUK}(R(\Sigma))$.
To prove this `fact' we need to find an element of 
$CF(M, (R(M),b_M)$ which gives the isomorphism.
The obvious candidate of such an element is the fundamental class ${\bf 1}_M$.
Our proof of Theorem \ref{functordef} shows that we can choose $b_M$ appropriately 
so that ${\bf 1}_M$ indeed becomes an isomorphism.
\end{rem}
\begin{rem}\label{rem415}
The proof of Theorem \ref{functordef} in this section 
has an analogue in the story of Wehrheim-Woodward functoriality.
In fact we can use it to prove the next theorem.
\begin{thm}\label{thm416}
Suppose we are in the situation Theorem \ref{WWfunc}.
Then for any spin immersed Lagrangian submanifold $L$ of $M_2$ and 
bounding cochain $b$ of the filtered $A_{\infty}$ algebra of $L$, 
we have a canonical isomorphism:
\begin{equation}
HF((L_{12},b_{12}),(L_1,b_1) \times (L,b)) 
\cong 
HF((L_2,b_2),(L,b)).
\end{equation}
\end{thm}
To prove Theorem \ref{thm416} we replace 
Figure 4.9 by the following Figure 4.14.
\par\newpage
\begin{center}
\includegraphics[scale=0.25]
{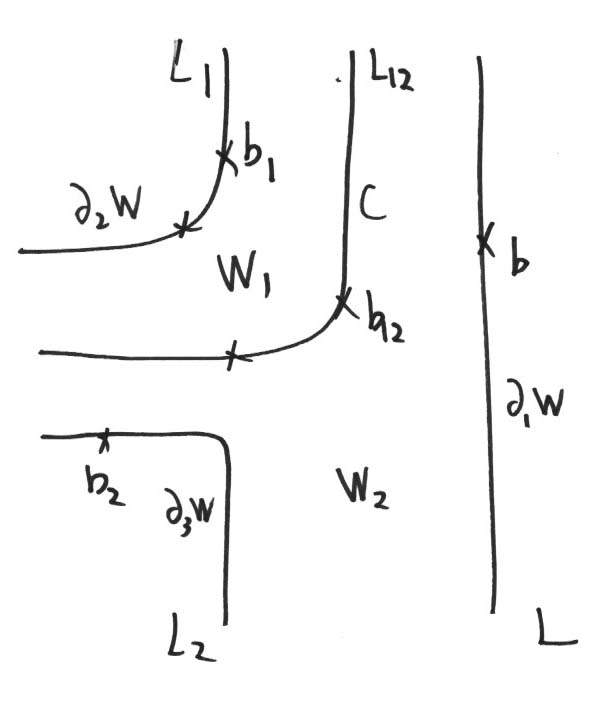}
\end{center}
\centerline{\bf Figure 4.14}
\par\medskip
Here we consider a combination of pseudo-holomorphic maps $u$ from $W$.
Namely on $W_1$ we consider $u$ as a map  to $M_1$ and on $W_2$ we 
consider $u$ as a map to $M_2$. 
\par
We put $C = W_1 \cap W_2$. We assume $u$ maps $C$ to $L_{12}$.
We also require $u(\partial_1W) \subset L$, $u(\partial_2 W) \subset L_1$ and $u(\partial_3W) \subset L_2$.
Here $\partial_2W = \partial W_1 \setminus C$, $\partial_1W \cup \partial_3W = \partial W_2 \setminus C$.
\par
We use bounding cochain $b_1$ on 
$\partial_2W$, $b_{12}$ on $C$, 
$b_2$ on $\partial_3 W$  and $b$ on $\partial_1 W$.
The role of ${\bf 1}_M$ 
in the proof of Theorem \ref{functordef} is taken by the fundamental class of 
$(L_1 \times L_2) \cap L_{12} \cong L_2$.
\par
We remark that Figure 4.14 is similar to 
\cite[Figure 1]{LL}. In fact in case $L_1$, $L_{12}$, $L_2$ and $L$ are all 
monotone, Figure 4.14  is a special case of 
\cite[Figure 1]{LL}.
In fact \cite{LL} discussed the case of composition of 
Lagrangian correspondences $L_{01}$ from $M_0$ to $M_1$ 
and $L_{12}$ from $M_1$ to $M_2$.
If we put $M_0 = $point, then it corresponds to our situation.
I also like to mention that the result of \cite{LL} obtained by \cite[Figure 1]{LL}
is a variant of an earlier results in \cite{WW} and the map 
$u : (W_1,W_2) \to (M_1,M_1 \times M_2)$ is a particular 
case of objects called pseudo-holomorphic quilt in \cite{WW}.
\end{rem}

\section{Gluing isomorphism of relative Floer homology}\label{sec:glue}

In this section we prove Theorem \ref{mainthm} (2).
The proof is based on a similar idea as the proof of Theorem \ref{functordef}
in Section \ref{sec:represent}.
We consider $(M_1,\mathcal E_1)$ 
and $(M_2,\mathcal E_2)$ as in Theorem \ref{mainthm} (2).
Let $(M,\mathcal E)$ be a closed 3-manifold with $SO(3)$ bundle 
obtained by gluing $(M_1,\mathcal E_1)$ 
and $(M_2,\mathcal E_2)$.
We fix a Riemannian metric on $M$.
We will construct a chain map
\begin{equation}\label{form51}
\Phi : CF((R(M_1),b_{M_1}),(R(M_2),b_{M_2})) \to CF(M,\mathcal E;\Lambda_0^{\Z_2}) 
\end{equation}
which is congruent to the identity map modulo $\Lambda_+^{\Z_2}$.
Here the chain complex $CF(M,\mathcal E;\Lambda_0^{\Z_2})$ is the Floer's chain 
complex which defines $SO(3)$-Floer homology of $(M,\mathcal E)$.
We first review its definition from \cite{fl1}, \cite{fl2}.
\par
We assume that $R(M_1)$ is transversal to $R(M_2)$ for simplicity.
(We can remove this assumption by appropriate perturbation.)
Then the set of flat connections on $(M,\mathcal E)$
is a finite set and is identified with 
$R(M_1) \cap R(M_2)$.
We denote this set by $R(M)$. 
\begin{defn}
Let $a_-,a_+ \in R(M)$. 
We denote by $\overset{\circ\circ}{\widetilde{\mathcal M}}(a_-,a_+)$
the gauge equivalence classes of all connections $\frak A$ on 
$(M \times \R,\mathcal E \times \R)$ such that the following holds.
\begin{enumerate}
\item 
$\frak A$ is an ASD connection. Namely it satisfies the equation
(\ref{ASDeq}).
\item
The energy $\mathcal E(\frak A)$ is finite. Here
$$
\mathcal E(\frak A) 
=
\int_{M \times \R}
\Vert F_{\frak A}\Vert^2 \Omega_{\bf g}.
 $$
 \item
We put the following asymptotic boundary condition.
There exists a gauge transformation $g$ such that:
\begin{equation}\label{form52}
\lim_{t\to \infty} g^*\frak A\vert_{M\times \{t\}} = a_+,
\qquad
\lim_{t\to -\infty} g^*\frak A\vert_{M\times \{t\}} = a_-.
\end{equation}
 \end{enumerate}
\end{defn}
\begin{rem}
Floer \cite{fl1} proved the following.
The transversality of $R(M_1)$ and $R(M_2)$ 
implies that the convergence in (\ref{form52}) is 
of exponential order in the sense of (\ref{expesti}).
(See also Remark \ref{rem2525}.)
\par
Floer also proved the following if $\frak A$ be a connection 
satisfying (1)(2) above. Then there exists $a_-$, $a_+$ 
such that (3) is satisfied.
\end{rem}
We divide $\overset{\circ\circ}{\widetilde{\mathcal M}}(a_-,a_+)$
by the $\R$ action defined by translation and denote by
$\overset{\circ\circ}{{\mathcal M}}(a_-,a_+)$ the quotient space.
Including bubble in the same way as Uhlenbeck compactification 
we obtain $\overset{\circ}{{\mathcal M}}(a_-,a_+)$.
Finally we define ${{\mathcal M}}(a_-,a_+)$ as the union of 
\begin{equation}\label{form5353}
\overset{\circ}{{\mathcal M}}(a_-,a_1) \times 
\overset{\circ}{{\mathcal M}}(a_1,a_2) \times 
\dots \times \overset{\circ}{{\mathcal M}}(a_k,a_+),
\end{equation}
where $a_1,\dots,a_k \in R(M)$. Then we can define a 
topological space
${{\mathcal M}}(a_-,a_+)$ so that it becomes compact and Hausdorff.
\par
We now define
\begin{defn}\label{defn53}
Let $CF(M,\mathcal E)$ be a $\Z_2$ vector space 
where the set of its basis is identified with $R(M)$.
\par
The Floer boundary operator $\partial : CF(M,\mathcal E)
\to CF(M,\mathcal E)$ is defined by:
\begin{equation}
\partial [a_-] = \sum_{a_+} \#({{\mathcal M}}(a_-,a_+)) [a_+],
\end{equation}
where the sum is taken over all $a_+ \in R(M)$ 
and a component of ${{\mathcal M}}(a_-,a_+)$ such that 
the virtual dimension of ${{\mathcal M}}(a_-,a_+)$ is zero.
\par
By using the moduli space ${{\mathcal M}}(a_-,a_+)$ in case 
its virtual dimension is $1$ we can prove 
$\partial \circ \partial = 0$.
We then put:
\begin{equation}
HF(M,\mathcal E) = \frac{{\rm Ker} \partial}{{\rm Im} \partial}
\end{equation}
and call it the $SO(3)$-{\it Floer homology}.
\end{defn}
Since in Lagrangian Floer theory, it is standard to use the universal Novikov ring 
$\Lambda_{0}^{\Z_2}$ as its coefficient ring, for the sake of consistency, 
we define a variant $HF(M,\mathcal E;\Lambda_0^{\Z_2})$ of 
$HF(M,\mathcal E)$ which is a module over $\Lambda_{0}^{\Z_2}$.
\par
We denote by ${{\mathcal M}}(a_-,a_+)_0$ the 
component of ${{\mathcal M}}(a_-,a_+)$ of virtual dimension $0$.
Using the monotonicity, we find that the energy 
$\mathcal E(\frak A)$ of elements $\frak A$ of ${{\mathcal M}}(a_-,a_+)_0$
is independent of $\frak A$.
We write it $\mathcal E(a_-,a_+;0)$.
(This number is not defined if ${{\mathcal M}}(a_-,a_+)_0$
is an empty set.) 
We put 
\begin{equation}\label{form55}
CF(M,\mathcal E; \Lambda_0^{\Z_2}) =CF(M,\mathcal E) \otimes_{\Z_2}
\Lambda_0^{\Z_2}
\end{equation} 
and define
$\partial' : CF(M,\mathcal E; \Lambda_0^{\Z_2})
\to CF(M,\mathcal E; \Lambda_0^{\Z_2})$ by:  
\begin{equation}\label{form56}
\partial' [a_-] = \sum_{a_+} T^{\mathcal E(a_-,a_+;0)}\#({{\mathcal M}}(a_-,a_+)) [a_+].
\end{equation}
The proof by Floer of the equality $\partial \circ \partial = 0$ 
can be used without change to show $\partial' \circ \partial' = 0$.
We now define:
\begin{defn}\label{defHFgauge}
\begin{equation}
HF(M,\mathcal E;\Lambda_0^{\Z_2}) = \frac{{\rm Ker} \partial'}{{\rm Im} \partial'}.
\end{equation}
\end{defn}
\begin{rem}
This remark is not related to the proof of Theorem \ref{mainthm} (2) so much 
but is related to the point that $\Lambda_0^{\Z_2}$ coefficient version 
might have some more information than $\Z_2$ coefficient version.
\begin{enumerate}
\item
Note 
$HF(M,\mathcal E;\Lambda_0^{\Z_2}) 
\ne HF(M,\mathcal E) \otimes_{\Z_2} \Lambda_0^{\Z_2}$.
In fact it is easy to see that
$$
\rank_{\Z_2} \frac{HF(M,\mathcal E;\Lambda_0^{\Z_2})}{\Lambda_+^{\Z_2}HF(M,\mathcal E;\Lambda_0^{\Z_2})} 
=
\# R(M),
$$
which can be in general different from $\rank_{\Z_2}  HF(M,\mathcal E)$.
This is so in case Floer's boundary operator is nontrivial.
\par
We 
use $\Z_2$ coefficient here since we use it in the main theorems 
of this paper. However we can certainly definite $HF(M,\mathcal E;\Lambda_0^{\Z})$
since orientation and sigh is fully worked out in 
gauge theory Floer homology by Floer.
\item
In case $R(M)$ is not a finite set (or is not Fredholm regular), 
we need to perturb. The independence of 
$HF(M,\mathcal E;\Lambda_0^{\Z_2})$ of the perturbation does {\it not} hold.
This is similar to the following fact:
The Floer homology of a pair of Lagrangian submanifolds
$HF((L_1,b_1),(L_2,b_2);\Lambda_0)$ over $\Lambda_0$ coefficient
is {\it not} an invariant of Hamiltonian perturbation of $L_1$, $L_2$.
(Namely for a pair of Hamiltonian diffeomorphisms $\varphi_1, \varphi_2$ 
of our symplectic manifold, the isomorphism 
$$
HF((\varphi_1(L_1),(\varphi_1)_*(b_1)),(\varphi_2(L_2),(\varphi_2)_*(b_2));\Lambda_0)
\cong
HF((L_1,b_1),(L_2,b_2);\Lambda_0)
$$
is false.)
On the other hand, if we use $\Lambda$ instead of $\Lambda_0$
as a coefficient ring then  $HF((L_1,b_1),(L_2,b_2);\Lambda)$  becomes invariant 
of Hamiltonian isotopy.
(See \cite[Theorem 4.1.4]{fooobook}  and \cite[Theorem 4.1.5]{fooobook}.)
So there is an issue for the well-defined-ness of 
$HF(M,\mathcal E;\Lambda_0)$ in case $R(M)$ is not
Fredholm regular. Nevertheless we can use the argument of 
\cite[Section 6.5.4]{fooobook} to define 
$HF(M,\mathcal E;\Lambda^F_0)$ also in case $R(M)$ is not
Fredholm regular, if $F$ is a field.
\par
The Floer homology $HF(M,\mathcal E;\Lambda^F_0)$ in general 
contain a torsion subgroup such as $\Lambda^F_0/T^{\lambda}\Lambda^F_0$
as its direct factor. It is not clear for the author whether such components 
can be applicable to the study of topology of 3 or 4 manifolds. 
\end{enumerate}
\end{rem}
Let us go back to the proof of Theorem \ref{mainthm} (2).
The chain complex in the right hand side of (\ref{form51}) 
is one defined by 
(\ref{form55}) and (\ref{form56}).
The main part of the construction of 
(\ref{form51})  is the definition of the moduli space 
we use for that purpose.
\par
We take a domain $W \subset \C$ such that
the following holds. (See Figure 5.1.)
\begin{conds}\label{conds54}
\begin{enumerate}
\item 
The intersection 
$W \cap \{z \in \C \mid {\rm Im} z < -2\}$
is $\{z \in \C \mid \vert{\rm Re}z \vert \le 1, {\rm Im} z < -2 \}$.
The intersection 
$W \cap \{z \in \C \mid {\rm Im} z > +2\}$
is $\{z \in \C \mid \vert{\rm Re}z \vert \le 1, {\rm Im} z > +2 \}$.
\item
The intersection 
$W \cap \{z \in \C \mid {\rm Im} z > +2\}$
is $\{z \in \C \mid \vert{\rm Re} z\vert \le 1, {\rm Im} z > +2 \}$.
The intersection 
$W \cap \{z \in \C \mid {\rm Im} z < -2\}$
is $\{z \in \C \mid \vert{\rm Re} z\vert \le 1, {\rm Im} z < -2 \}$.
\item
The boundary $\partial W$ has four connected components 
$\partial_i W$ ($i=1,2,3,4$) each of which is a $C^{\infty}$ submanifold of $\C$ 
and is diffeomorphic to 
$\R$.
Moreover
$
\partial_1W = \{z \in \C \mid {\rm Re} z > 0, {\rm Im}z>0\}
$,
$
\partial_2W 
\subset \{z \in \C \mid {\rm Re} z < 0, {\rm Im}z>0\}
$, 
$
\partial_3W 
\subset \{z \in \C \mid {\rm Re} z < 0, {\rm Im}z<0\}
$.
$
\partial_4W 
\subset \{z \in \C \mid {\rm Re} z > 0, {\rm Im}z<0\}
$.
\end{enumerate}
\end{conds}
\begin{center}
\includegraphics[scale=0.25]
{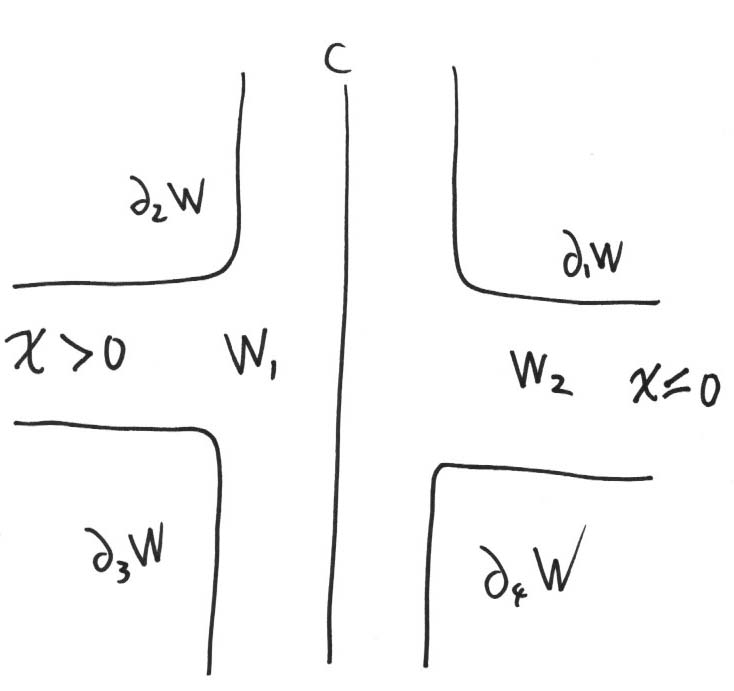}
\end{center}
\centerline{\bf Figure 5.1}
\par\medskip
We put $C = \{z \mid {\rm Re}z = 0\} 
\subset W$
and 
$$
W_1 = \{z \in W \mid {\rm Re}z  < 0\},
\qquad 
W_2 = \{z \in W \mid {\rm Re}z  > 0\}.
$$
We take $\chi : W \to [0,1]$ and a Riemannian metric 
${\bf g}$ on $\Sigma \times W_1$ with the following 
properties.
(See Figure 5.2.)
\begin{conds}\label{cond55}
\begin{enumerate}
\item
$\{z \in \C \mid \chi(z) > 0\} = W_1$.
\item
On  $\{z \in W \mid \vert{\rm Im} z \vert> -3\}$, 
$\chi(z) = \chi({\rm Re} z)$, where $\chi$ in the right hand side is the same 
function as one appeared in (\ref{form211}).
${\bf g} = \chi^2 g_{\Sigma} + ds^2 + dt^2$ 
on $\{z \in W \mid \vert{\rm Im} z \vert> -3\}$, 
where we put $z = s + \sqrt{-1}t$.
\item
On $\{z \in W \mid {\rm Re} z < -3\}$, 
$\chi(z) = 1$.
${\bf g} = g_{\Sigma} + ds^2 + dt^2$ 
on $\Sigma \times \{z \in W \mid {\rm Re} z < -3\}$, 
where we put $z = s + \sqrt{-1}t$.
\item
In a neighborhood of $\Sigma \times \partial_2 W$ 
(resp. $\Sigma \times \partial_3 W$),
the space $\Sigma \times W_1$ with metric ${\bf g}$ is isometric to the direct product
$g_{\Sigma} \times (0,\epsilon) \times \R$.
Here $g_{\Sigma} \times \{0\} \times \R$
corresponds to $\Sigma \times \partial_2 W$ (resp. $\Sigma \times \partial_3 W$).
This isometry is compatible with the isometry obtained by 
items (2)(3) in the domain described by those items.
\item 
On a $(-\epsilon,0) \times \R$,
$\chi(z) = \chi({\rm Re}z)$, where $\chi$ in the right hand side is the same 
function as one appeared in (\ref{form211}).
On $\Sigma\times (-\epsilon,0) \times \R$ we have
${\bf g} = \chi^2 g_{\Sigma} + ds^2 + dt^2$ 
where we put $z = s + \sqrt{-1}t$.
\end{enumerate}
\end{conds}
\par
\begin{center}
\includegraphics[scale=0.25]
{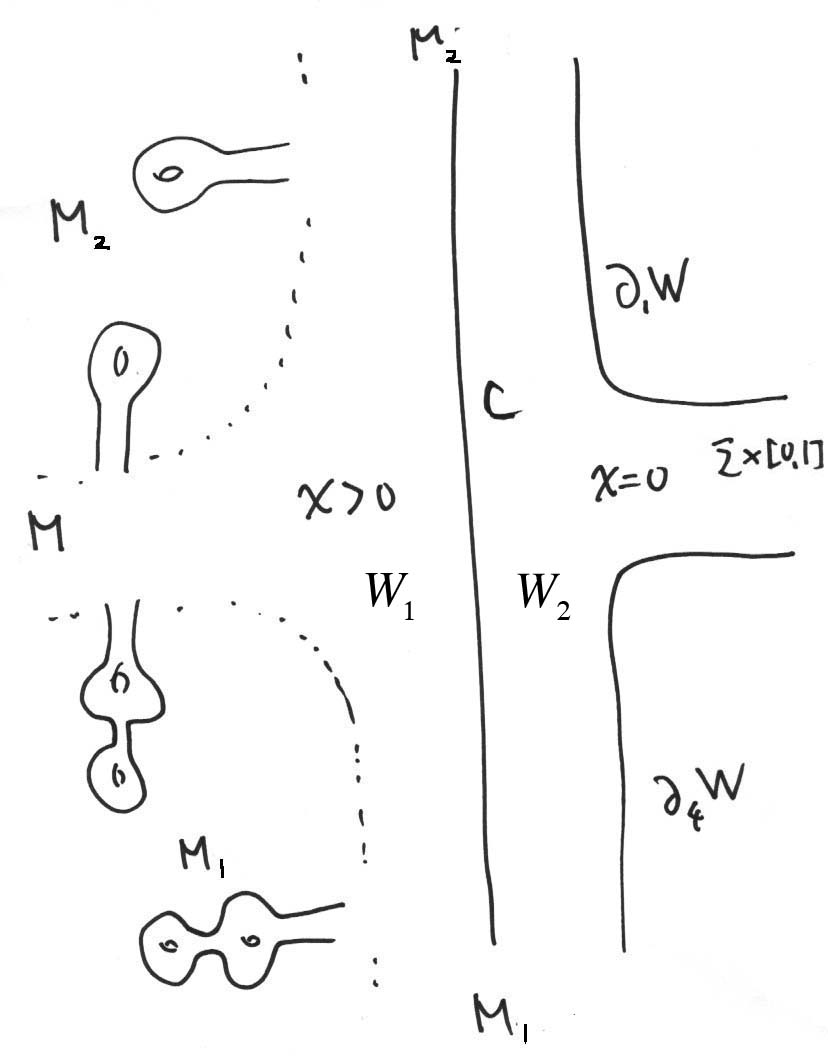}
\end{center}
\centerline{\bf Figure 5.2}
\par\medskip
We extend the metric ${\bf g}$ on $\Sigma \times W_1$ to 
a `singular metric' on $\Sigma \times W$ by putting 
${\bf g} = 0 g_{\Sigma} + ds^2 + dt^2$ outside $\Sigma \times W_1$.
\par
By Condition \ref{cond55} $(\Sigma \times W_1,{\bf g})$
is isometric to $(\Sigma \times (0,\epsilon) \times \R,g_{\Sigma} + ds^2 + dt^2)$
at a neighborhood of $\partial_2W$ (resp. $\partial_3W$).
We remark that $M_{2,0}\times \R$ (resp.  $M_{1,0}\times \R$) is isometric to 
$(\Sigma \times (-\epsilon,0) \times \R,g_{\Sigma} + ds^2 + dt^2)$
at a neighborhood of its boundary.
Here $M_{1,0} = M_1 \setminus (\Sigma \times (-1,1])$, 
(resp. $-M_{2,0} = -M_2 \setminus (\Sigma \times (-1,1])$).
Therefore we can glue them together to obtain 
$(Y,{\bf g})$. Here $\bf g$ is a `Riemannian metric' which is degenerate
on $\Sigma \times \overline{W_2}$.
\par
We remark that $Y$ has $4$ ends and $2$ boundary components.
The $4$ ends corresponds to ${\rm Re} z \to \pm \infty$ 
and  ${\rm Im} z \to \pm \infty$.
The end corresponding to ${\rm Re} z \to + \infty$ is 
of the form $\Sigma \times [-1,1] \times [c,\infty)$.
The end corresponding to ${\rm Im} z \to + \infty$
is of the form $-M_2 \times [-1,1] \times [c,\infty)$.
The end ${\rm Re} z \to - \infty$ is of the form 
$M \times (-\infty,-c]$.
The end  corresponding to ${\rm Im} z \to - \infty$ is 
of the form $M_1 \times (-\infty,-c]$.
\par
The boundaries are $\Sigma \times \partial_1W$ and 
$\Sigma \times \partial_4W$.
We remark that we glued $M_{2,0} \times \R$
(resp. $M_{1,0} \times \R$) to $\Sigma \times \partial_2W$
(resp. $\Sigma \times \partial_3W$). So 
$\Sigma \times \partial_2W$ and $\Sigma \times \partial_3W$
are not  boundary components of $Y$.
\par
The $SO(3)$ bundles $\mathcal E_1,\mathcal E_2$ on $M_1, -M_2$ induce an
$SO(3)$ bundle on $Y$  in an obvious way, which we denote by $\mathcal E_Y$.
For a smooth connection $\frak A$ of $\mathcal E_Y$ we can 
consider the `ASD-equation'. Namely we require (\ref{ASDeq})
on $Y \setminus (\Sigma \times \overline{W_2})$
and (\ref{ASDprod}) on $\Sigma \times W \subset Y$.
(Note (\ref{ASDeq}) coincides with (\ref{ASDprod}) on the overlapped part.)
We say $\frak A$ is an ASD-connection by an abuse of notation 
if it satisfies (\ref{ASDeq})
on $Y \setminus (\Sigma \times \overline{W_2})$
and (\ref{ASDprod}) on $\Sigma \times W \subset Y$.
\par
We consider also $\Omega$ which contains $W_2$.
Namely $\Omega$ is a union of $W_2$
and trees of disk and sphere components attached 
to $\partial W_2$ and ${\rm Int} W_2$, respectively.
We consider the pair $(\Omega,u)$
which satisfies Condition \ref{conds27}, except we replace 
$(0,1)\times \R$  by
$W_2$.
We call this condition Condition \ref{conds27}''.
\par
Now we modify Definitions \ref{defn2929}, \ref{defn23} and
\ref{defn4664} as follows.
We consider the decompositions
(\ref{form236}) and  (\ref{form235}).  Let $I(R(M_j))$  be 
the index sets as in there.
Namely
$$
\aligned
R(M_j)\times_{R(\Sigma)} R(M_j) &= \bigcup_{k\in I(R(M_j))} R(M_j)_k 
\qquad j=1,2\\
R(M_1) \times_{R(\Sigma)} R(M_2) &= R(M).
\endaligned
$$
Let $R_k$ $j=1,2$ be connected components of 
$R(M_j) \times_{R(\Sigma)}R(M_j)$ and 
$a_-,a_+ \in R(M) = R(M_1) \cap R(M_2)$.
\par
We also take $\vec i^{(1)}$,  $\vec i^{(2)}$
as in Definition \ref{defn4664} (5) below.
\begin{defn}\label{defn4664}
We define the set $\overset{\circ}{{\mathcal M}}((Y,\mathcal E_Y);a_-,a_+;R_1,R_2;
\vec i^{(1)},\vec i^{(2)};E)$
as the set of all equivalence classes of $(\frak A,{
\frak z},{\frak w},\Omega,u,\vec z^{(1)},\vec z^{(2)})$ satisfying the following conditions.
(See Figure 5.3.)
\begin{enumerate}
\item
$\frak A$ is a connection of $\mathcal E_{Y}$ satisfying equations  
 (\ref{ASDeq}), (\ref{ASDprod}).
\item
$\frak z = (\frak z_1,\dots,\frak z_{m_1})$ is an {\it unordered} 
$m_1$-tuple of points of $Y \setminus (\Sigma \times \overline{W_2})$.
We put $\Vert \frak z\Vert = m_1$.
We say the subset $\{\frak z_1,\dots,\frak z_{m_1}\}
\subset Y \setminus (\Sigma \times \overline{W_2})$ the {\it support} of $\frak z$ 
and denote it by $\vert\frak z\vert$.
We define ${\rm multi} : \vert\frak z\vert \to \Z_{>0}$ by 
$
{\rm multi}(x) = \#\{i \mid z_i = x\}$
and call it the {\it multiplicity function}.
\item
$\frak w = (\frak w_1,\dots,\frak w_{m_2})$ is an {\it unordered} 
$m_2$-tuple of points of $C$.
We put $\Vert \frak w\Vert = m_2$.
We say the subset $\{\frak w_1,\dots,\frak w_{m_2}\} \subset C$ the {\it support} of $\frak w$.
We define ${\rm multi} : \vert\frak w\vert \to \Z_{>0}$ by 
$
{\rm multi}(x) = \#\{i \mid w_i = x\}$
and call it the {\it multiplicity function}.
\item
$\Omega$ satisfies Condition \ref{conds27}''.
\item
$\vec i^{(1)} = ( i^{(1)}(1),\dots,i^{(1)}(k_1)) 
\in I(R(M_1))^{k_1}$ and  $\vec i^{(2)} = ( i^{(2)}(1),\dots,i^{(2)}(k_2)) 
\in I(R(M_2))^{k_2}$
\item
$\vec z^{(1)} = (z_1^{(1)},\dots,z_{k_1}^{(1)})$
(resp. $\vec z^{(2)} = (z_1^{(2)},\dots,z_{k_2}^{(2)})$).
$z_i^{(1)}$ lies on $\partial_4 \Omega$,
(resp. $z_i^{(2)}$ lies on $\partial_1 \Omega$).
None of $z^{(1)}_i$ or $z^{(2)}_i$ is a nodal point. 
If $i \ne j$ then $z^{(1)}_i \ne z^{(1)}_j$, $z^{(2)}_i \ne z^{(2)}_j$.
$(z^{(1)}_1,\dots,z^{(1)}_{k_1})$ 
(resp. $(z^{(2)}_1,\dots,z^{(2)}_{k_2})$ ) 
respects counter clockwise orientation of 
$\partial_4 \Omega$ (resp. $\partial_1 \Omega$).
\item
There exists a smooth map $\gamma^{(1)} : \partial_4 \Omega \setminus \{z^{(1)}_1,
\dots,z^{(1)}_{k_1}\}
\to 
R(M_1)$
such that $u(z) = i_{R(M_1)}(\gamma(z))$ on $\partial_4 W \setminus \{z^{(1)}_1,
\dots,z^{(1)}_{k_1}\}$.
\par
There exists a smooth map $\gamma^{(2)} : \partial_1 \Omega \setminus \{z^{(2)}_1,
\dots,z^{(2)}_{k_2}\}
\to 
R(M_2)$
such that $u(z) = i_{R(M_2)}(\gamma(z))$ on $\partial_4 W \setminus \{z^{(2)}_1,
\dots,z^{(2)}_{k_2}\}$.
\item
For $j=1,\dots,k_1$ the following holds. 
\begin{equation}\label{form58777rev}
(\lim_{z \uparrow z^{(1)}_j}\gamma^{(1)}(z),\lim_{z\downarrow z^{(1)}_j}\gamma^{(1)}(z))
\in \widehat{R(M_1)}_{i^{(1)}(j)}.
\end{equation}
Here the notation $z \uparrow z_j$, $z \downarrow z_j$ is 
defined in the same way as (\ref{form23777})\par
For $j=1,\dots,k_2$ the following holds. 
\begin{equation}\label{form58777revrev}
(\lim_{z \uparrow z^{(2)}_j}\gamma^{(2)}(z),\lim_{z\downarrow z^{(2)}_j}\gamma^{(2)}(z))
\in \widehat{R(M_2)}_{i^{(2)}(j)}.
\end{equation}
Here the notation $z \uparrow z_j$, $z \downarrow z_j$ is 
defined in the same way as (\ref{form23777})
\item
We replace Condition \ref{conds28} (3)
by the stability of 
$(\Omega,u,\vec z^{(1)},\vec z^{(2)})$. Namely the set of all maps $v : \Omega \to \Omega$ 
satisfying the next three conditions is a finite set.
\begin{enumerate}
\item
$v$ is a homeomorphism and is holomorphic on each of the irreducible components.
\item
$v$ is the identity map on $(0,1] \times \R \subseteq \Omega$.
\item $u \circ v = u$.
\item $v(z^{(1)}_j) = z^{(1)}_j$, $j=1,\dots,k_1$ 
and $v(z^{(2)}_j) = z^{(2)}_j$, $j=1,\dots,k_3$. 
\end{enumerate}
\item
For $(s,t) \in W_2$
we have
$$
[A(s,t)] = u(s,t). 
$$
Here $A(s,t)$ is obtained from $\frak A$ by  (\ref{AnadfraA}).
\item
The energy of $(\frak A,{
\frak z},{\frak w},\Omega,u)$
which is defined in the same way as Definition \ref{defnenergy2} is $E$.
\item
We assume the following 
asymptotic boundary conditions, which are defined by using $R_1,R_2,a_-,a_+$.
\begin{enumerate}
\item
\begin{equation}\label{form58777revrevrev}
\lim_{z \to +\infty- \sqrt{-1}}\gamma^{(4)}(z) = \lim_{z\to 
+\infty+\sqrt{-1}}\gamma^{(1)}(z)
= a_+.
\end{equation}
Here $\lim_{z \to +\infty- \sqrt{-1}}$ is the limit 
when the real part of $z \in \partial_4 W$ goes to $+\infty$.
The meaning of 
$ \lim_{z\to 
+\infty+\sqrt{-1}}$ is similar.
\item
We consider the restriction of $\frak A$ to $\Sigma \times \{z \in W \mid {\rm Re} z = -c\}$
for $c > 3$. We glue it with the restriction $\frak A$ to 
$(M_1)_0 = M_1 \setminus (\Sigma \times (-1,1])$
and $(M_2)_0 = M_2 \setminus (\Sigma \times (-1,1])$
which are attached to $\Sigma \times \{-c -\sqrt{-1}\}$ and 
$\Sigma \times \{-c +\sqrt{-1}\}$ respectively.
(See Figure 5.4.)
We call it $\frak A\vert_{{\rm Re} z = -c}$.
It is a connection on $M = M_1 \cup -M_2$.
We assume that $\frak A\vert_{{\rm Re} z = -c}$ converges to a 
flat connection $a_-$ as $c \to \infty$. 
\item
We consider the restriction of $\frak A$ to $\Sigma \times \{z \in W\mid {\rm Im} z = c\}$
for $c >3$. We glue it with the restriction $\frak A$ to 
$-M_2 = -M_2 \setminus (\Sigma \times (-1,1])$, which is attached to $-1 + c\sqrt{-1}$.
(See Figure 5.5.) We call it $\frak A\vert_{{\rm Im} z = c}$.
It is a connection of $-M_2$.
We assume that $\frak A\vert_{{\rm Im} z = c}$ converges to a 
flat connection as $c \to + \infty$. We write its limit
$\lim_{c\to+\infty} \frak A\vert_{{\rm Im} z = c}$.
\par
Then we also assume
\begin{equation}\label{form53777revrevrev2}
(\lim_{c\to+\infty} \frak A\vert_{{\rm Im} z = c},\lim_{z\to 
+1+\infty\sqrt{-1}}\gamma^{(1)}(z)
\in R_2.
\end{equation}
Here the meaning of $\lim_{z\to
+1+\infty\sqrt{-1}}$ is similar to 
(\ref{form58777revrevrev}).
\item
We consider the restriction of $\frak A$ to $\Sigma \times \{z \mid {\rm Im} z = -c\}$
for $c > 3$. We glue it with the restriction $\frak A$ to 
$(M_1)_0$ which is attached to $-c + \sqrt{-1}$.
(See Figure 5.6.) We call it $\frak A\vert_{{\rm Im} z = -c}$.
It is a connection of $M_1$.
We assume that $\frak A\vert_{{\rm Im} z = -c}$ converges to a 
flat connection as $c \to \infty$. We write its limit
$\lim_{c\to\infty} \frak A\vert_{{\rm Im} z = -c}$.
\par
Then we also assume
\begin{equation}\label{form58777revrevrev3}
(\lim_{c\to\infty} \frak A\vert_{{\rm Im} z = -c},\lim_{z\to 
-\infty -\sqrt{-1}}\gamma^{(3)}(z))
\in R_1.
\end{equation}
Here the meaning of $\lim_{z\to 
-\infty -\sqrt{-1}}$ is similar to 
(\ref{form58777revrevrev}).
\end{enumerate}
\end{enumerate}
The equivalence relation is defined in the same way as Definition \ref{equivlimiobj}.
\end{defn}
\par
\begin{center}
\includegraphics[scale=0.25]
{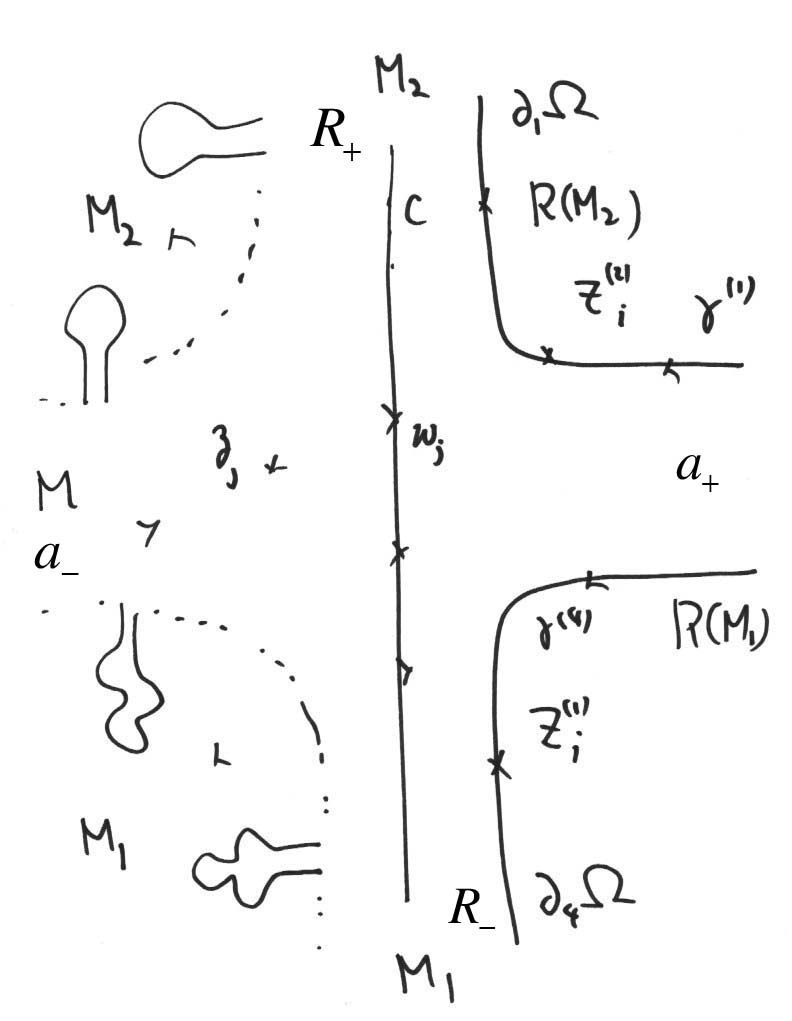}
\end{center}
\centerline{\bf Figure 5.3}
\par
\begin{center}
\includegraphics[scale=0.25]
{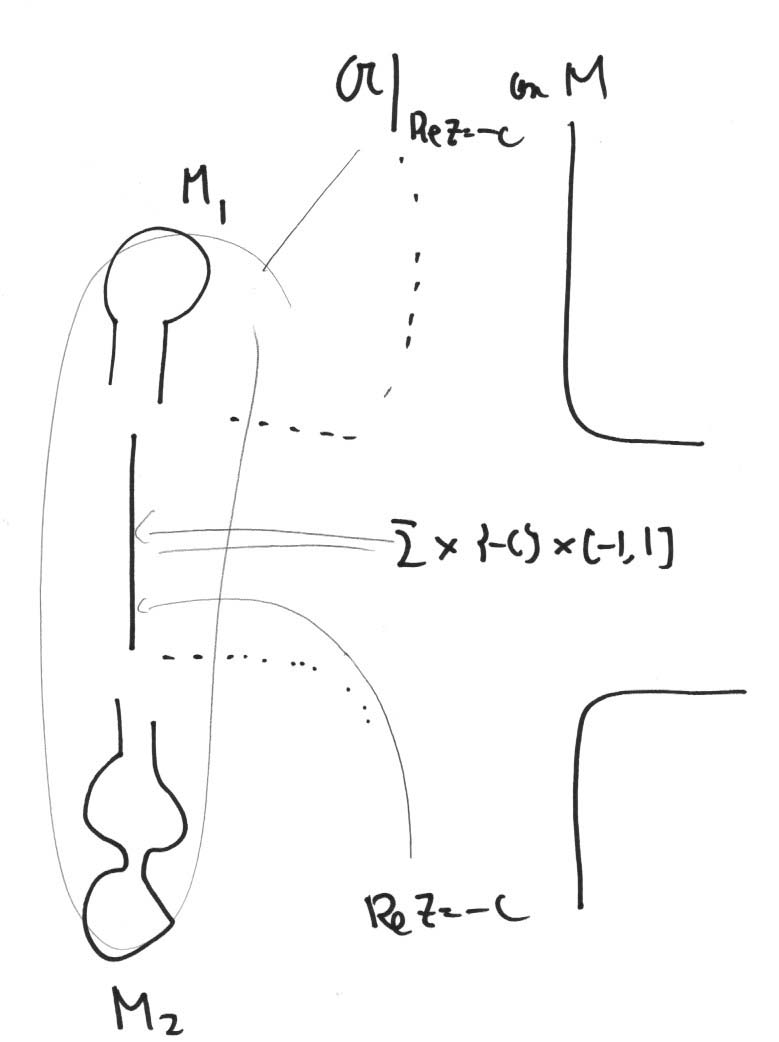}
\end{center}
\centerline{\bf Figure 5.4}
\par
\begin{center}
\includegraphics[scale=0.25]
{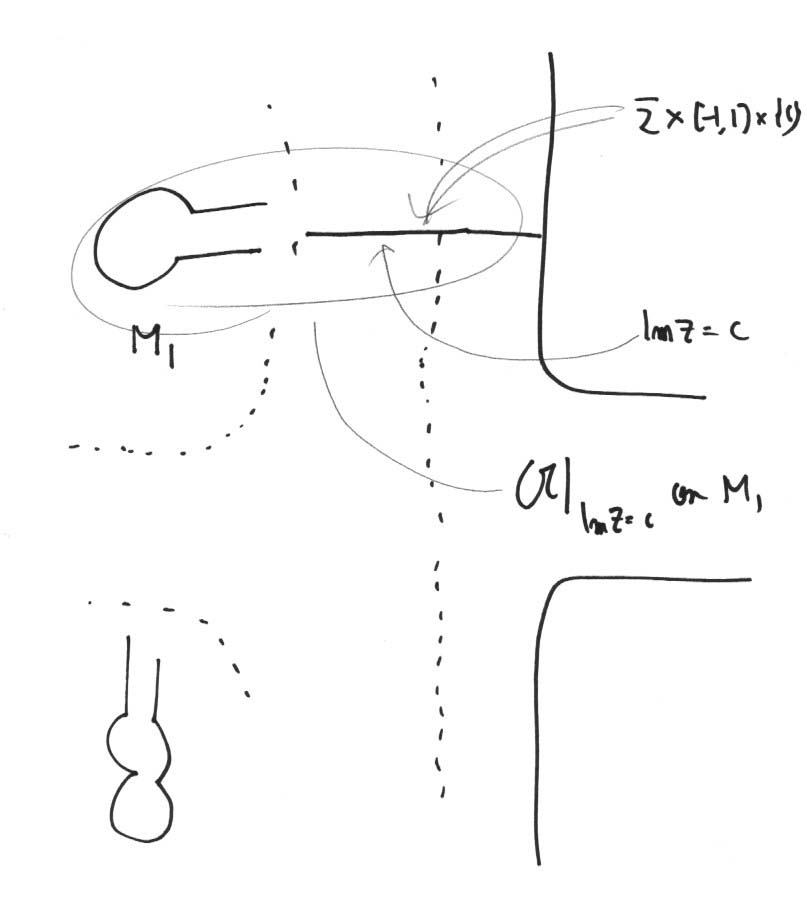}
\end{center}\centerline{\bf Figure 5.5}
\par
\begin{center}
\includegraphics[scale=0.25]
{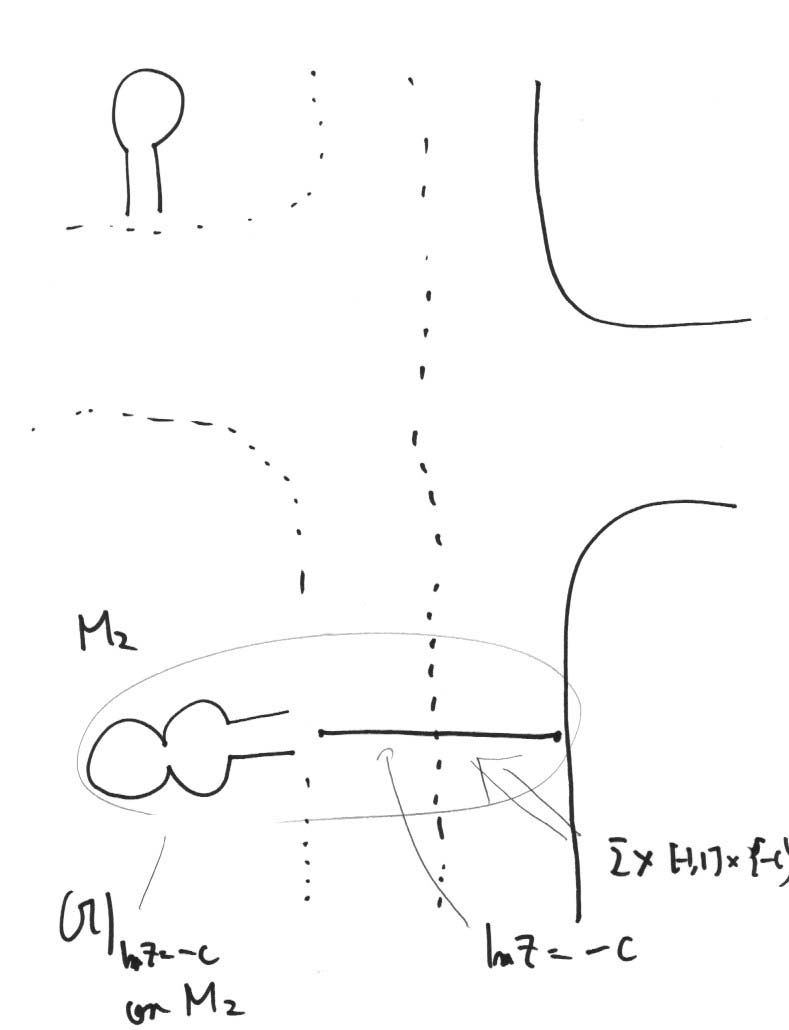}
\end{center}
\centerline{\bf Figure 5.6}
\par
We can define a topology on
$\overset{\circ}{{\mathcal M}}((Y,\mathcal E_Y);a_-,a_+;R_1,R_2;
\vec i^{(1)},\vec i^{(2)};E)$ by modifying Definition \ref{defn1222}
in an obvious way.
\par
We put
$$
\aligned
&\overset{\circ}{{\mathcal M}}_{k_1,k_2}((Y,\mathcal E_Y);a_-,a_+;R_1,R_2;E) \\
&=
\bigcup_{\vec i^{(1)}; \vert \vec i^{(1)} \vert = k_1}
\bigcup_{\vec i^{(2)}; \vert \vec i^{(2)} \vert = k_2}
\overset{\circ}{{\mathcal M}}((Y,\mathcal E_Y);a_-,a_+;R_1,R_2;
\vec i^{(1)},\vec i^{(2)};E)
\endaligned
$$
$\overset{\circ}{{\mathcal M}}_{k_1,k_2}((Y,\mathcal E_Y);a_-,a_+;R_1,R_2;E)$ is a Hausdorff space.
\par
We define evaluation maps
\begin{equation}
\aligned
&{\rm ev}^{(1)}  :
\overset{\circ}{{\mathcal M}}_{k_1,k_2}((Y,\mathcal E_Y);a_-,a_+;R_1,R_2;E)
\to (R(M_1) \times_{R(\Sigma)} R(M_1))^{k_1}
\\
&{\rm ev}^{(2)}  :
\overset{\circ}{{\mathcal M}}_{k_1,k_2}((Y,\mathcal E_Y);a_-,a_+;R_1,R_2;E)
\to (R(M_2) \times_{R(\Sigma)} R(M_2))^{k_2}.
\endaligned
\end{equation}
They are defined by (\ref{form58777rev}) and (\ref{form58777revrev}), respectively.
\par
We also define the evaluation maps
\begin{equation}
{\rm ev}^{\infty}_i :
\overset{\circ}{{\mathcal M}}_{k_1,k_2}((Y,\mathcal E_Y);a_-,a_+;R_1,R_2;E)
\to R_i
\end{equation}
for $i=1,2$. They are defined by 
(\ref{form58777revrevrev3}) and
(\ref{form53777revrevrev2}).
\par
Note $\overset{\circ}{{\mathcal M}}_{k_1,k_2}((Y,\mathcal E_Y);a_-,a_+;R_1,R_2;E)$
is not yet compact. There are still four types ends, which are,
\begin{enumerate}
\item[(I)]
A pseudo-holomorphic strip escape to the direction ${\rm Re}(z) \to +\infty$.
\item[(II)]
An ASD-connection escape to the direction ${\rm Im}(z) \to +\infty$.
\item[(III)]
An ASD-connection escape to the direction ${\rm Re}(z) \to -\infty$.
\item[(IV)]
An ASD-connection escape to the direction ${\rm Im}(z) \to -\infty$.
\end{enumerate}
By definition, ends of type (I) correspond to the union of the following direct products.
\begin{equation}\label{5527427}
\aligned
&\overset{\circ}{{\mathcal M}}_{k'_1,k'_2}((Y,\mathcal E_Y);a_-,a'_+;R_1,R_2;E_1) \\
&\times 
{{\mathcal M}}_{k''_1,k''_2}(R(M_1),R(M_2);\{a'_+\},\{a_+\};E_2).
\endaligned
\end{equation}
Here $k'_1 + k''_1 = k_1$, $k'_2 + k''_2 = k_2$, $E_1 + E_2 = E$ and 
$a'_+ \in R(M_1) \cap R(M_2)$.
(Figure 5.7)
\par
The moduli space 
${{\mathcal M}}_{k''_1,k''_2}(R(M_2),R(M_1);\{a'_+\},\{a_+\};E_2)$ is 
defined as the union of (\ref{fibercompactkarefref}).
Note in our case we assume $R(M_1)$ is transversal to $R(M_2)$.
Therefore each of the connected components of $R(M_1) \cap R(M_2)$
consists of a single point.
\par
\begin{center}
\includegraphics[scale=0.25]
{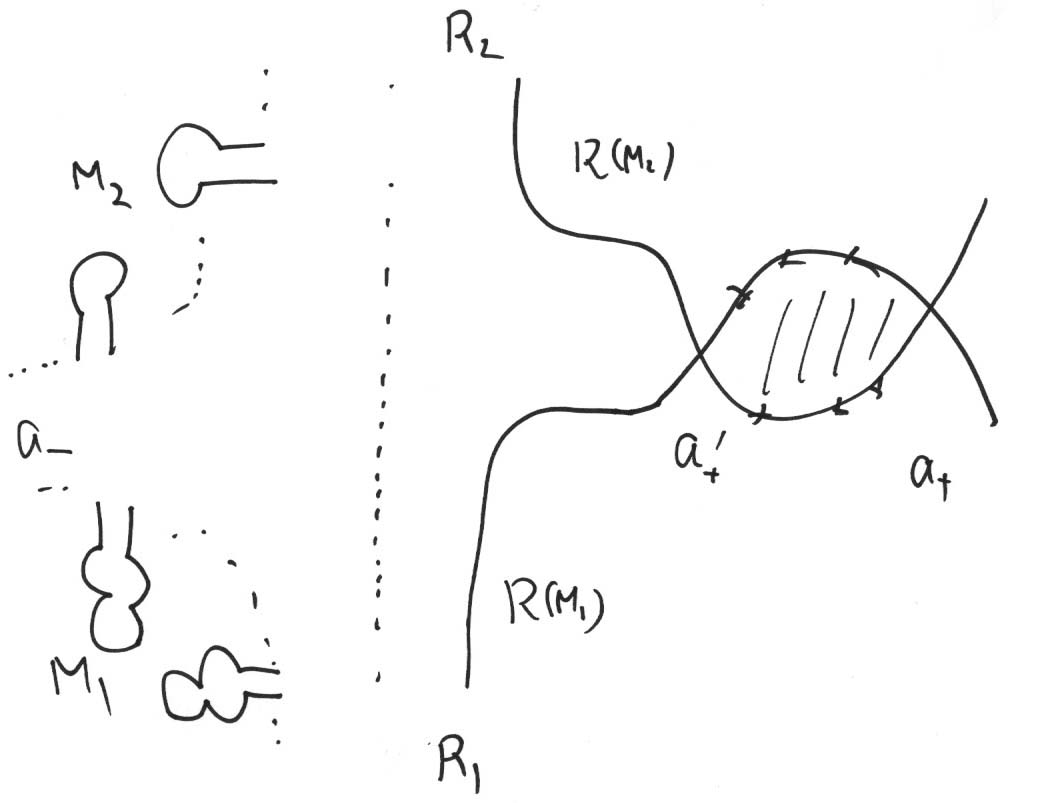}
\end{center}
\centerline{\bf Figure 5.7}
\par\medskip
By definition, ends of type (III) correspond to the union of the following direct products.
\begin{equation}\label{54444488}
\overset{\circ}{{\mathcal M}}_{k_1,k_2}((Y,\mathcal E_Y);a_+,a'_-,;R_1,R_2;E_1) \\
\times 
{{\mathcal M}}(a'_-,a_-;E_2).
\end{equation}
Here $E_1+E_2 = E$ and $a'_- \in R(M) = R(M_1) \cap R(M_2)$.
The moduli space ${{\mathcal M}}(a'_-,a_-;E_2)$ is defined as the 
union of (\ref{form5353}) 
(Figure 5.8).
Note we need to rotate the bubble appearing at the part ${\rm Re} z \to 
-\infty$ by 90 degree clock-wise direction. Therefore 
after rotation $a'_-$ will appear in the part ${\rm Im} z \to - \infty$.
So $a'_-$ and $a_-$ appears as  $a'_-,a_-$ 
in the second factor of (\ref{54444488}).
\par
\begin{center}
\includegraphics[scale=0.25]
{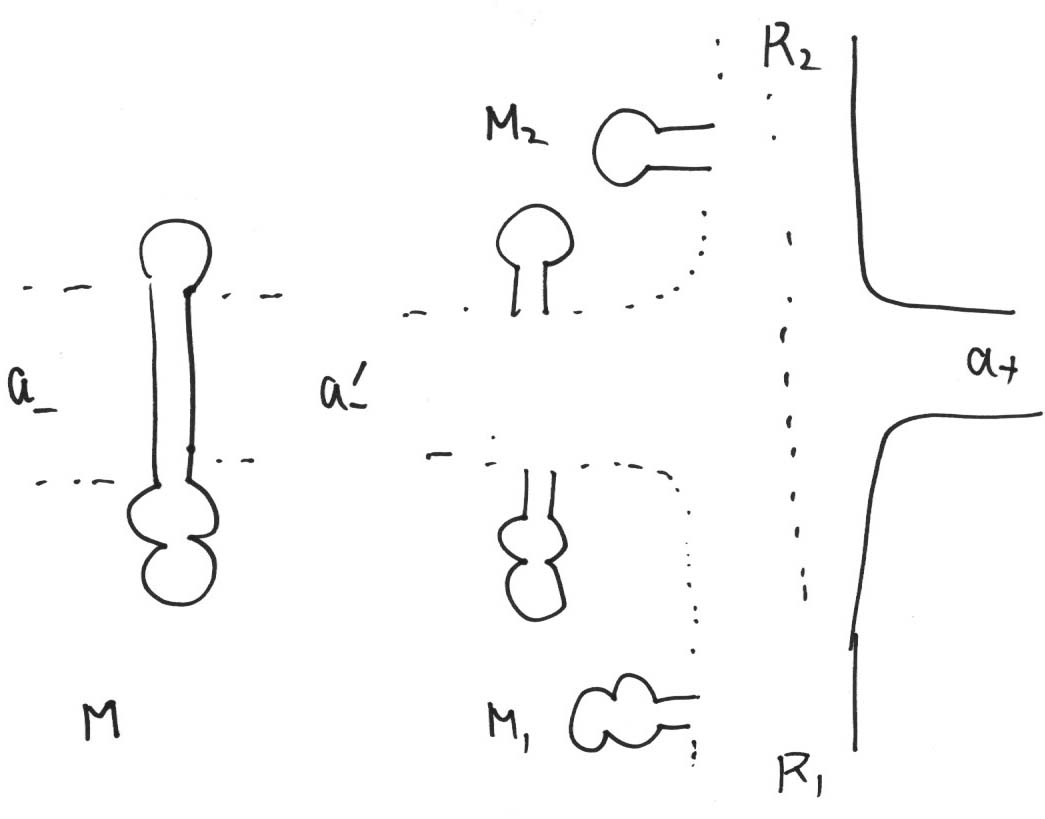}
\end{center}
\centerline{\bf Figure 5.8}
\par\medskip
We next describe the ends of type (II).
For this purpose we use the moduli space 
${{\mathcal M}}_k((M,\mathcal E),L;R_-,R_+;E)$ defined in Definition \ref{defn23}.
This moduli space is defined as the set of solutions of a 
partial differential equation on a space which is obtained from a 3-manifold with boundary, which was 
denoted by $M$ in Section \ref{HF3bdfunctor}.
Here we consider either $M_1$ or $-M_2$ as a 3-manifold $M$ with boundary.
So we use the notations ${{\mathcal M}}_k((M_1,\mathcal E_1),L;R_-,R_+;E)$
or  ${{\mathcal M}}_k((-M_2,\mathcal E_2),L;R_-,R_+;E)$.
(We also take $L=R(M_1)$ or $L = R(M_2)$.)
\par
Now the ends of type (II) is described by the fiber product:
\begin{equation}\label{5524444edcd}
\aligned
&\overset{\circ}{{\mathcal M}}_{k_1,k'_2}((Y,\mathcal E_Y);a_-,a_+;R_1,R'_2;E_1) 
\\
&{}_{{\rm ev}^{\infty}_2}\times_{{\rm ev}_-} 
{{\mathcal M}}_{k''_2}((-M_2,\mathcal E_2),R(M_2);R'_2,R_2;E_2).
\endaligned
\end{equation}
Here $k_2 = k'_2 + k''_2$, $E = E_1 + E_2$ and 
$R'_2$ is a connected component of $R(M_2) \times_{R(\Sigma)} R(M_2)$. 
(Figure 5.9)
\par
\begin{center}
\includegraphics[scale=0.25]
{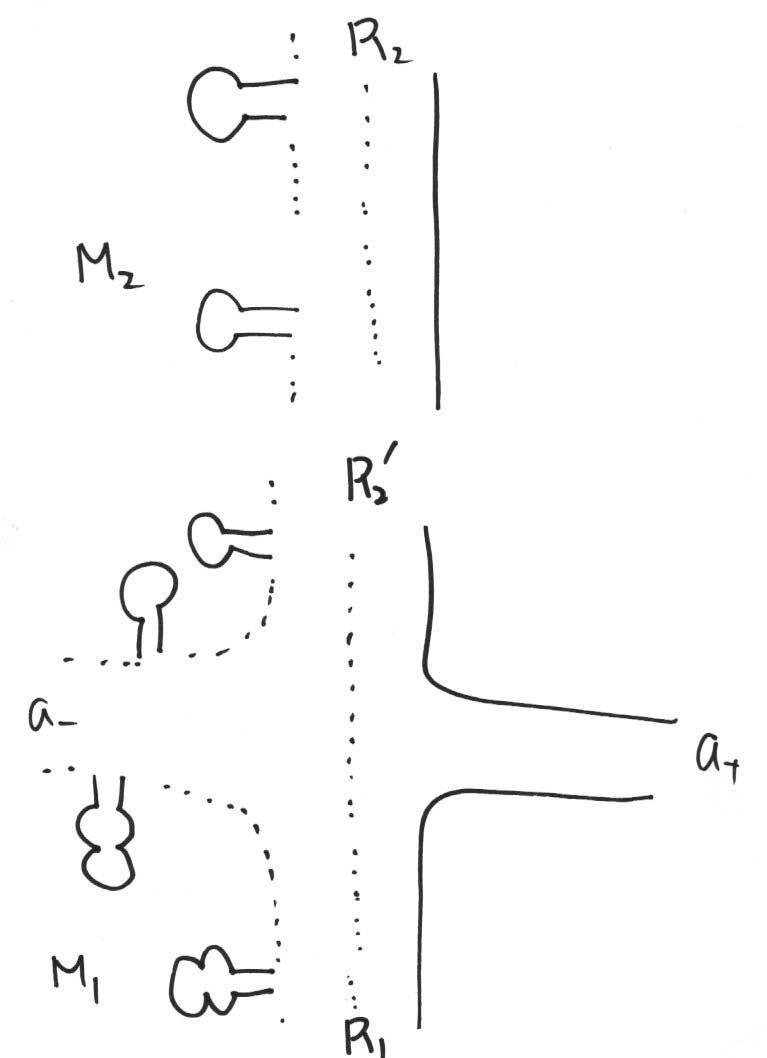}
\end{center}
\centerline{\bf Figure 5.9}
\par\medskip
Similarly the ends of type (II) is described by the fiber product:
\begin{equation}\label{5524444edcd2}
\aligned
&\overset{\circ}{{\mathcal M}}_{k'_1,k_2}((Y,\mathcal E_Y);a_-,a_+;R'_1,R_2;E_1) \\
&{}_{{\rm ev}^{\infty}_1}\times_{{\rm ev}_+} 
{{\mathcal M}}_{k''_1}((M_1,\mathcal E_1),R(M_1);R_1,R'_1;E_2).
\endaligned
\end{equation}
Here $k_1 = k'_1 + k''_1$, $E = E_1 + E_2$ and 
$R'_1$ is a connected component of $R(M_1) \times_{R(\Sigma)} R(M_1)$. 
(Figure 5.10)
\par
\begin{center}
\includegraphics[scale=0.25]
{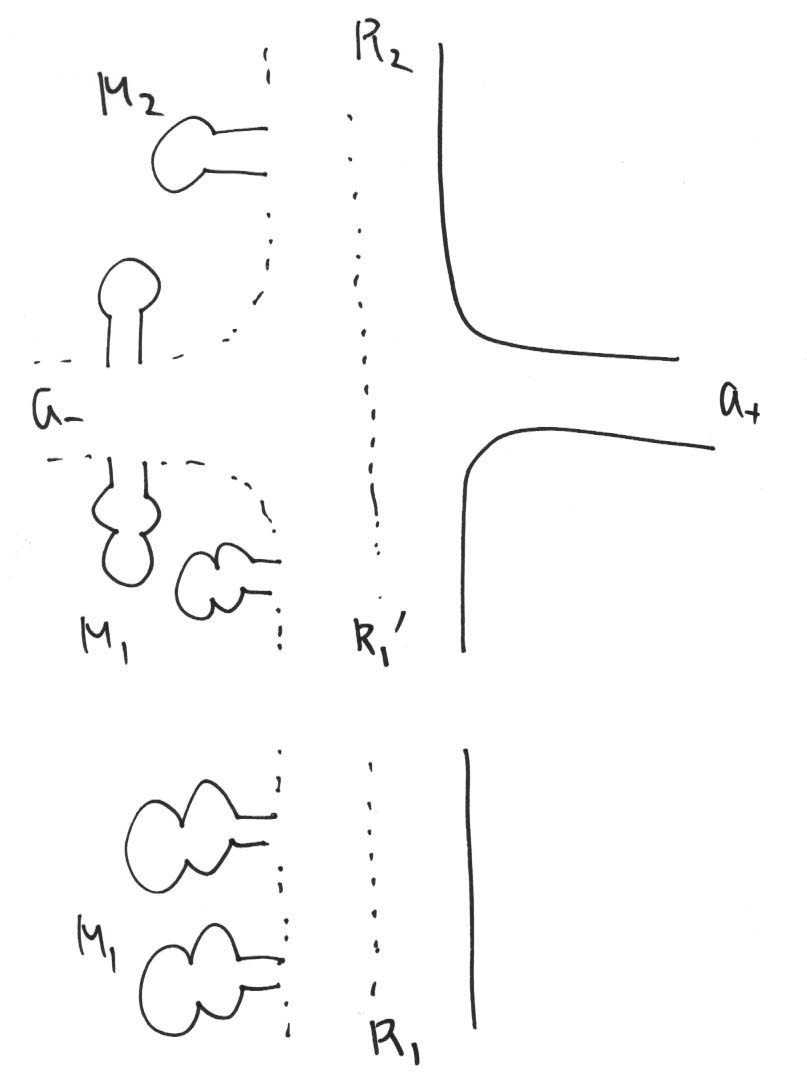}
\end{center}
\centerline{\bf Figure 5.10}
\par
\begin{prop}\label{prop5858}
We can compactify 
$\overset{\circ}{{\mathcal M}}_{k_1,k_2}((Y,\mathcal E_Y);a_-,a_+;R_1,R_2;E)$
to a compact Hausdorff space 
${{\mathcal M}}_{k_1,k_2}((Y,\mathcal E_Y);a_-,a_+;R_1,R_2;E)$.
It carries a virtual fundamental chain,
whose boundary is the sum  of the virtual fundamental chains of the
following 6 types of spaces.
\begin{enumerate}
\item
$$
{{\mathcal M}}_{k'_1,k'_2}((Y,\mathcal E_Y);a_-,a'_+;R_1,R_2;E_1) 
\times 
{{\mathcal M}}_{k''_1,k''_2}(R(M_1),R(M_2);\{a'_+\},\{a_+\};E_2),
$$
where the notations are as in (\ref{5527427}).
\item
$$
{{\mathcal M}}_{k_1,k_2}((Y,\mathcal E_Y);a'_-,a_+;R_1,R_2;E_1) \\
\times 
{{\mathcal M}}(a_-,a'_-;E_2),
$$
where the notations are as in (\ref{54444488}).
\item
$$
{{\mathcal M}}_{k_1,k'_2}((Y,\mathcal E_Y);a_-,a_+;R_1,R'_2;E_1) 
{}_{{\rm ev}^{\infty}_2}\times_{{\rm ev}_-} 
{{\mathcal M}}_{k''_2}((-M_2,\mathcal E_2),R(M_2);R'_2,R_2;E_2),
$$
where the notations are as in (\ref{5524444edcd}).
\item
$$
{{\mathcal M}}_{k'_1,k_2}((Y,\mathcal E_Y);a_-,a_+;R'_1,R_2;E_1) 
{}_{{\rm ev}^{\infty}_1}\times_{{\rm ev}_+} 
{{\mathcal M}}_{k''_1}((M_1,\mathcal E_1),R(M_1);R_1,R'_1;E_2),
$$
where the notations are as in (\ref{5524444edcd2}).
\item
$$
{{\mathcal M}}_{k'_1,k_2}((Y,\mathcal E_Y);a_-,a_+;R_1,R_2;E_1)
\,{}_{{\rm ev}^{(1)}_i}\times_{{\rm ev}_0}
{{\mathcal M}}_{k''_1}(R(M_1);E_2),
$$
where $k_1 +1 = k'_1 + k''_1$, $E = E_1 + E_2$, $i=1,\dots,k'_1$.
\item
$$
{{\mathcal M}}_{k_1,k'_2}((Y,\mathcal E_Y);a_-,a_+;R_1,R_2;E_1)
\,{}_{{\rm ev}^{(2)}_i}\times_{{\rm ev}_0}
{{\mathcal M}}_{k''_2}(R(M_2);E_2),
$$
where $k_2 +1 = k'_2 + k''_2$, $E = E_1 + E_2$, $i=1,\dots,k'_2$.
\end{enumerate}
\end{prop}
\begin{proof}
The first 4 items describe the boundaries (I), (II), (III), (IV), respectively.
(5) describes the disk bubble on $\partial_4W$ and (6) 
describes the disk bubble on $\partial_1W$.
All the other bubbles occur in codimension 2 or higher.
\end{proof}
We next rewrite Proposition \ref{prop5858} to an algebraic formula.
\par
We first need a digression.
The second factor in Proposition \ref{prop5858} (3).
we used  $-M_2$ and we use 
${\rm ev}_-$ to take fibre product.
In the construction of right filtered $A_{\infty}$ module
structure in Section \ref{HF3bdfunctor} the evaluation map ${\rm ev}_-$
corresponds to the input variables.
When we change from $-M_2$ to $M_2$
the input variables becomes the output variables.
Namely we have an isomorphism
$$
{{\mathcal M}}_{k}((-M_2,\mathcal E_2),R(M_2);R_-,R_+;E)
\cong
{{\mathcal M}}_{k}((M_2,\mathcal E_2),R(M_2);R_+,R_-;E)
$$
which intertwine  $({\rm ev}_-,{\rm ev}_+)$ to $({\rm ev}_+,{\rm ev}_-)$.
Moreover $i$-th evaluation map ${\rm ev}_i
: {{\mathcal M}}_{k}((-M_2,\mathcal E_2),R(M_2);R_-,R_+;E)
\to R(M_2)$ becomes 
$(k-i)$-th evaluation map ${\rm ev}_{k-i}
: {{\mathcal M}}_{k}((M_2,\mathcal E_2),R(M_2);R_-,R_+;E)
\to R(M_2)$
by this isomorphism.
\par
This fact is used in Lemma \ref{lem5959} below.
\par
We now define the map
\begin{equation}
\aligned
\Phi_{k_1,k_2;E} :
&CF(M_1,R(M_1))
\otimes CF(R(M_1))^{k_2\otimes}\\
&\otimes  CF(R(M_1),R(M_2))
\\
&   \otimes CF(M_2,R(M_2)) \otimes CF(R(M_2))^{k_1\otimes}
\to CF(M;\mathcal E)
\endaligned
\end{equation}
by the formula
\begin{equation}
\aligned
&\Phi_{k_1,k_2;E}(y_1;
x^{(1)}_1,\dots,x^{(1)}_{k_1}; a_+;y_2;x^{(2)}_1,\dots,x^{(2)}_{k_2}) \\
&= 
\sum_{a_-} 
\#\big(
{{\mathcal M}}_{k_1,k_2}((Y,\mathcal E_Y);a_-,a_+;R_1,R_2;E)
\\
&\qquad
{}_{\rm ev}\times (
y_1\times x^{(1)}_1 \times \dots \times x^{(1)}_{k_1}  \times y_2
\times x^{(2)}_1\times \dots \times x^{(2)}_{k_2})
\big) [a_-].
\endaligned
\end{equation}
Here $y_1 \in CF(M_1,R(M_1))$ and $y_2 \in CF(M_2,R(M_2))$, 
$a_+ \in CF(R(M_1),R(M_2))$, $a_- \in CF(M;\mathcal E)$.
\par
We then put
\begin{equation}
\Phi_{k_1,k_2} = \sum_E T^E \Phi_{k_1,k_2;E}.
\end{equation}
\par
Proposition \ref{prop5858} implies 
the next formula:
\begin{lem}\label{lem5959}
$$
\aligned
&\partial(\Phi_{k_1,k_2}(
y_1;
x^{(1)}_1,\dots,x^{(1)}_{k_1};a_+;y_2;x^{(2)}_1,\dots,x^{(2)}_{k_2})) 
\\
=&
\sum \Phi_{k''_1,k''_2}(y_1;
x^{(1)}_1,\dots,x^{(1)}_{k''_1};\frak n_{k'_1,k'_2}(x^{(1)}_{k''_1+1},\dots,x^{(1)}_{k_1}
a_+; x^{(2)}_1,\dots,x^{(2)}_{k'_2}) ;y_2;x^{(2)}_{k'_2+1},\dots,x^{(2)}_{k_2}) \\
&+
\sum \Phi_{k''_1,k_2}(\frak n_{k'_1}(y_1;x^{(1)}_{1},\dots,x^{(1)}_{k'_1}); 
\dots,x^{(1)}_{k_1};
a_+;y_2;
x^{(2)}_1,\dots,x^{(2)}_{k_2}) \\
&+
\sum \Phi_{k_1,k''_2}(y_1;x^{(1)}_1,\dots,x^{(1)}_{k_1};a_+;
\frak n_{k'_2}(y_2;
x^{(2)}_{k_2},\dots,x^{(2)}_{k''_2+1}); x^{(2)}_{1},\dots,x^{(2)}_{k''_2}) \\
&+
\sum \Phi_{k''_1,k_2}(y_1;x^{(1)}_1,\dots,
\frak m_{k'_1}(x^{(1)}_i,\dots,x^{(1)}_{i+k'_1-1}),\dots,x^{(1)}_{k_1};a_+;y_2;
x^{(2)}_1,\dots,x^{(2)}_{k_2}) \\
&+
\sum \Phi_{k_1,k''_2}(y_1;x^{(1)}_1,\dots,x^{(1)}_{k_1};a_+;y_2;
x^{(2)}_1,\dots,\frak m_{k'_2}(x^{(2)}_i,\dots,x^{(2)}_{i+k'_2-1}),\dots,x^{(2)}_{k_2}).
\endaligned
$$
Here $\partial$ in the first line is Floer's boundary operator (Definition \ref{defn53}).
\end{lem}
\begin{proof}
1st, 2nd, 3rd, 4th, 5th, 6th lines of the formal corresponds to items 
(2), (1), (4), (3), (5), (6) of Proposition \ref{prop5858}, respectively.
\end{proof}
We define a map
$\Phi : CF(R(M_1),R(M_2)) \to CF(M;\mathcal E)$ by
\begin{equation}
\Phi(a_+) 
= 
\sum_{k_1,k_2} \Phi_{k_1,k_2}({\bf 1}_{M_1};b_{M_1},\dots,b_{M_1};a_+;
{\bf 1}_{M_2};b_{M_2},\dots,b_{M_2}).
\end{equation}
\begin{lem}
$\Phi$ is a chain map. Namely
$$
\partial \circ \Phi = \Phi \circ {}^{b_{M_1}}(\frak n_0)^{b_{M_2}}
$$
\end{lem}
\begin{proof}
We recall
$$
{}^{b_{M_1}}(\frak n_0)^{b_{M_2}}(a_+)
=
\sum_{k_1,k_2} \frak n_{k_1,k_2}(b_{M_1},\dots,b_{M_1};a_+;
b_{M_2},\dots,b_{M_2}).
$$
We put $y_1 = {\bf 1}_{M_1}$, $y_2 = {\bf 1}_{M_2}$,
$x^{(1)}_i = b_{M_1}$, $x^{(2)}_i = b_{M_2}$ in 
Lemma \ref{lem5959}.
Then the first line becomes $\partial \circ \Phi$.
The second line becomes 
${}^{\frak b_{M_1}}(\frak n_0)^{\frak b_{M_2}}(a_+)$.
The third line cancels since 
$d^{b_{M_2}}({\bf 1}_{M_1}) = 0$.
The fourth line cancels since 
$d^{b_{M_1}}({\bf 1}_{M_2}) = 0$.
The fifth line cancels since $b_{M_1}$ is a bounding cochain.
The sixth line cancels since 
$b_{M_2}$ is a bounding cochain.
\end{proof}
We thus defined a chain map (\ref{form51}).
To show that it is an isomorphism it suffices to prove the next:
\begin{lem}
$\Phi \equiv {\rm id} \mod \Lambda_{+}^{\Z_2}$.
\end{lem}
\begin{proof}
We remark that both 
$CF(R(M_1),R(M_2))$ and $CF(M;\mathcal E)$
are free $\Lambda_{0}^{\Z_2}$ module with basis 
$R(M) = R(M_1) \cap R(M_2)$.
So the statement makes sense.
\par
Since $b_{M_1},b_{M_2} \equiv 0 \mod \Lambda_{+}^{\Z_2}$
it suffices to study $\Phi_{0,0;0}$.
\par
The map $\Phi_{0,0;0}$ is defined by the moduli space
${{\mathcal M}}_{0,0}((Y,\mathcal E_Y);a_-,a_+;R_1,R_2;0)$.
\par
From the definition it consists of equivalence classes of elements
$(\frak A,{
\frak z},{\frak w},\Omega,u)$
such that $\frak A$ is a flat connection $u$ is a constant map 
and $\Omega = W$.
Therefore $\frak A = a_- = a_+$.
Using the fact that ${\bf 1}_{M_1}$ and ${\bf 1}_{M_2}$ are fundamental classes 
also we can prove the lemma easily.
\end{proof}
The proof of Theorem \ref{mainthm}  and other results stated 
in the introduction are now complete
modulo the construction and the check of its basic properties of the 
moduli spaces we used.
\qed
\begin{rem}\label{rem513}
This remark is a continuation of Remark \ref{rem415}.
Let $M_1$ and $M_2$ be symplectic manifolds and $L_{12}$
an immersed Lagrangian submanifold of $M_1 \times  -M_2$.
Let $L_1$, $L'_1$ be immersed Lagrangian submanifolds of $M_1$.
We obtain an immersed Lagrangian submanifolds $L_2$ (resp. $L'_2$) of $M_2$ 
by using $L_1$ and $L_{12}$ (resp. $L'_1$ and $L_{12}$).
We assume $M_1$, $M_2$, $L_{12}$, $L_1$, $L'_1$ are spin.
Then $L_2$ and $L'_2$ are spin.
Let $b_1$, $b'_1$, $b_{12}$ be bounding cochains of 
the filtered $A_{\infty}$ algebra associated to $L_1$, $L'_1$,
$L_{12}$, respectively.
\par
Then by Theorem \ref{WWfunc}, we obtain bounding cochains 
$b_{2}$ and $b'_2$ of $L_2$ and $L'_2$, respectively.
In the same way as the construction of the chain map 
(\ref{form51}) and the proof of Theorem \ref{mainthm} (2)
we can construct a canonical homomorphism
\begin{equation}
\Phi_{(L_{12},b_{12})} : 
HF((L_1,b_1),(L'_1,b'_1)) \to HF((L_2,b_2),(L'_2,b'_2)).
\end{equation}
The construction is based on the following Figure 5.11.
\par
\begin{center}
\includegraphics[scale=0.25]
{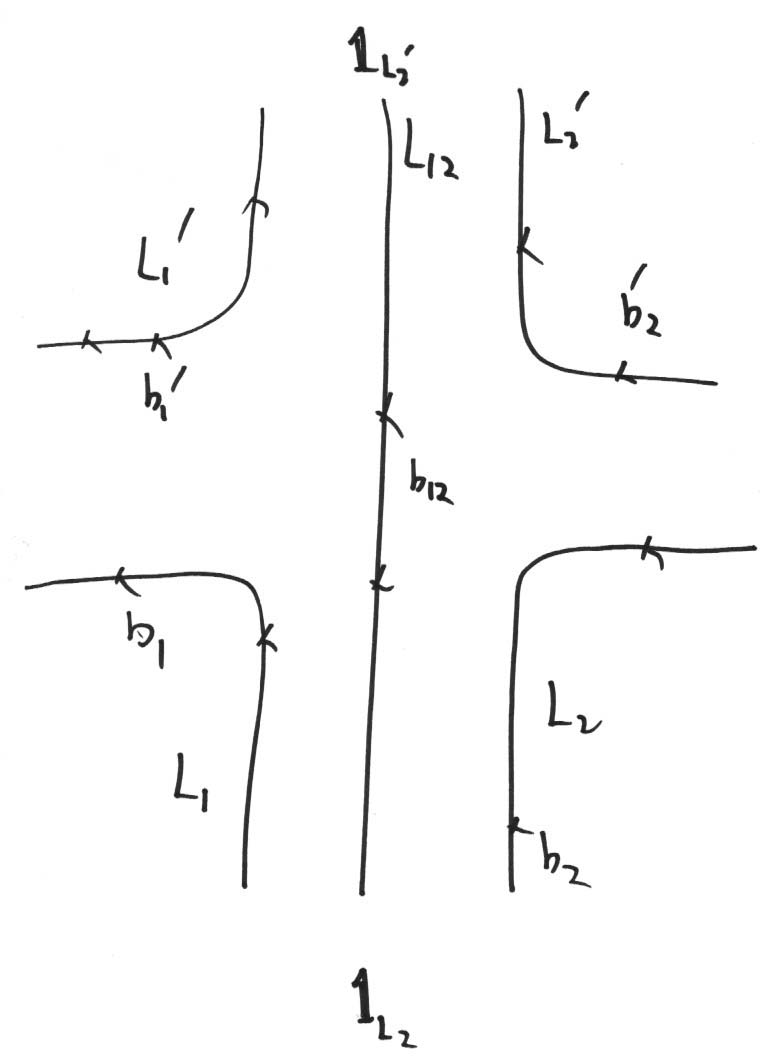}
\end{center}
\centerline{\bf Figure 5.11}
\par\medskip
We can enhance the homomorphism $\Phi_{(L_{12},b_{12})}$ to 
an $A_{\infty}$ functor
\begin{equation}\label{formula525}
\Phi_{(L_{12},b_{12})} : 
\mathscr{FUK}(M_1) \to \mathscr{FUK}(M_2).
\end{equation}
Note this map $\Phi_{(L_{12},b_{12})}$ may not be an isomorphism 
in general. In fact $L_1 \cap L'_1 \ne L_2 \cap L'_2$ in general.
\par
This functor was constructed by Ma'u-Wehrheim-Woodward
\cite{MWW} under certain additional assumptions.
In the same way as Lekili-Lipyanskiy \cite{LL}
we can prove its compatibility with composition.
Namely:
\begin{thm}
Let $M_3$ be another spin symplectic manifold 
and $L_{23}$ an immersed spin Lagrangian submanifold of 
$M_2 \times -M_3$.
Suppose $b_{23}$ is a bounding cochain of $L_{23}$.
Let $\tilde L_{13} = \tilde L_{12} \times_{M_2} \tilde L_{23}$.
(Then $i_{L_{13}} : \tilde L_{13} \to L_{13} \subset M_1 \times M_3$ is an 
immersed Lagrangian submanifold.)
\par
Applying Lagrangian correspondence $L_{23}$ to $L_2$ and $L'_2$ 
we obtain immersed Lagrangian submanifolds $L_3$ and $L'_3$ respectively.
(They coincide with the immersed Lagrangian submanifolds obtained from 
$L_1$ and $L'_1$ by $L_{13}$.)
\par
We apply Theorem \ref{WWfunc}  to $(L_2,b_2)$ (resp. $(L'_2,b'_2)$) and $(L_{23},b_{23})$ 
and obtain bounding cochain $b_{3}$ (resp. $b'_3$) of $L_3$ (resp. $L'_3$).
\begin{enumerate}
\item
The immersed Lagrangian submanifold 
$L_{13}$ is unobstracted. Moreover the bounding cochains $b_{12}$ and $b_{23}$ induce a 
bounding cochain $b_{13}$ of $L_{13}$ in a canonical way, up to gauge equivalence.
\item
The bounding cochain $b_3$ (resp. $b'_3$) is gauge equivalent to the bounding cochains 
obtained by applying Theorem \ref{WWfunc} to 
$(L_{13},b_{13})$ and  $b_3$ (resp. $b'_3$).
\item
We have equality
$$
\Phi_{(L_{23},b_{23})}\circ\Phi_{(L_{12},b_{12})}
=\Phi_{(L_{13},b_{13})}:
HF((L_1,b_1),(L'_1,b'_1)) \to HF((L_3,b_3),(L'_3,b'_3)).
$$
\end{enumerate}
\end{thm}
\begin{proof}[Sketch of the proof]
Proof of (1): 
We consider the maps $u$ from the cylinder of infinite length in the 
Figure 5.12 below, where three regions 
$W_1$, $W_2$ and $W_3$ are mapped to 
$M_1$, $M_2$ and $M_3$, respectively, by a pseudo-holomorphic map $u$. 
We also require that the three lines 
$C_{12}$, $C_{23}$ and $C_{13}$
are mapped to $L_{12}$, $L_{23}$ and $L_{13}$ respectively.
\par
We are in Bott-Morse situation and the asymptotic limit as we go to plus 
or minus infinity is an element of 
$$
L_{13} \times_{M_1\times M_3} L_{13}
= (L_{12} \times L_{23} \times L_{13}) \times_{(M_1\times M_2 \times M_3)^2}
\Delta_{M_1\times M_2 \times M_3}
$$
Here $\Delta_{M_1\times M_2 \times M_3}$ is a diagonal.
(We can work out analytic detail 
by regarding this moduli space as those to define 
Floer homology
$$
HF(L_{12} \times L_{23} \times L_{13},\Delta_{M_1\times M_2 \times M_3})
$$
of a pair of immersed Lagrangian submanifolds of 
$(M_1\times M_2 \times M_3)^2$.)
\par
We use bounding cochains $b_{12}$, $b_{23}$ to cancel the effect of 
bubbles on $C_{12}$, $C_{23}$, respectively, in the same way as 
\cite{fooobook}. 
\par
Then using marked points on the line $C_{13}$, we can define a structure of 
filtered right $A_{\infty}$ module on 
$CF(L_{12} \times L_{23} \times L_{13},\Delta_{M_1\times M_2 \times M_3})$
over the filtered $A_{\infty}$ algebra $CF(L_{13})$.
In the same way as Section \ref{sec:unbolt}, 
we can show that the fundamental class is a cyclic element of this right 
filtered $A_{\infty}$ module.
Thus applying Proposition \ref{thm35}, 
we obtain $b_{13}$.
\par
\begin{center}
\includegraphics[scale=0.25]
{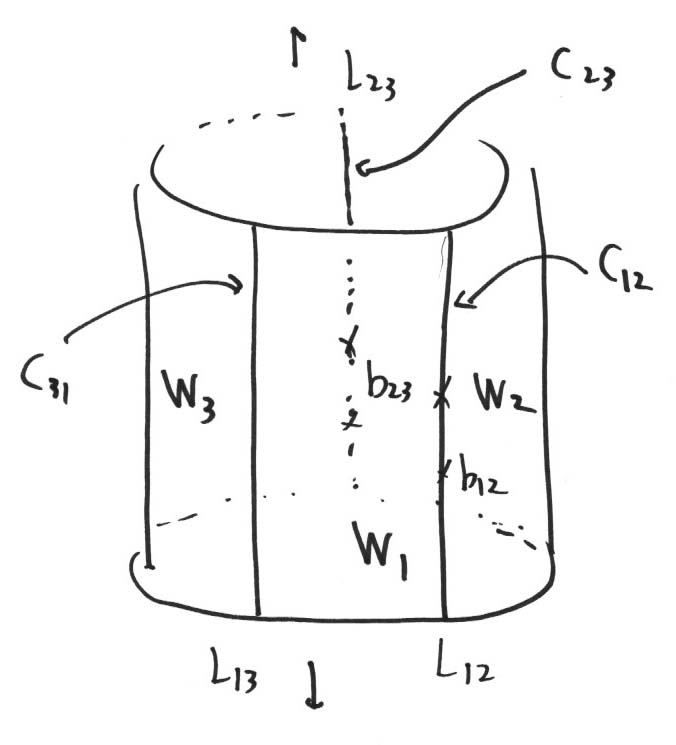}
\end{center}
\centerline{\bf Figure 5.12}
\par\medskip
Proof of (2):
We denote by $b_{3}^{1}$ (resp. $b_{3}^{\prime 1}$) be the 
bounding cochain on $L_3$ (resp. on $L'_3$) 
obtained from $(L_{23},b_{23})$ and $(L_2,b_2)$ (resp.
$(L'_2,b'_2)$)). 
We remark that $(L_2,b_2)$ (resp.
$(L'_2,b'_2)$) is obtained from $(L_{12},b_{12})$ and $(L_1,b_1)$ (resp.
$(L'_1,b'_1)$).
\par
We denote by $b_{3}^{2}$ (resp. $b_{3}^{\prime 2}$) the 
bounding cochain on $L_3$ (resp. on $L'_3$) 
obtained from $(L_{13},b_{13})$ and $(L_1,b_1)$ (resp.
$(L'_1,b'_1)$)). 
\par Let $(L,b)$ be a pair of an immersed spin Lagrangian submanifold of $M_3$ and 
its bounding cochain. We consider the moduli space of 
pseudo-holomorphic quilt as in Figure 5.13 below:
\par
\begin{center}
\includegraphics[scale=0.25]
{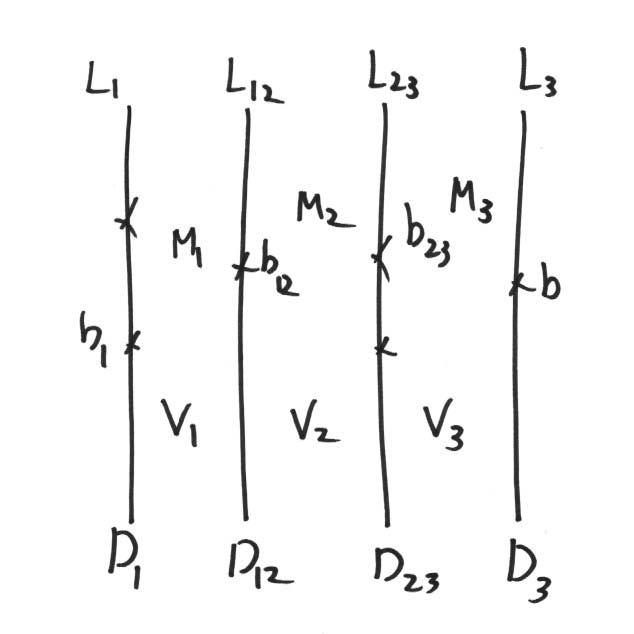}
\end{center}
\centerline{\bf Figure 5.13}
\par\medskip
Here the domains $V_1$, $V_2$ and $V_3$ are mapped to $M_1$, $M_2$ and
$M_3$ respectively and the maps are pseudo-holomorphic.
We also require the four lines $D_1$, $D_{12}$, $D_{23}$ and $D_3$ are mapped to 
$L_1$, $L_{12}$, $L_{23}$, $L$  respectively.
Using $b_1$, $b_{12}$, $b_{23}$ and $b$ to cancel the bubble 
on the lines $D_1$, $D_{12}$, $D_{23}$ and $D_3$, respectively,  
we obtain a chain complex, on the $\Lambda_0^{\Q}$ module.
$$
C(L_1 \times_{M_1} L_{12} \times_{M_2} L_{23} \times_{M_3} L;\Lambda_{0}^{\Q}).
 = CF(L_1;L_{12};L_{23};L)
$$
We write its cohomology by
$HF((L_1,b_1);(L_{12},b_{12});(L_{23},b_{23});(L,b))$.
We claim
\begin{equation}\label{form526}
HF((L_1,b_1);(L_{12},b_{12});(L_{23},b_{23});(L,b)) 
\cong HF((L_1,b_1) \times (L,b);(L_{13},b_{13})).
\end{equation}
The proof of (\ref{form526}) is basically the same as the proof of the 
corresponding result  in \cite{LL}.
Namely we use the next Figure 5.14.
(Figure 5.14 is the same as  \cite[Figure 1]{LL}.)
\par\newpage
\begin{center}
\includegraphics[scale=0.25]
{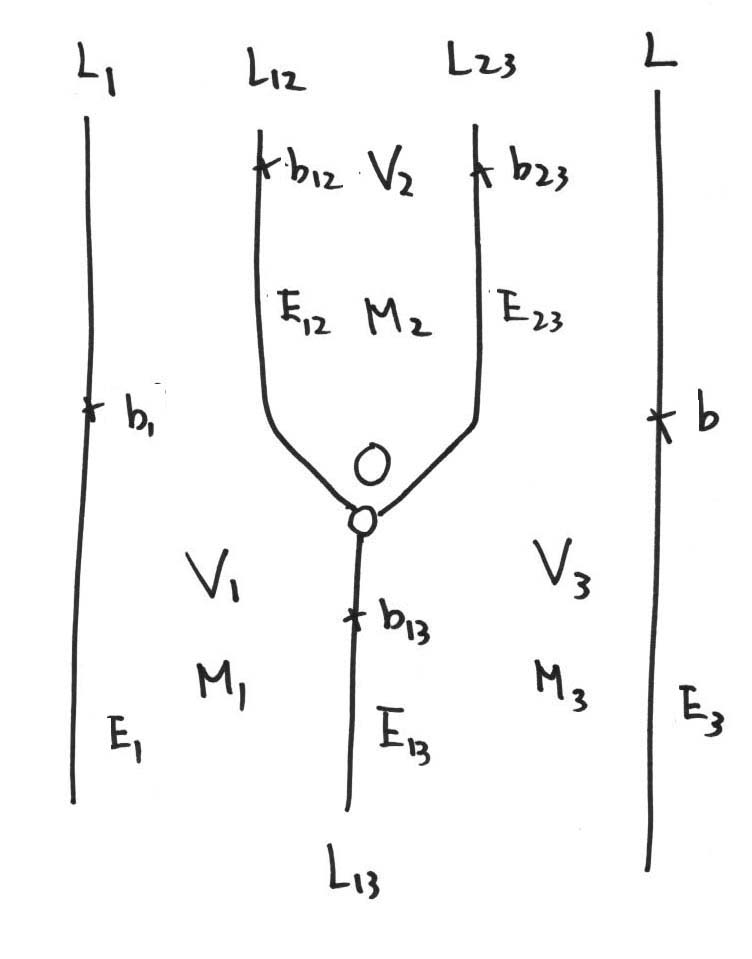}
\end{center}
\centerline{\bf Figure 5.14}
\par\medskip
We consider maps $u$ from $V$ in Figure 5.14 such that 
$V_1$, $V_2$, $V_3$ are mapped by $u$  to $M_1$, 
$M_2$, $M_3$, respectively.
We also require that the curves $E_1$, $E_{12}$, $E_{23}$, $E_{13}$, $E_3$
are mapped to $L_1$, $L_{12}$, $L_{23}$, $L_{13}$, $L$, respectively.
Moreover we conformally identify a neighborhood of the point $O$ as the 
(half of the) cylinder as in Figure 5.12.
\par
We use bounding cochains $b_1$, $b_{12}$, $b_{23}$, $b_{13}$, $b$
to cancel bubbles on the lines $E_1$, $E_{12}$, $E_{23}$, $E_{13}$, $E_3$,
respectively, in the same way as \cite{fooobook}.
At the point $O$ we use the cyclic element (the fundamental chain) to 
define asymptotic boundary condition.
\par
We remark that $V_1$, $V_2$, $V_3$ are in the clock-wise order 
in Figure 5.14 and are in counter-clock-wise order 
in Figure 5.12. This changes the cycle we put to $O$ from output variable 
to input variable.
\par
Then in the same way as the proof of Theorem \ref{mainthm} (2) given 
in this section, the remaining boundary component is 
the next two cases.
\begin{enumerate}
\item[(I)]
A pseudo-holomorphic strip escape to the direction ${\rm Im}z \to +\infty$.
\item[(II)]
A pseudo-holomorphic strip escape to the direction ${\rm Im}z \to -\infty$.
\end{enumerate}
(I) gives the boundary operator defining 
$HF((L_1,b_1);(L_{12},b_{12});(L_{23},b_{23});(L,b))$.
(II) given the boundary operator defining
$HF((L_1,b_1) \times (L,b);(L_{13},b_{13}))$.
Therefore we obtain a chain map between two 
chain complexes defining them.
We can show that this chain map is congruent to the 
identity map mod $\Lambda_{0}^{\Q}$, 
since the energy zero element of this moduli space 
is a constant map.
This proves (\ref{form526}).
\par
We next prove:
\begin{equation}\label{form527}
HF((L_1,b_1);(L_{12},b_{12});(L_{23},b_{23});(L,b)) 
\cong HF((L_3,b_{3}^{1}),(L,b)).
\end{equation}
To prove (\ref{form527}) we consider the next Figure 5.15.
We consider a map $u$ such that it is a map to $M_1$, $M_2$, $M_3$ 
on the domain $U_1$, $U_2$, $U_3$, respectively.
We also require $u$ maps $F_1$, $F_2$, $F_3$, $F'_3$, $F_{12}$, $F_{23}$ to 
$L_1$, $L_2$, $L_3$, $L$, $L_{12}$, $L_{23}$ respectively.
\par
We put bounding cochains $b_1$, $b_2$, $b_3^{1}$, $b$, $b_{12}$, $b_{23}$
on $F_1$, $F_2$, $F_3$, $F'_3$, $F_{12}$, $F_{23}$, respectively to cancel 
the contribution of disk bubbles there.
\par
We regard the ends $O_{12}$ and $O_{23}$ are as Figure 4.15 
and put asymptotic boundary conditions by using fundamental classes there.
\par
\begin{center}
\includegraphics[scale=0.25]
{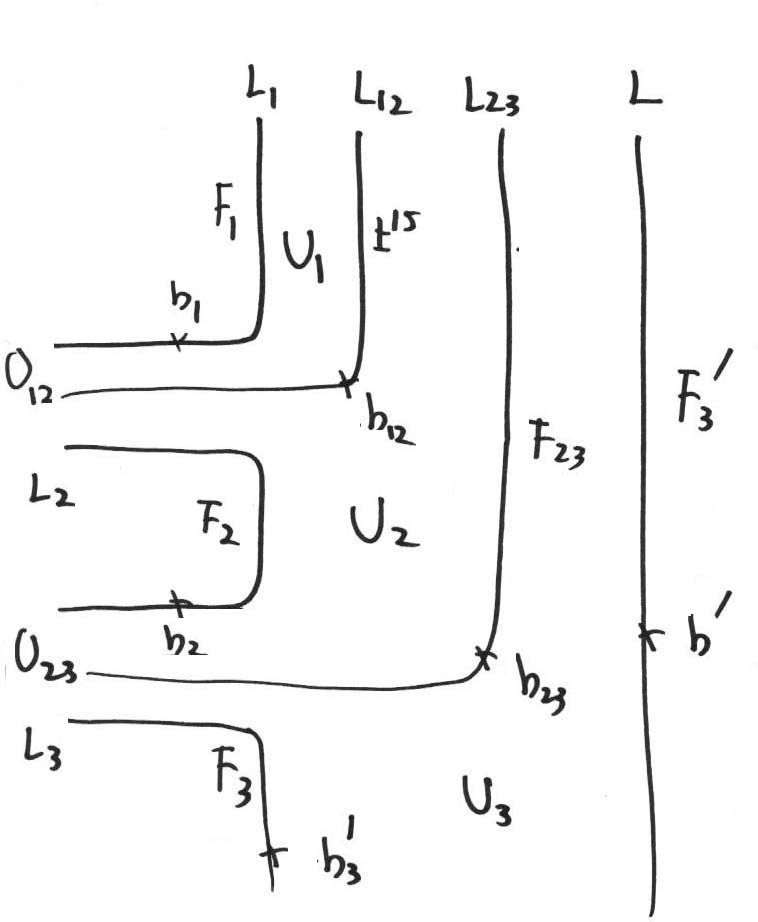}
\end{center}
\centerline{\bf Figure 5.15}
\par\medskip
Thus the remaining boundary component of this moduli space is 
as in (I) and (II) above.
(I)  corresponds to the boundary operator defining left hand side of 
(\ref{form527}). 
(II) corresponds to the  boundary operator defining right hand side of 
(\ref{form527}). 
Therefore we obtain a chain map.
Using the fact that energy zero solution corresponds to the 
constant map and we put fundamental classes at $O_{12}$ and $O_{23}$,
we find that this chain map is congruent to the identity map modulo $\Lambda_0^{\Q}$. 
We proved (\ref{form527}).
\par
On the other hand, Theorem \ref{WWfunc} implies
\begin{equation}\label{form528}
HF((L_1,b_1) \times (L,b),(L_{13},b_{13}))
\cong HF((L_3,b_{3}^{2}),(L,b)).
\end{equation}
Combining (\ref{form526}), (\ref{form527}), (\ref{form528}) we find
\begin{equation}\label{form529}
HF((L_3,b_{3}^{1}),(L,b)) \cong HF((L_3,b_{3}^{2}),(L,b)).
\end{equation}
It is straightforward to check that the isomorphism (\ref{form529}) is functorial.
Namely two filtered $A_{\infty}$ functors represented by 
$(L_3,b_{3}^{1})$ and by $(L_3,b_{3}^{2})$ are homotopy equivalent.
Using $A_{\infty}$-Yoneda's lemma two objects 
$(L_3,b_{3}^{1})$ and $(L_3,b_{3}^{2})$ are homotopy equivalent.
It implies (2).
\par
We omit the proof of (3). (We will prove it in \cite{fu9}.)
We can also enhance the isomorphism in (3) to a homotopy equivalence between 
two filtered $A_{\infty}$ functors.
\end{proof}
\end{rem}

\section{Concluding remarks}\label{conceal}

\subsection{Topological field theory description}\label{topofie}

We remark that the results of this paper provides 2-3 dimensional topological
field theory picture of Gauge theory Floer homology.
(Such a picture is initiated by \cite{Don}, \cite{fu1} and also by G. Segal, 
in early 1990's.)
\par
Let $(M,\mathcal E_M)$ be as in Situation \ref{situ21}.
We divide $\partial M$ into two pieces, input part $\partial_{{\rm in}}M
= \Sigma_{{\rm in}}$
and output part  $\partial_{{\rm in}}M = \Sigma_{\rm out}$.
For  $\partial_{{\rm out}}M = \Sigma_{{\rm out}}$  we invert the orientation.
\begin{defn}
We call this situation that 
$(M,\mathcal E_M)$ is a {\it cobordism} from $(\Sigma_{{\rm in}},\mathcal E_{{\rm in}})$  to $(\Sigma_{{\rm out}},\mathcal E_{\rm out})$.
\end{defn}
\par
The space $R(\Sigma)$ with it symplectic structure $\omega_{\Sigma}$ 
is written as 
$$
(R(\Sigma),\omega_{\Sigma}) 
= (R(\Sigma_{{\rm in}}),\omega_{\Sigma_{{\rm in}}})
\times 
(R(\Sigma_{{\rm out}}),-\omega_{\Sigma_{{\rm out}}}).
$$
The space of flat connections $R(M)$ of $M$ is an immersed Lagrangian submanifold 
of it.
By Theorem \ref{mainthm} (1) we obtain a bounding cochain $b_M$ of 
the filtered $A_{\infty}$ algebra associated to $R(M)$.
Therefore (\ref{formula525}) induces a filtered $A_{\infty}$
functor
\begin{equation}
\Phi_{M} : 
\mathscr{FUK}(R(\Sigma_{{\rm in}})) \to \mathscr{FUK}(R(\Sigma_{{\rm out}})).
\end{equation}
This construction behave functorially as follows.
Let $(M_{12},\mathcal E_{12})$ 
(resp. $(M_{23},\mathcal E_{23})$) be a cobordism from $(\Sigma_1,\mathcal E_1)$ to $(\Sigma_2,\mathcal E_2)$ (resp. from $(\Sigma_2,\mathcal E_2)$ to $(\Sigma_3,\mathcal E_3)$).
\par
We glue $(M_{12},\mathcal E_{12})$  and
$(M_{23},\mathcal E_{23})$ along $(\Sigma_2,\mathcal E_2)$
to obtain $(M_{13},\mathcal E_{13})$.
\begin{thm}\label{thm62}
The composition $\Phi_{M_{23}}\circ \Phi_{M_{12}}$ is homotopy equivalent to 
$\Phi_{M_{13}}$ as filtered $A_{\infty}$ functors.
\end{thm}
\begin{proof}[Sketch of the proof]
We remark that in case $\Sigma_1 = \Sigma_3 = \emptyset$, this is Theorem \ref{mainthm} (2).
For the proof of the general case we first consider filtered $A_{\infty}$
bifunctor
\begin{equation}\label{bifunctor}
\mathscr{FUK}(R(\Sigma_{1}))^{\rm op} \times \mathscr{FUK}(R(\Sigma_{2})) 
\to \mathscr{CH}.
\end{equation}
Here ${\rm op}$ stands for the opposite $A_{\infty}$ category.
(See \cite[Definition 7.8]{fu4}.)
Note the left hand side of (\ref{bifunctor}) is a full subcategory of 
$\mathscr{FUK}(- R(\Sigma_{1}) \times R(\Sigma_{2}))$.
(See \cite{Limo}.)
\par
Therefore object of 
$\mathscr{FUK}(R(\Sigma_{1}) \times -R(\Sigma_{2}))$
represents a functor as in 
(\ref{bifunctor}). 
The pair $(R(M_{12}),b_{M_{12}})$  is such an object and 
it represents 
\begin{equation}\label{phigatoro}
((L_1,b_1),(L'_2,b'_2)) \mapsto HF(\Phi_{M_{12}}(L_1,b_1),(L'_2,b'_2)).
\end{equation}
This is a consequence of Theorem \ref{WWfunc}. 
Namely (\ref{phigatoro}) is isomorphic to 
\begin{equation}\label{64dd}
HF((R(M_{12}),b_{M_{12}}),(L_1,b_1) \times (L'_2,b'_2))
\end{equation}
by  Theorem \ref{WWfunc}.
\par
On the other hand, Theorem \ref{represent} implies that 
the Floer homology group (\ref{64dd}) is isomorphic to 
$HF((M_{12},\mathcal E_{12}),(L_1,b_1) \times (L'_2,b'_2)))$.
We can rephrase the definition of this group in Section \ref{HF3bdfunctor} and 
find that it is  
a homology group of the chain complex $C(R(M_{12})\times_{R(\Sigma)}(L_1 \times L'_2);\Lambda_0^{\Z_2})$
whose boundary operator is defined by using the moduli space 
of solutions of  (\ref{ASDeq}), (\ref{ASDprod}) on the domain described in the 
next figure.
\par\newpage
\begin{center}
\includegraphics[scale=0.25]
{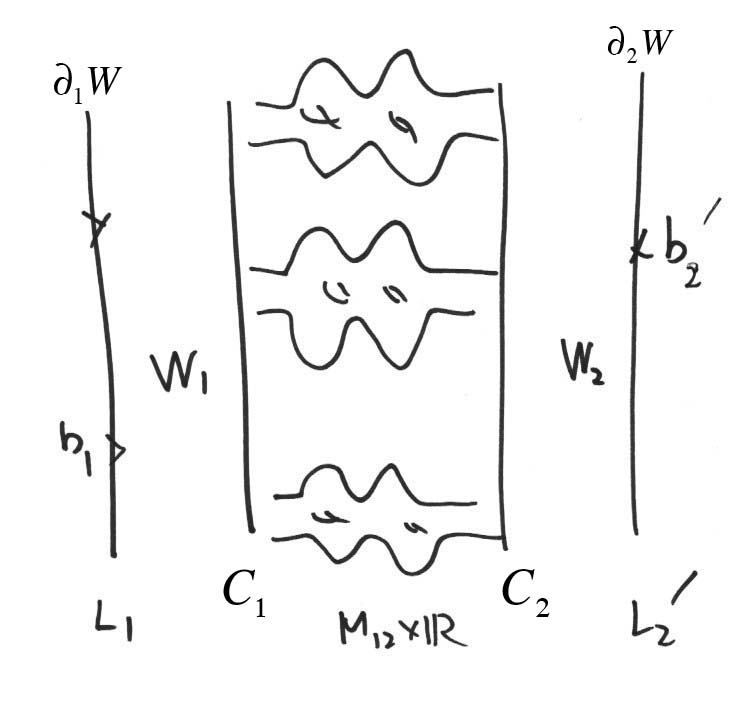}
\end{center}
\centerline{\bf Figure 6.1}
\par\medskip
Here the domain is divided into 3 parts. On two of them $W_1$ and $W_2$ 
we use degenerate metrics $0 g_{\Sigma_1} + ds^2 + dt^2$ and 
$0 g_{\Sigma_2} + ds^2 + dt^2$, respectively and our equation is a holomorphic
curve equation to $R(\Sigma_1)$ and to $R(\Sigma_2)$, respectively.
The third part is $M_{12} \times \R$, on which we consider the ASD-equation 
(\ref{ASDeq}). The metric is of the form $\chi^2 g_{\Sigma_2} + ds^2 + dt^2$
near their borderline curves $C_1$ and $C_2$, where $\chi$ is a function similar to those 
we used several times.
\par
We require that the holomorphic curve takes boundary value 
in $L_1$ and $L'_2$ on the boundary $\partial_1 W$ and $\partial_2 W$,
 respectively.
 \par
 We require asymptotic boundary conditions for $t \to \pm \infty$ 
 by using elements of $L_1 \times_{R(\Sigma_1)} R(M_{12}) \times_{R(\Sigma_2)} L'_2$.
We cancel the contribution of the disk bubble on $\partial_1 W$ and on $\partial_2 W$
by using bounding cochains $b_1$ and $b'_2$ in the same way as 
\cite{fooobook}.
We thus obtain a boundary operator on 
$C(R(M_{12})\times_{R(\Sigma)}(L_1 \times L'_2);\Lambda_0^{\Z_2})$.
\par
It is easy to see from definition that the homology group of this chain complex is
\begin{equation}\label{group65}
HF((M_{12},\mathcal E_{12}),((L_1 \times L'_2),(b_1 \times b'_2))).
\end{equation}
The isomorphism $(\ref{64dd}) \cong (\ref{group65})$ is 
a consequence of Theorem \ref{represent}.
\par
We go back to the proof of Theorem \ref{thm62}.
Let $(L_1,b_1)$ (resp. $(L'_3,b'_3)$) be an 
object of $\mathscr{FUK}(R(\Sigma_{1}))$
(resp. $\mathscr{FUK}(R(\Sigma_{3}))$).
We put $\Phi_{M_{12}}(L_1,b_1) = (L_2,b_2)$.
To prove Theorem \ref{thm62}
it suffices to construct a functorial isomorphism:
\begin{equation}\label{form666}
HF(\Phi_{M_{23}} (L_2,b_2),(L'_3,b'_3))
\cong 
HF(\Phi_{M_{13}} (L_1,b_1),(L'_3,b'_3)).
\end{equation}
We use the moduli space 
obtained by the solution of  (\ref{ASDeq}), (\ref{ASDprod}) on the domain described in the 
next figure.
\par\newpage
\begin{center}
\includegraphics[scale=0.25]
{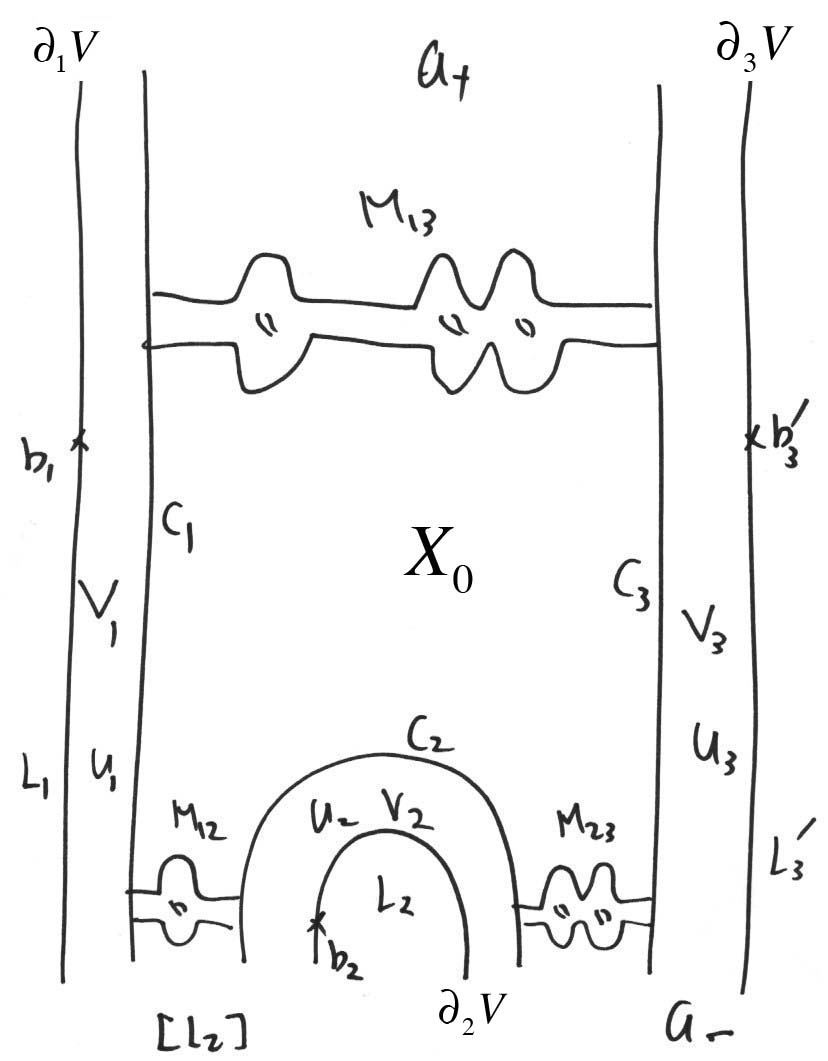}
\end{center}
\centerline{\bf Figure 6.2}
\par\medskip
This domain is divided into 4 pieces. On three of them 
$V_1$, $V_2$, $V_3$ we use the degenerate metrics 
$0 g_{\Sigma_1} + ds^2 + dt^2$,
$0 g_{\Sigma_2} + ds^2 + dt^2$, 
$0 g_{\Sigma_3} + ds^2 + dt^2$, respectively.
So our equation is a holomorphic curve equation to 
$R(\Sigma_1)$, $R(\Sigma_2)$, $R(\Sigma_3)$
on $V_1$, $V_2$, $V_3$ respectively.
\par
The fourth part of the domain $X_0$ is a Riemannian 
4-manifold. It has three ends. Two of them are in the part $t\to
-\infty$ and are isometric to $M_{12} \times (-\infty,-c]$
and to $M_{23} \times (-\infty,-c]$.
The third end is in the part $t \to +\infty$ and is 
isometric to $M_{13} \times (c,+\infty)$.
On $X_0$ we require the ASD-equation.
\par
$X_0$ intersects with $\Sigma_i \times V_i$ 
on $\Sigma \times C_i$.
Near the borderline curve $C_i$ (which is diffeomorphic to $\R$) 
the `metric' is isometric to one of the form $\chi(s)^2g_{\Sigma_1} + ds^2 + dt^2$
where $\chi$ is a similar function as we used several times in this paper.
\par
We put $\partial_iV = \partial V_i \setminus C_i$.
They are diffeomorphic to $\R$. 
On $\partial_iV$ we put the following boundary conditions.
We put $\Phi_{12}(L_1,b_1) = (L_2,b_2)$.
\begin{enumerate}
\item
The image of the restriction of $u_1$ to $\partial_1V$ is in $L_1$.
\item
The image of the restriction of $u_2$ to $\partial_2V$ is in $L_2$.
\item
The image of the restriction of $u_3$ to $\partial_3V$ is in $L'_3$.
\end{enumerate}
Actually in case $L_1$, $L_2$ or $L'_3$  are immersed we need to state the boundary condition 
a bit more carefully. Since we explained the way how to do so in a similar 
situation several times already in this paper we omit it here. 
\par
We now consider the asymptotic boundary conditions for 
three ends.
We take 
\begin{equation}\label{form67}
a_-, a_+ \in \tilde L_1 \times_{R(M_{13})} \tilde L'_3 
\cong \tilde L_2 \times_{R(M_{23})} \tilde L'_3.
\end{equation}
Note the isomorphism in (\ref{form67}) can be proved by 
$$
R(M_{13}) = R(M_{12}) \times_{R(\Sigma_2)} R(M_{23}) 
$$
and
$$
\tilde L_2 = \tilde L_1 \times_{R(\Sigma_1)} R(M_{12}).
$$
\begin{enumerate}
\item[(I)] 
At the end isometric to $M_{12} \times (-\infty,-c]$
we require that our connection converges to a flat connection 
in $L_2$.
In other words we put the fundamental class 
$$
[L_2] \in C(M_{12} \times_{R(\Sigma_1) \times R(\Sigma_2)} (\tilde L_1\times \tilde L_2))
\cong C(\tilde L_2 \times_{R(\Sigma_2)} \tilde L_2).
$$
\item[(II)]
At the end isometric to $M_{23} \times (-\infty,-c]$
we require that our connection is asymptotic to $a_-$.
\item[(III)]
At the end isometric to $M_{13} \times [c,+\infty)$
we require that our connection is asymptotic to $a_+$.
\end{enumerate}
Finally we use bounding cochains $b_1$, $b_2$ and $b^1_3$ 
on $\partial_1V$, $\partial_2V$ and $\partial_3V$, 
respectively, to cancel the contribution of the disk bubble there.
\par
We use this moduli space in case when its virtual dimension is $0$
to define a matrix element 
$\langle\Phi(a_-),a_+\rangle$ of the map
\begin{equation}\label{homo68}
\Phi : CF(\Phi_{M_{23}} (L_2,b_2),(L'_3,b'_3))
\to CF((M_3,\mathcal E_3),(L_1,b_1)\times (L'_3,b'_3)).
\end{equation}
We claim that $\Phi$ is a chain map.
To prove it we consider the same moduli space
when its virtual dimension is $1$.
We study its boundary. 
\par
The contribution of the codimension one boundary corresponding to the 
disk bubbles on $\partial_iV$ is zero since 
we used bounding cochains $b_1$, $b_2$ and $b^1_3$ to cancel it.
The contribution of the codimension one boundary corresponds to the 
end of type (I) also vanishes. 
This is because of the choice of $b_2$.
(In fact we used Proposition \ref{thm35} to find $b_2$.
In other words, the fundamental class which we put as an 
asymptotic value at this ends, is a cycle.)
\par
The two other ends (II) and (III) 
correspond to the boundary operator of the 
source and target of (\ref{homo68}), 
respectively. Thus (\ref{homo68}) is a chain map.
\par
The energy $0$ solution of our equation consists of flat 
connections and constant maps.
We use this fact to prove that (\ref{homo68}) 
is congruent to the identity map modulo $\Lambda_0^{\Z_2}$.
\par
We thus constructed an isomorphism (\ref{form666}).
\par 
We omit the proof of functoriality.
\end{proof}
\begin{thm}\label{thm63}
We have the following additional properties of our 
functor $M_{12} \mapsto \Phi_{M_{12}}$ and 
$\Sigma \mapsto \mathscr{FUK}(R(\Sigma))$.
\begin{enumerate}
\item
$\mathscr{FUK}(R(\Sigma_1 \sqcup \Sigma_2))
= \mathscr{FUK}(R(\Sigma_1) \times \mathscr{FUK}(R(\Sigma_2))$.
\item 
$\mathscr{FUK}(R(-\Sigma)) = \mathscr{FUK}(R(\Sigma))^{\rm op}$.
Here ${\rm op}$ denotes the opposite filtered $A_{\infty}$ category.
\item
We invert the orientation of $M_{12}$ and obtain $M_{21}$. We use the same bundle 
$\mathcal E_{12} = \mathcal E_{21}$. Then  $\Phi_{M_{12}}$ is the adjoint functor
to $\Phi_{M_{21}}$. Namely there exists an isomorphism
$$
HF(\Phi_{M_{12}}(L_1,b_1),(L'_2,b'_2)) \cong
HF(\Phi_{M_{21}}(L'_2,b'_2),(L_1,b_1)),
$$
which is functorial. 
Namely the left and right hand sides are homotpy equivalent 
as filtered $A_{\infty}$ bifunctors (\ref{bifunctor}).
\item
Suppose $\partial_{\rm out}(M,\mathcal E) \cong -\partial_{\rm in}(M,\mathcal E)$.
We glue $\partial_{\rm out}(M,\mathcal E)$ and 
$\partial_{\rm in}(M,\mathcal E)$ in $(M,\mathcal E)$ to obtain 
a closed manifold with bundle,
$(\hat M,\hat{\mathcal E})$.  Then the following isomorphism holds.
$$
HF(\hat M,\hat{\mathcal E}) \otimes_{\Z_2} \Lambda^{\Z_2}
\cong
H(\mathscr{HOM}(\frak{ID},\Phi_{M_{12}})) \otimes_{\Z_2} \Lambda^{\Z_2}.
$$
Here $\frak{ID} : \mathscr{FUK}(R(\Sigma_{\rm in}),R(\Sigma_{\rm in}))$ 
is the identity functor and the right hand side is the homology of the 
chain complex consisting of $A_{\infty}$ pre-natural transformations
from $\frak{ID}$ to $\Phi_{12}$. (See \cite[Definition 7.49]{fu4}.)
\end{enumerate}
\end{thm}
\begin{proof}
(1) is obvious from the definition.
\par
(2) We observe that $R(\Sigma) = 
R(-\Sigma)$ as spaces. On the other hand, the symplectic forms 
$\omega_{R(\Sigma)}$, $\omega_{R(-\Sigma)}$
and complex structures 
$J_{R(\Sigma)}$, $J_{R(-\Sigma)}$
are related by the formula:
$$
\omega_{R(-\Sigma)} = -\omega_{R(\Sigma)},
\qquad 
J_{R(-\Sigma)} = -J_{R(\Sigma)}.
$$
This implies (2) as follows.
We remark that an (immersed) Lagrangian 
submanifold $L$ of 
$(R(\Sigma),\omega_{R(\Sigma)})$ is also 
a Lagrangian 
submanifold $L$ of 
$(R(\Sigma),-\omega_{R(\Sigma)})$.
We consider the 
moduli space
$\mathcal M_{k+1}((R(\Sigma),\omega_{R(\Sigma)});L;\beta)$,
which appeared (\ref{diskmoduli}) and
is the case when we take $(R(\Sigma),\omega_{R(\Sigma)})$ 
as the ambient symplectic manifold.
\par
We have an isomorphism 
$$
\mathcal M_{k+1}((R(\Sigma),\omega_{R(\Sigma)});L;\beta)
\cong
\mathcal M_{k+1}((R(\Sigma),-\omega_{R(\Sigma)});L;-\beta)
$$
of spaces with Kuranishi structures. 
This isomorphism is defined by sending an element
$
(u;(z_0,z_1,\dots,z_k))
$
of the left hand side 
to $(\overline{u};(\overline z_0,\overline z_k,\overline z_{k-1},\dots,\overline z_1))$.
Here
$
\overline u(z) = u(\overline z)
$.
Therefore evaluation maps
$$
{\rm ev} = ({\rm ev}_0,\dots,{\rm ev}_{k+1}) :
\mathcal M_{k+1}((R(\Sigma),\omega_{R(\Sigma)});L;\beta)
\to L^{k+1}$$ satisfies
$
{\rm ev}_{k-i} \circ I = {\rm ev}_{i}.
$
(Our situation is somewhat similar to \cite{fooo:inv}.)
If we write the structure map of filtered $A_{\infty}$
algebra of $L \subset (R(\Sigma),\omega_{R(\Sigma)})$
(resp. of $L \subset (R(\Sigma),-\omega_{R(\Sigma)})$) 
by $\frak m_k^+$ (resp. $\frak m_k^-$) we have an equality
$$
\frak m_k^+(x_1,\dots,x_k) = 
\frak m_k^-(x_k,\dots,x_1).
$$
This implies that the filtered $A_{\infty}$ algebra $(CF(L),\{\frak m_k^+\})$
is an opposite algebra of  $(CF(L),\{\frak m_k^-\})$.
We can generalize this fact to the case we have several Lagrangian submanifolds
in a straight forward way.
It implies (2).
\par
(3) 
We remark that $\partial M_{12} = \Sigma_1\sqcup -\Sigma_2$.
Therefore $\partial M_{21} = -\Sigma_1\sqcup \Sigma_2$.
 \par
The moduli space of the connection on the space in 
 Figure 6.2, which we use to define $HF(\Phi_{M_{12}}(L_1,b_1),(L'_2,b'_2))$
 is isomorphic to the moduli space of the connection on the space in 
 Figure 6.2, which we use to define $HF(\Phi_{M_{21}}(L'_2,b'_2),(L_1,b_1))$.
 In fact such an isomorphism is obtained by the map which sends $t$ to $-t$.
(3) follows from this fact.
\par
(4) follows from Theorem \ref{mainthm} and $A_{\infty}$ Yoneda lemma 
\cite[Theorem 9.1]{fu4}  as follows.
The diagonal $\Delta \subset R(\Sigma_{\rm in}) \times R(\Sigma_{\rm out})$
represents the identity functor.
Therefore $A_{\infty}$ Yoneda lemma implies
$$
HF((\Delta,0),(R(M_{12},b_{12})) \otimes_{\Z_2} \Lambda^{\Z_2}
\cong
H(\mathscr{HOM}(\frak{ID},\Phi_{M_{12}})).
$$
On the other hand, Theorem \ref{mainthm} (3)  implies that the functor represented by
$(\Delta,0)$ is homotopy equivalent to the functor associated to  $(\Sigma_{\rm in}\times [0,1],b_{\Sigma_{\rm in}\times [0,1]})$ 
by Theorem \ref{functordef}.
Thus by Theorem \ref{mainthm} (2) we have
$$
HF((\Delta,0),(R(M_{12},b_{12}))
\cong
HF(\hat M,\hat{\mathcal E}),
$$
as required.
\end{proof}
Theorems \ref{thm62} and \ref{thm63} 
show that the theory of relative 
$SO(3)$-Floer homology we developed in this paper 
satisfies (at least certain large portion of) the 
axioms of topological field theory.
\par
We also remark that we can give an alternative proof of \cite[Theorem 2]{BD1},
(which is attributed to Floer) over $\Z_2$ coefficient from Theorem \ref{thm63} (4) in the case 
$\Sigma$ is a torus. In fact $R(T^2)$ consists of one point.
\begin{rem}
It seems that Wehrheim and Woodward proposed 
in \cite{WW2} and several other papers, 
to  use certain expected properties, which are similar to Theorems \ref{thm62}, \ref{thm63} etc..
as an axiom and use it together with various results in differential 
topology 
to show the existence and uniqueness of a version of the relative Floer theory.
It seems to the author that their idea is, in this way 
one may avoid studying gauge theory of 3 or 4 manifolds
directly and can prove the results expected from the gauge theory 
by a combinatorial 
method.
\par
An origin of such an idea is Floer's paper \cite{fl3},
where Floer tried to use his Dehn surgery triangle 
as a main axiom to characterize Floer homology 
of 3-manifolds. 
(See \cite{BD1} for certain discussion about  it.)
The author in \cite{fu1}, \cite{fu2}  proposed 
to use this Dehn surgery approach 
to prove Theorem \ref{mainthm} (2), without 
using analysis so much.
(This proposal by the author is not yet successful.)
\par
The distinguished example where this kinds of idea 
works very much successfully is  Heegard Floer theory by Ozvath-Szabo.
\par
In this paper and in this subsection we take opposite route.
Namely we define functors $\Phi_{M_{21}}$ 
directly by a geometric and analytic method and show its expected properties
directly without using combinatorial method.
\end{rem}

\subsection{Using similar moduli spaces}\label{other moduli}

In this paper, we use the moduli space introduced in \cite{fu3} or its 
variant to define and study Floer homology of 3-manifolds with boundary.
There are several other moduli spaces which are similar to but is slightly different 
from that.
\par\smallskip
The moduli space studied by Lipyanskiy in \cite{Li} is somewhat of similar flavor.
An element of his moduli space is also a combination of an ASD-connection and a
pseudo-holomorphic curve. 
The difference is in place of the metric $\chi(s)^2 g_{\Sigma} + ds^2 + dt^2$
Lipyanskiy used direct product metric $g_{\Sigma} + ds^2 + dt^2$ and its ASD-connection.
He instead introduce matching condition on the line where he switch from 
ASD-equation to pseudo-holomorphic curve equation.
Lipyanskiy obtained also removable singularity and compactness results, 
which are mixture of Uhlenbeck and Gromov compactness.
It seems likely that we can use Lipyanskiy's moduli space 
instead of one in \cite{fu3} to prove all the results of this paper,
though the author did not check the detail.
\par\smallskip
In \cite{fu1} and \cite{We} the moduli space of different flavor is 
proposed and established, respectively. 
Namely we study the moduli space of ASD-connections on $M\times \R$, 
for example, where $M$ has a boundary $\partial M$.
We need certain boundary condition on $\partial M \times \R
= \Sigma \times \R$.
Such a boundary condition must be an `infinite dimensional enhancement' 
of the Lagrangian submanifold of $R(\Sigma)$.
The one which the author proposed in  \cite{fu1} for this `infinite dimensional enhancement' 
is different from one used in \cite{We}.
With respect to this point, it seems that the one in \cite{We} 
(and not the one in \cite{fu1}) is the correct choice.
The moduli space studied in \cite{We} can be used for problems 
related to those discussed in this paper.
\par
However if we try to use it, in the same way as we are using the moduli space 
of \cite{fu3} in this paper, to prove the results of this paper, then there will be an issue.
Let us explain this issue briefly.
\par
 Let $L$ be an immersed Lagrangian submanifold of $R(\Sigma)$.
 We consider ASD equation on $M \times \R$. 
 (Here the metric we use near $\partial M \times \R$ is 
 $g_{\Sigma} + ds^2 + dt^2$ and does not degenerate.)
 We use $L$ to define a boundary condition as in \cite{We}.
 Requiring asymptotic boundary condition as $t \to \pm\infty$ 
 using $a_{\pm} \in R(M)\times_{R(\Sigma)} \tilde L$
 in the same way as Definition \ref{defn2626} (3), we can define a moduli space 
 $\mathcal M'((M,\mathcal E),L;a_-,a_+;E)$ and try to use it 
 (instead of  $\mathcal M((M,\mathcal E),L;a_-,a_+;E)$ in Definition \ref{defn2929})
 to define $HF((M,\mathcal E),L)$. 
 As is shown by Salamon-Wehrheim \cite{Sawe} this story works
as far as $L$ is embedded and monotone, since all the codimension one boundaries are 
of the form 
$$
\mathcal M'((M,\mathcal E),L;a_-,a;E_1)
\times \mathcal M'((M,\mathcal E),L;a,a_+;E_2),
$$
in that case.
In case $L$ is immersed or is not monotone, 
we need to combine this construction with the story of bounding cochains, 
since a disk bubble may produce codimension one boundary component.
For this purpose it seems that we need to glue an
element of  $\mathcal M'((M,\mathcal E),L;a_-,a_+;E)$ with a 
pseudo-holomorphic disk which bounds $L$. 
Namely we need to introduce $\mathcal M'_k((M,\mathcal E),L;a_-,a_+;E)$,
an analogue of $\mathcal M_k((M,\mathcal E),L;a_-,a_+;E)$ in Definition \ref{defn216},
(where $k$ is the number of boundary
marked points) and show that the fiber product
\begin{equation}\label{beforeglued}
\mathcal M'_{k_1}((M,\mathcal E),L;a_-,a_+;E_1)\,\, {}_{{\rm ev}_i}\times_{{\rm ev}_0} \mathcal M_{k_2+1}(L;E_2)
\end{equation}
appears at the boundary of 
$\mathcal M'_{k_1+k_2-1}((M,\mathcal E),L;a_-,a_+;E_1+E_2)$.
\par
The gluing analysis which we need to prove this statement is extremely difficult.
Let me elaborate this point more. Let us start with an element  $((\frak A,\vec z),(u,\vec z'))$ of 
(\ref{beforeglued}). 
The point $u(z'_0) \in R(\Sigma)$ is the gauge 
equivalence class of $\frak A\vert_{\Sigma \times \{z_i\}}$.  
The usual method for gluing is to regard $u$ as a family of flat connections 
and put it near $\Sigma \times \{z_i\}$ and try to glue it with $\frak A$.
The issue is we need to take a conformal diffeomorphism from 
a small neighborhood of $z_i$ to the domain of $u$ so that the family of 
flat connections $u$ is supported in this small neighborhood of $z_i$
after reparametrization.
When we scale it so that the diameter of the domain of this family of flat 
connections becomes something  like $1$, 
then the metric becomes 
$$
\frac{1}{\epsilon^2} g_{\Sigma} + ds^2 + dt^2.
$$
We observe that the family of flat connections $u$ actually is very far from being a solution 
of the ASD-equation with respect to the scaled metric. 
It would be close to the ASD-connection if the factor $\frac{1}{\epsilon^2}$ were very small. However this 
factor is actually very large.
By this reason the standard way of gluing does not seem to work.
\par
We remark that the situation is different in the case when we consider the same gluing 
problem for the moduli space 
$\mathcal M_k((M,\mathcal E),L;a_-,a_+;E)$ in Definition \ref{defn216}, 
which we use in this paper.
In fact, an element of this moduli space  is  $(\frak A,{
\frak z},{\frak w},\Omega,u,\vec z)$ where $u$ is a genuine pseudo-holomorphic curve to 
$R(\Sigma)$ 
in a neighborhood of the boundary.
Therefore, gluing an element of 
$\mathcal M_k((M,\mathcal E),L;a_-,a_+;E)$ with a pseudo-holomorphic disk 
is similar to the gluing between two pseudo-holomorphic disks.
In other words, if we have an appropriate Fredholm theory, we can work out the gluing analysis,  
in a way similar to those written in various literatures, eg. \cite[Sections 7.1.4 and 
A.1.4]{fooobook2}.
The situation seems to be similar in the case of Lipyanskiy's moduli space.
\par\medskip
On the other hand, if we restrict ourselves to the situation when 
$R(M)$ is embedded in $R(\Sigma)$ then it is likely that
the moduli space of \cite{We} can used to prove for example Corollary \ref{maincor}.
The  proof could be in 4 steps.
\par\smallskip
\noindent(Step  A)
Let $L_1$, $L_2$ be two embedded monotone Lagrangian submanifolds 
in $R(\Sigma)$. 
Let $a_{\pm} \in L_1 \cap L_2$. 
We consider $\Sigma \times [-1,1] \times \R$ and 
study ASD connections on it under the boundary condition 
which is induced by $L_1$ on $\Sigma \times \{-1\} \times \R$
and $L_2$ on $\Sigma \times \{1\} \times \R$ and
asymptotic boundary conditions given by $a_-$ and $a_+$.
Counting the solution in case its virtual dimension is $0$, 
we obtain a matrix coefficient 
$\langle \partial^G a_- ,a_+\rangle$.
It gives a boundary operator 
$\partial^G$ on the $\Z_2$ vector space with basis 
$L_1 \cap L_2$. The monotonicity implies that 
$\partial^G \circ \partial^G = 0$.
Let $HF(L_1,L_2)^{\rm gauge}$ be its homology group.
This step is already worked out in \cite{Sawe}, in a harder case when 
the bundle on $\Sigma$ is a trivial $SU(2)$ bundle.
\par\smallskip
\noindent(Step  B)
We can use pseudo-holomorphic strip and define 
Floer homology of Lagrangian intersection $HF(L_1,L_2)$.
(This step was done by Oh \cite{Oh}.)
\par\smallskip
\noindent(Step  C)
It may be possible to use adiabatic argument 
in a way similar to \cite{DS} to show
$HF(L_1,L_2) \cong HF(L_1,L_2)^{\rm gauge}$.
\par\smallskip
\noindent(Step  D)
We consider the case $\partial(M_1,\mathcal E_1) = -\partial(M_2,\mathcal E_2)
= (\Sigma,\mathcal E_{\Sigma})$
and suppose $R(M_1)$ and $R(M_2)$ are both embedded Lagrangian submanifolds 
of $R(\Sigma)$.
It seems that we can prove an isomorphism 
$HF(M,\mathcal E) \cong HF(R(M_1),R(M_2))^{\rm gauge}$ 
in a way similar to Section \ref{sec:glue} of this paper as follows.
We consider the domain $W$ as in Figure 5.1 and 
Condition \ref{conds54}.
We glue $-M_2 \times \R$ to $\partial_2W$ and $M_1 \times \R$ to $\partial_3W$
in the same way to obtain the 4 manifold $Y$. This time we consider the 
direct product metric $g_{\Sigma} + ds^2 + dt^2$ on $\Sigma \times W$.
We consider the set of gauge equivalence classes of connections, 
which are
ASD connections on $Y$ with respect to this metric and 
which satisfy the boundary conditions induced by $R(M_2)$ on $\partial_1W$ and 
by $R(M_1)$ on $\partial_4W$. 
We can use this moduli space 
to construct a chain map from 
$CF(M,\mathcal E)$  to  $CF(R(M_1),R(M_2))^{\rm gauge}$, which is 
congruent to the identity map.
Using the basic analytic results of \cite{Sawe} it is very likely 
that we can work out the detail of this proof.
\par\medskip
Let us consider the case when the bundle $\mathcal E_{\Sigma}$
is a trivial $SU(2)$ bundle, especially the case of handle body 
$M$. This is the case of  Atiyah-Floer conjecture in its original form. In this case $R(\Sigma)$ is not a symplectic manifold since it has a singularity.
Nevertheless as far as (Step D) concerns, there may not be  so much  big difference 
between this case and the case of nontrivial $SO(3)$ bundle on $\Sigma$.
Namely we may be able prove the isomorphism
\begin{equation}\label{welconj}
HF(M) \cong HF(R(M_1),R(M)_2))^{\rm gauge}
\end{equation}
as follows. Here $M_1$, $M_2$ are handle bodies whose boundaries are $\Sigma$.
$R(M_i)$ is the space of flat connections on $M_i$ regarded as a Lagrangian 
submanifold of $R(\Sigma)$.
The closed 3 manifold $M$, which we assume to be a homology 3 sphere, 
is obtained by gluing $M_1$ and $M_2$ along $\Sigma$.
The left hand side is a Floer homology of 3 manifold $M$ 
defined by Floer \cite{fl1} and the right hand side is defined by
Salamon-Wehrheim \cite{Sawe}. Here we consider 
trivial $SU(2)$ bundles.
We remark that (\ref{welconj}) is 
\cite[Conjecture 4.3]{We2}.
\par
We consider $\Sigma \times W$ where $W$ is as in Figure 5.1.
We glue $M_2 \times \R$ and $M_1 \times \R$ to $\Sigma \times \partial_2W$
and $\Sigma \times \partial_3W$ respectively and obtain $X$.
We put direct product metric $g_{\Sigma} + ds^2 + dt^2$ on $\Sigma \times W$
and extend it   to $Y$.
\par
For $a_-,a_+ \in R(M)  =  R(M_1) \cap R(M_2)$ 
we consider the moduli space $\mathcal M'(Y;a_-,a_+;E)$ of ASD-connections
on $Y$ of energy $E$ such that:
\begin{enumerate}
\item
On $\Sigma \times \partial _1W$ it satisfies the boundary condition of 
Wehrheim \cite{We} induced by $R(M_2) \subset R(\Sigma_1)$.
\item
On $\Sigma \times \partial _4W$ it satisfies the boundary condition of 
Wehrheim \cite{We} induced by $R(M_1) \subset R(\Sigma_1)$.
\item
On the end where ${\rm Re}z \to +\infty$, it will converge to $a_+$.
\item
On the end where ${\rm Re}z \to -\infty$, it will converge to $a_-$.
\end{enumerate}
Note at the end where ${\rm Im}z \to +\infty$ we use the fundamental class
$R(M_2)$ as the asymptotic boundary condition and 
on the end where ${\rm Im}z \to -\infty$ we use the fundamental class
$R(M_1)$ as the asymptotic boundary condition.
It seems that these conditions are automatically satisfied 
from our assumption that the energy is finite.
We however need some argument to prove it since $R(M_i)$ does not satisfy 
\cite[Condition L.3]{Sawe} when we consider $M_i$ as the 3 manifold. (Note the 3 manifold
which we denote by $M_i$ above, 
is denoted by $Y$
in \cite[page 748]{Sawe}.)
\par
Using the moduli space $\mathcal M'(Y;a_-,a_+;E)$ in case its virtual dimension is $0$ we 
may obtain 
matrix elements of the map
$$
\Phi : CF(M) \to CF(R(M_1),R(M_2))
$$
between chain complexes defining left and right hand sides of 
(\ref{welconj}).
To show that this map $\Phi$ is a chain map we consider the moduli space,
$\mathcal M'(Y;a_-,a_+;E)$, in case its virtual dimension is $1$.
We consider its boundary.
The possibilities are
\begin{enumerate}
\item[(I)]
ASD connections escape to the direction ${\rm Im}z \to +\infty$.
\item[(II)]
ASD connections escape to the direction ${\rm Im}z \to -\infty$.
\item[(III)]
ASD connections escape to the direction ${\rm Re}z \to -\infty$.
\item[(IV)]
ASD connections escape to the direction ${\rm Re}z \to +\infty$.
\end{enumerate} 
Using monotonicity in the same way as the proof of Proposition \ref{prop310}
it seems that we can show that there is no contribution from the end of types 
(I) and (II).
(We need some argument at this point also since $R(M_i)$ contains reducible 
connections.)
(The monotonicity also implies that there is no contribution from the bubble 
at $\Sigma \times \partial_1W$ and $\Sigma \times \partial_4W$.)
\par
The ends of type (III) (resp. of type (IV)) gives the boundary operator 
of the left hand side (resp. right hand side) of (\ref{welconj}) 
composed with $\Phi$.
Therefore we may be able to prove that $\Phi$ is a chain map in this way.
\par
The set of zero energy solutions  $\mathcal M'(Y;a_-,a_+;0)$ 
is the empty set if $a_- \ne a_+$ and consists of a single point $a$
if $a = a_- = a_+$.
Therefore $\Phi$ is congruent to the identity map modulo 
$\Lambda_+$.
\par
We thus sketched a possible way to  prove (\ref{welconj}).
\par\medskip
Note (Step A) in the case of handle bodies is established by Salamon-Wehrheim 
\cite{Sawe}. 
\par
Therefore, according to the author's opinion, (Step B) is the most difficult part  which remains 
to be worked out to prove Atiyah-Floer conjecture
based on the argument summarized above. Here (Step B) means: finding
correct notion of Lagrangian intersection Floer homology 
for a pair of Lagrangian submanifolds in the singular space $R(\Sigma)$, 
by using pseudo-holomorphic map to $R(\Sigma)$.
(Note this step had been done by Oh in case $R(\Sigma)$ is smooth.)
\par
We also remark that in the proof of removable singularity theorem and compactness 
theorem in \cite{fu3}, the author used the fact 
that all the elements of $R(\Sigma)$
are irreducible. So to use the moduli space of \cite{fu3} to study 
Atiyah-Floer conjecture in the original case of handle body, 
we first need to improve this point.

\subsection{Cobordism method and degeneration method}\label{other moduli}
In \cite{DS} the key idea of the proof of their main theorem, which is equivalent to 
Corollary \ref{maincor} in the case $M_1 =M_2 = \Sigma \times [0,1]$, 
is studying an adiabatic limit, where $\Sigma$ collapses to a point.
The method of this paper does not study this degeneration directly but
replace it by a cobordism argument.\footnote{However 
the proof of removable singularity and compactness theorem in \cite{fu3}
uses a similar estimate as \cite{DS}.}
In other words, we bypass some of the difficult analytic issue by using 
cobordism argument.
We remark that the idea proposed in \cite[Section 4]{We2} to prove Atiyah-Floer conjecture (in its original form) 
is based on degeneration analysis, which is harder than \cite{DS}.
\par
A similar point appears also when we consider the related problem of the study of 
Wehrheim-Woodward functoriality (\cite{WW}).
At the most important point of their work, Wehrheim-Woodward
used strip shrinking, which is of a similar flavor as taking 
adiabatic limit. Then it appeared interesting and deep problem 
of figure eight bubble. In \cite{WW}, 
figure eight bubble is excluded by using monotonicity.
The main idea of Lekili-Lipyanskiy in \cite{LL} is 
to replace difficult analytic problem of strip shrinking
by a cobordism argument.
\par
If we go beyond the monotone case, the effect of figure eight bubble
becomes nonzero while studying strip shrinking.
In the recent works by Bottman \cite{bott},
Bottman-Wehrheim \cite{bww}, several important 
facts are discovered about figure eight bubble.
There are certain heuristic discussions on the moduli space of 
figure eight bubbles in \cite{bww}. 
\par
The author conjectured that  the moduli space of 
figure eight bubble gives a bounding cochain 
in the sense of \cite{fooobook}
of the filtered $A_{\infty}$ algebra 
 associated by \cite{AJ} to the immersed Lagrangian submanifold obtained by 
the Lagrangian correspondence.
\par
If the author's conjecture is correct, 
then it is also very likely that this bounding cochain 
(up to gauge equivalence) coincides with one we obtained in 
Theorem \ref{WWfunc}.
We emphasize that as far as analytic results concern we need only  
well-established or standard results to prove Theorem \ref{WWfunc}.
Especially we do not need to study  strip shrinking or figure eight bubble.
In other words, we replace difficult analysis by an algebraic lemma\footnote{Studying figure eight bubble is certainly interesting for its own sake
and can be applied in various places.} (Proposition \ref{thm35}).
This is similar to Lekili-Lipyanskiy's argument 
which replaces strip shrinking by a cobordism argument.
Actually the starting point of author's research, which leads to this paper, 
was to try to generalize  Lekili-Lipyanskiy's method and combine it 
with the obstruction-deformation theory of \cite{fooobook} and 
immersed Lagrangian Floer theory of \cite{AJ}.
\par\medskip
\noindent
{\bf Acknowledgement.} The research of the author is supported partially by JSPS Grant-in-Aid for Scientific Research
No. 23224002 and NSF Grant No. 1406423.
\bibliographystyle{amsalpha}

\end{document}